\documentclass[leqno]{article} 

\usepackage[frenchb,english]{babel}
\usepackage[T1]{fontenc}
\usepackage[utf8]{inputenc}

\usepackage{amsmath}
\usepackage{amssymb}
\usepackage{amsfonts}
\usepackage{enumerate}
\usepackage{vmargin}
\usepackage[all]{xy}
\usepackage{mathrsfs}
\usepackage{mathtools}
\usepackage{lmodern}
\usepackage[pagebackref=true]{hyperref}
\usepackage{ulem}
\usepackage[all]{xy}
\usepackage{comment}
\setmarginsrb{3cm}{3cm}{3.5cm}{3cm}{0cm}{0cm}{1.5cm}{3cm}

\newcommand{\C}{\ensuremath{\mathbb{C}}}
\newcommand{\E}{\ensuremath{\mathcal{E}}}

\newcommand{\Gc}{\ensuremath{\mathcal{G}}}

\newcommand{\N}{\ensuremath{\mathbb{N}}}

\newcommand{\R}{\ensuremath{\mathbb{R}}}

\newcommand{\Z}{\ensuremath{\mathbb{Z}}}

\renewcommand{\epsilon}{\varepsilon}

\renewcommand{\leq}{\ensuremath{\leqslant}}
\renewcommand{\geq}{\ensuremath{\geqslant}}
\newcommand{\qed}{\hfill \vrule height6pt  width6pt depth0pt}

\newcommand{\ot}{\otimes}
\newcommand{\epsi}{\varepsilon}
\newcommand{\ovl}{\overline}

\newcommand{\co}{\colon}
\renewcommand{\d}{\mathop{}\mathopen{}\mathrm{d}} 

\newcommand{\supp}{\mathrm{supp}}

\newcommand{\dist}{\mathrm{dist}}
\newcommand{\Id}{\mathrm{Id}}

\newcommand{\HI}{H^\infty}
\DeclareMathOperator{\Str}{Str}
\newcommand{\spr}[2]{\langle #1, #2 \rangle}
\DeclareMathOperator{\Ker}{Ker} 
\DeclareMathOperator{\Ran}{Ran} 

\DeclareMathOperator{\esssup}{ess-sup}
\DeclareMathOperator{\essinf}{ess-inf}
\newtheorem{thm}{Theorem}[section]
\newtheorem{defi}[thm]{Definition}

\newtheorem{prop}[thm]{Proposition}

\newtheorem{cor}[thm]{Corollary}
\newtheorem{lemma}[thm]{Lemma}

\newtheorem{remark}[thm]{Remark}

\newenvironment{proof}[1][]{\noindent {\it Proof #1} : }{\hbox{~}\qed
\smallskip
}

\numberwithin{equation}{section}

\begin{document}
\selectlanguage{english}
\title{\bfseries{$\HI$ calculus for submarkovian semigroups on weighted $L^2$ spaces}}
\date{September 2019}
\author{\bfseries{Komla Domelevo, Christoph Kriegler and Stefanie Petermichl}}
\maketitle

\begin{abstract}
Let $(T_t)_{t \geq 0}$ be a markovian (resp. submarkovian) semigroup on some $\sigma$-finite measure space $(\Omega,\mu)$.
We prove that its negative generator $A$ has a bounded $\HI(\Sigma_\theta)$ calculus on the weighted space $L^2(\Omega,wd\mu)$ as long as the weight $w : \Omega \to (0,\infty)$ has finite characteristic defined by $Q^A_2(w) = \sup_{t > 0} \left\| T_t(w) T_t \left(w^{-1} \right) \right\|_{L^\infty(\Omega)}$ (resp. by a variant for submarkovian semigroups).
Some additional technical conditions on the semigroup have to be imposed and their validity in examples is discussed.
Any angle $\theta > \frac{\pi}{2}$ is admissible in the above $\HI$ calculus, and for some semigroups also certain $\theta = \theta_w < \frac{\pi}{2}$ depending on the size of $Q^A_2(w)$.
The norm of the $\HI(\Sigma_\theta)$ calculus is linear in the $Q^A_2$ characteristic for $\theta > \frac{\pi}{2}$.
We also discuss negative results on angles $\theta < \frac{\pi}{2}$.
Namely we show that there is a markovian semigroup on a probability space and a $Q^A_2$ weight $w$ without H\"ormander functional calculus on $L^2(\Omega,w d\mu)$.
\end{abstract}


\makeatletter
 \renewcommand{\@makefntext}[1]{#1}
 \makeatother

\tableofcontents

\section{Introduction}
\label{sec-introduction}

It is well established by now that the $\HI$ functional calculus of a sectorial operator
has important applications in the spectral theory of partial differential operators
and the theory of evolution equations, e.g., in determining the domain of fractional
powers of a partial differential operator in the solution of Kato's problem
(e.g. \cite{AHLLMT,AuTc,DDHPV,Gig,Yag}), in connection with maximal regularity of parabolic
evolution equations (e.g. \cite{HiPr,LaLaMe,LaMe,LeM1,PrSi,Weis}) and certain estimates in
control theory (\cite{HaLe,HaOu,LeM2}).
Today it is known that many systems
of elliptic partial differential operators, Schr\"odinger operators and related important examples of semigroup generators do have
an $\HI$ calculus (\cite{BlKu,CDMcIY,DuOS,DSY,GCMMST,GY,HvNP,KW04,vNP}).
Also from an abstract point of view, a lot of effort has been achieved to establish, characterise and transfer $\HI$ calculus (\cite{AFLM,Haase,KaWe1,KaWe2,KaKuWe}).

We let $\theta \in (0,\pi)$ be an angle and define $\Sigma_\theta = \{ \lambda \in \C \backslash \{ 0 \} :\: |\arg \lambda| < \theta \}$ to be the sector around the positive half-axis with half opening angle equal to $\theta$.
The $\HI$ class is then $\HI(\Sigma_\theta) = \{ f : \: \Sigma_\theta \to \C : \: f\text{ is holomorphic and bounded} \}$, which is a Banach algebra when equipped with pointwise multiplication and norm $\|m\|_{\infty,\theta} = \sup_{\lambda \in \Sigma_\theta} |m(\lambda)|$.
Let now $-A$ be the generator of a $c_0$-semigroup $(T_t)_{t \geq 0}$ on some Banach space $X$.
Suppose that $A$ is sectorial, i.e. the spectrum is contained in some $\overline{\Sigma_{\theta'}}$ and the resolvents are appropriately norm controlled, and suppose for simplicity that $A$ has dense range (see Subsection \ref{subsec-HI} for details).
Furthermore, let $m \in \HI(\Sigma_\theta)$ decay polynomially at $0$ and at $\infty$.
Then one defines the $\HI$ functional calculus $m(A)$, a bounded linear operator on $X$, by means of the Cauchy integral formula over resolvents
\[ m(A) = \frac{1}{2\pi i} \int_{\partial\Sigma_{\theta''}} m(\lambda) R(\lambda,A) d\lambda,\]
with angles $\theta' < \theta'' < \theta$.
Then $\HI$ calculus is the question whether this ad hoc formula can be reasonably extended to all $m \in \HI(\Sigma_\theta)$ and whether one obtains the estimate
\begin{equation}
\label{equ-HI-calculus}
\|m(A)\|_{B(X)} \leq C \|m\|_{\infty,\theta} \quad ( m \in \HI(\Sigma_\theta) ).
\end{equation}
This is a difficult task and its solution, for more or for less concrete operators $A$, requires several fundamental tools from harmonic analysis such as square functions (see e.g. \cite[Section 6]{CDMcIY}, \cite{KaWe2}), bounded imaginary powers of $A$ \cite[Section 5]{CDMcIY}, bilinear embeddings \cite[Section 4]{CDMcIY} and transference principles \cite{AFLM,CoWe,Fen,HiPr}.
Note that a positive answer of \eqref{equ-HI-calculus} depends in general on $\theta$ and a smaller angle yields a more restrictive condition, since $\HI(\Sigma_\theta) \subseteq \HI(\Sigma_\sigma)$ if $\sigma \leq \theta$ by uniqueness of analytic continuation.

Let us give a brief overview of important operator theoretic results when an $\HI$ calculus is known.
Let $(\Omega,\mu)$ be a $\sigma$-finite measure space.
First suppose that the semigroup $(T_t)_{t \geq 0}$ is markovian (see Subsection \ref{subsec-semigroups} for the definition of this classical notion), so contractive on all $L^p(\Omega)$, self-adjoint on $L^2(\Omega)$, lattice positive and $T_t(1) = 1$.
The first universal multiplier
theorem was proved by E. M. Stein, who showed that if $m$ is of Laplace transform type, then $m(A)$ is bounded on $L^p(\Omega)$ for $1 < p < \infty$ \cite[Corollary 3, p. 121]{Ste70}.
This result was later extended to submarkovian semigroups (see Subsection \ref{subsec-semigroups} for the definition) and for $m$ belonging to $\HI(\Sigma_\theta)$ by Cowling \cite[Theorem 1]{Cow} and Meda \cite[Theorem 3]{Med}.
The angle of the functional calculus depends on $p$ and by complex interpolation with the self-adjoint calculus on $L^2(\Omega)$, one obtains $\theta > \pi \left| \frac1p - \frac12 \right|$.
Later on it was observed by Duong \cite{Duo}, (see also \cite{HiPr} for $\theta > \frac{\pi}{2}$ and \cite[Corollary 5.2]{KaWe1} for $\theta < \frac{\pi}{2}$) that semigroups acting on a single $L^p(\Omega)$ space and consisting of positive and contractive operators, or even only regular contractive operators \cite{CoWe,Fen,Ste70} suffices to obtain an $\HI(\Sigma_\theta)$ calculus.
A recent extension of \cite{HiPr} and \cite[Corollary 10.15]{KW04} is \cite[Theorem 4]{Xu2015} where the setting is a vector valued semigroup of the form $T_t = T_t^{(0)} \ot \Id_Y$ acting on the Bochner space $X = L^p(\Omega;Y)$, where $Y$ is an intermediate UMD space and where $T_t^{(0)}$ is an analytic semigroup consisting of regular contractive operators acting on a single $L^p(\Omega)$ space.
Here the novelty compared to \cite[Corollary 10.15]{KW04} is an angle of $\HI(\Sigma_\theta)$ calculus $\theta < \frac{\pi}{2}$.
Concerning the optimality of the $\HI$ calculus angle, the recent breakthrough result \cite{CaDr} yields $\theta > \theta_p = \arctan\left( \frac{|p-2|}{2\sqrt{p-1}} \right)$ on $X = L^p(\Omega)$, $1 < p < \infty$, within the class of submarkovian semigroups (or even the class of self-adjoint semigroups which are contractive on the $L^p(\Omega)$ scale).
Here the angle $\theta_p$ is already optimal in the simple example of the Ornstein-Uhlenbeck semigroup acting on $L^p(\R^d,\mu)$ where $d\mu(x) = \left(2 \pi\right)^{-\frac{d}{2}} \exp\left( - \frac{|x|^2}{2} \right) dx$ is Gaussian measure and $A = - \Delta + x \cdot \nabla$ \cite{GCMMST}.

In the present article we consider markovian and submarkovian semigroups, and add a weight $w$ to the picture, so that $\|f\|_X = \|f\|_{L^2(\Omega,wd\mu)} = \left(\int_\Omega |f(x)|^2 w(x) d\mu(x) \right)^{\frac12}$.
Weighted estimates for spectral multipliers have been recently studied by \cite[Theorem 3.1 and Theorem 3.2]{DSY} and \cite[Theorem 4.1 and Theorem 4.2]{GY}.
In the latter works, the space $\Omega$ is supposed to be of homogeneous type and the semigroup $(T_t)_{t \geq 0}$ is self-adjoint and has an integral kernel satisfying Gaussian bounds (see \eqref{equ-GE} below).
The multiplier function $m$ is allowed to belong to the so-called H\"ormander-Mihlin class which consists of certain functions defined on $(0,\infty)$ which are differentiable up to a prescribed order.
Note that the H\"ormander-Mihlin class contains $\HI(\Sigma_\theta)$ for any $\theta \in (0,\pi)$, so that \cite{DSY,GY} yield an $\HI$ calculus to any angle on weighted $L^p$ spaces.
The weights that are allowed here belong to a certain (spatially defined) Muckenhoupt class, see also Remark \ref{rem-DSY-GY} for a comparison with our results.

In this work, we settle the case of markovian and submarkovian semigroups without any dimension assumption on $\Omega$ nor integral kernel estimates of $(T_t)_{t \geq 0}$. Our underlying Banach space will always be $X = L^2(\Omega,wd\mu)$.
The natural condition for the weight $w$ is the semigroup characteristic
\[ Q^A_2(w) = \sup_{t > 0} \esssup_{x \in \Omega} T_tw(x) T_t\left(w^{-1}\right)(x) < \infty . \]
Then our first main result reads as follows.
\begin{thm}[see Corollary \ref{cor-bilinear}]
\label{thm-1-intro}
Let $(\Omega,\mu)$ be a $\sigma$-finite measure space and $(T_t)_{t \geq 0}$ be a markovian semigroup on $(\Omega,\mu)$.
Let $w$ be a weight on $\Omega$ such that $Q^A_2(w) < \infty$.
Then under some technical condition on the semigroup (e.g. $\mu(\Omega) < \infty$ suffices), the negative generator $A$ of $(T_t)_{t \geq 0}$ is $\frac{\pi}{2}$-sectorial on $L^2(\Omega,w d\mu)$ and \eqref{equ-HI-calculus} holds, more precisely,
\begin{equation}
\label{equ-1-thm-1-intro}
\|m(A)\|_{L^2(\Omega,wd\mu) \to L^2(\Omega,wd\mu)} \leq C_\theta Q^A_2(w) \left(|m(0)| + \|m\|_{\infty,\theta} \right)
\end{equation}
for any $\theta > \frac{\pi}{2}$.
\end{thm}
We are also able to push the angle $\theta$ to be equal to $\frac{\pi}{2}$ and show in Corollary \ref{cor-bilinear} an $\HI(\Sigma_{\frac{\pi}{2}};J)$ calculus on $L^2(\Omega,wd\mu)$, where $\HI(\Sigma_{\frac{\pi}{2}};J)$ consists of holomorphic functions bounded on $\C_+$ admitting a boundary function on $i\R$ which belongs to a certain Besov class (see Definition \ref{defi-HI-Besov}).
For the additional term $|m(0)|$ compared to \eqref{equ-HI-calculus} which disappears when $A$ is injective, we refer to the classical Remark \ref{rem-bilinear-to-calculus-weight}.
We refer to Theorem \ref{thm-bilinear} for the precise meaning of the technical conditions.
They are satisfied in the following particular cases: $\mu(\Omega) < \infty$, or $\Omega$ being a space of homogeneous type and $(T_t)_{t \geq 0}$ having an integral kernel satisying Gaussian estimates (see Remarks \ref{rem-automatic-infty-domain} and \ref{rem-local-diffusion}).
Also if $\Omega$ is a locally compact separable metric measure space, $(T_t)_{t \geq 0}$ is a Feller semigroup and the weight $w$ is continuous, then one of the needed technical conditions is satisfied (see Remark \ref{rem-automatic-pointwise-convergence} 2.)
The method of proof of Theorem \ref{thm-1-intro} is establishing a bilinear estimate
\begin{equation}
\label{equ-bilinear-intro}
\int_0^\infty \left|\langle AT_tf, T_tg \rangle \right| dt \leq C \|f\|_{L^2(\Omega,wd\mu)} \|g\|_{L^2(\Omega,w^{-1}d\mu)}
\end{equation}
(see Theorem \ref{thm-bilinear}) which is well-known to yield $\HI$ calculus (see \cite[Theorem 4.4]{CDMcIY} or Proposition \ref{prop-bilinear-to-calculus} and Proposition \ref{prop-bilinear-to-calculus-weight}).
For the proof of \eqref{equ-bilinear-intro} in turn, we use a Bellman function from \cite{DoPe}, see Lemma \ref{L: existence and properties of the Bellman function}, in the spirit of \cite{CaDr}, capturing the weight variables.
That is, we will define a functional
\[ \mathcal{E}(t) = \int_{\Omega} B(T_tf,T_tg,T_t(w^{-1}),T_tw) dt \geq 0\]
(see \eqref{equ-E-functional})
with the two properties that 
\begin{align}
\mathcal{E}(0) & \lesssim \|f\|_{L^2(\Omega,wd\mu)}^2 + \|g\|_{L^2(\Omega,w^{-1}d\mu)}^2 \nonumber \\
\intertext{and}
- \mathcal{E}'(t) & \geq \frac{c}{Q}\left| \langle AT_t,T_tg \rangle \right| \label{equ-intro-E-functional}
\end{align}
(see \eqref{equ-quantitative-monotonicity-version-2}, \eqref{equ-calculation-1}, \eqref{equ-calculation-2} and Proposition \ref{prop-markovian-estimate}),
reflecting the properties of the Bellman function in Lemma \ref{L: existence and properties of the Bellman function}.
Then an integration of \eqref{equ-intro-E-functional} over $t \in (0,\infty)$ will yield \eqref{equ-bilinear-intro}.
Note that despite the directness of this approach, the proof is rather involved, and it is notably the differentiability of $\mathcal{E}(t)$ which imposes additional hypotheses in Theorem \ref{thm-1-intro} on the markovian semigroup.

Our result also extends to submarkovian semigroups.
Note that then the characteristic has to be modified into the a priori larger expression
\[ \tilde{Q}^A_2(w) = \sup_{t > 0} \esssup_{x \in \Omega'} S_t(w')(x) S_t\left(w'^{-1}\right)(x) ,\]
where $\Omega' = \Omega \cup \{ \infty \}$ includes a supplementary cemetery point,
\[ S_t(f')(x) = \begin{cases} T_t(f'|_{\Omega})(x) + f'(\infty)(1-T_t(1))(x) & : \: x \in \Omega \\ f'(\infty) & : \: x = \infty \end{cases} \]
is the conservative semigroup extension of $T_t$, i.e. $S_t(1) = 1$, and $w'(x) = w(x) 1_{x \in \Omega} + 1_{x = \infty}$.
Then our second main results is the following.
\begin{thm}[see Corollary \ref{cor-bilinear-submarkovian}] 
\label{thm-2-intro}
Let $(\Omega,\mu)$ be a $\sigma$-finite measure space and $(T_t)_{t \geq 0}$ be a submarkovian semigroup on $(\Omega,\mu)$.
Let $w$ be a weight on $\Omega$ such that $\tilde{Q}^A_2(w) < \infty$.
Then under some technical condition on the semigroup (e.g. $\mu(\Omega) < \infty$ suffices), the negative generator $A$ of $(T_t)_{t \geq 0}$ is $\frac{\pi}{2}$-sectorial on $L^2(\Omega,w d\mu)$ and \eqref{equ-HI-calculus} holds, more precisely,
\[ \|m(A)\|_{L^2(\Omega,wd\mu) \to L^2(\Omega,wd\mu)} \leq C_\theta \tilde{Q}^A_2(w) \left(|m(0)| + \|m\|_{\infty,\theta} \right) \]
for any $\theta > \frac{\pi}{2}$.
\end{thm}
The angle of $\HI(\Sigma_\theta)$ calculus in Theorems \ref{thm-1-intro} and \ref{thm-2-intro} above is $\theta > \frac{\pi}{2}$.
In some cases, this angle can be reduced to $\theta < \frac{\pi}{2}$.

\begin{prop}[see Proposition \ref{prop-maximal-regularity}, Remark \ref{rem-power-characteristic} and Proposition \ref{prop-submarkovian-maximal-regularity}]
\label{prop-3-intro}
Let $(\Omega,\mu)$ be a $\sigma$-finite measure space and $(T_t)_{t \geq 0}$ be a markovian (resp. submarkovian) semigroup on $(\Omega,\mu)$.
Suppose that the weight $w$ on $\Omega$ satisfies $Q^A_2\left(w^\delta\right) < \infty$ (resp. $\tilde{Q}^A_2\left(w^\delta\right) < \infty$) for some $\delta > 1$.
Then under the same technical conditions as in Theorem \ref{thm-1-intro} (resp. Theorem \ref{thm-2-intro}), $A$ has an $\HI(\Sigma_\theta)$ calculus on $L^2(\Omega,wd\mu)$ for some $\theta < \frac{\pi}{2}$ and consequently, $A$ has maximal regularity on $L^2(\Omega,wd\mu)$.
\end{prop}

Note that if $\Omega = \R^n$ 
and if the class $Q^A_2$ equals the class $Q^{class}_2$, where
\[ Q^{class}_2(w) = \sup_{B \text{ ball in }\R^n} \left( \frac{1}{|B|} \int_B w \right) \left( \frac{1}{|B|} \int_B w^{-1} \right),\]
then any weight $w$ with $Q^A_2(w) < \infty$ automatically satisfies $Q^A_2\left(w^\delta\right) < \infty$ for some $\delta > 1$.
Consequently, Proposition \ref{prop-3-intro} above applies then for any $Q^A_2$ weight.
Note that $\delta > 1$ depends then on $w$, and then $\theta < \frac{\pi}{2}$ also depends on $w$.
Moreover, the dependence of the $\HI(\Sigma_\theta)$ functional calculus norm on $Q^A_2(w)$ is a priori not linear any more for $\theta < \frac{\pi}{2}$ as it was in \eqref{equ-1-thm-1-intro}
Then the question arises whether one can lower the angle in Theorem \ref{thm-1-intro} universally within the class of all markovian semigroups and all $Q^A_2(w)$ weights.
As a partial negative result, we obtain the following.

\begin{thm}[see Theorem \ref{thm-no-Hormander}]
\label{thm-4-intro}
There is a markovian semigroup with negative generator $A$ on a probability space and a $Q^A_2$ weight $w$ such that for no $s > 0$, $A$ has a H\"ormander $\mathcal{H}^s$ calculus on $L^2(\Omega,w d\mu)$.
In other words, for no $s > 0$ there exists a constant $C > 0$ such that
\[ \|m(A)\|_{L^2(\Omega,w d\mu) \to L^2(\Omega,w d\mu)} \leq C \left( |m(0)| + \|m\|_{\mathcal{H}^s} \right) \]
holds for all H\"ormander multipliers $m \in \mathcal{H}^s$.
\end{thm}
For a definition of the H\"ormander $\mathcal{H}^s$ class which consists of functions defined on $\R_+$, we refer to \eqref{equ-def-Hormander}.
The counter-example that exhibits Theorem \ref{thm-4-intro} is based on a markovian semigroup defined on a space $\Omega$ consisting only of two points.

Let us close the introduction with an overview of the rest of the article.
In the preliminary Section \ref{sec-preliminaries}, we introduce the objects of study for the rest of the paper.
In Subsection \ref{subsec-semigroups}, standard notions of markovian and submarkovian semigroups are discussed.
Then Subsection \ref{subsec-HI} contains the necessary material on $\HI$ calculus.
Here, the space $\HI(\Sigma_\theta;J)$ consisting of holomorphic functions with boundary term belonging to a Besov class (see Definition \ref{defi-HI-Besov} and Lemma \ref{lem-properties-Besov}) might be less standard.
Lemma \ref{lem-HI-to-Besov} characterizing the $\HI(\Sigma_\theta;J)$ calculus in terms of the growth of the constant appearing in the $\HI(\Sigma_\sigma)$ calculus \eqref{equ-HI-calculus} when $\sigma \to \theta +$, is possibly new.
Also that the bilinear estimate \eqref{equ-bilinear-intro} implies a $\HI(\Sigma_{\frac{\pi}{2}};J)$ calculus with sharp angle $\frac{\pi}{2}$ (see Propositions \ref{prop-bilinear-to-calculus} and \ref{prop-bilinear-to-calculus-weight}) might be less known.
Then in Subsection \ref{subsec-weights}, we define weights in our setting and show a cut-off property of the characteristic $Q^A_2$.
In Section \ref{sec-Bellman} we introduce the Bellman function which is a crucial ingredient of the proof of Theorems \ref{thm-1-intro} and \ref{thm-2-intro}.
Some technical properties needed in the subsequent two sections are proved.
Then in Section \ref{sec-bilinear}, we state and prove the $\HI$ calculus on weighted $L^2$ space for markovian semigroups (Theorem \ref{thm-1-intro}), and discuss the technical conditions we have to impose.
Parallelly to that, in Section \ref{sec-submarkovian}, we state and prove the companion result for submarkovian semigroups (Theorem \ref{thm-2-intro}).
Finally, in Section \ref{sec-negative-results}, we state and prove Theorem \ref{thm-4-intro}.

\section{Preliminaries}
\label{sec-preliminaries}

\subsection{Semigroups}
\label{subsec-semigroups}

In this article, $(\Omega,\mu)$ always denotes a $\sigma$-finite measure space.

\begin{defi}
\label{defi-submarkovian}
Let $(T_t)_{t \geq 0}$ be a $c_0$-semigroup on $L^2(\Omega)$.
\begin{enumerate}
\item Then $(T_t)_{t \geq 0}$ is called a submarkovian semigroup, if
\begin{enumerate}
\item $T_t$ extends boundedly to an operator on $L^p(\Omega)$ for all $p \in [1,\infty]$ and we have $\|T_t\|_{p \to p} \leq 1$ for any $t \geq 0$ and $p \in [1,\infty]$.
\item $T_t$ is self-adjoint for any $t \geq 0$.
\item $T_tf \geq 0$ for any $t \geq 0$ whenever $f \in \bigcup_{p \in [1,\infty]} L^p(\Omega)$ with $f \geq 0$.
\end{enumerate}
\item $(T_t)_{t \geq 0}$ is called a markovian semigroup, if $(T_t)_{t \geq 0}$ is submarkovian and in addition, $T_t(1) = 1$ for any $t \geq 0$.
\end{enumerate}
In the same way, we call a single operator $T$ submarkovian, if $T$ is a contraction on $L^p(\Omega)$ for all $p \in [1,\infty]$, self-adjoint on $L^2(\Omega)$ and $Tf \geq 0$ for $f \geq 0$.
We also call a single operator $T$ markovian, if in addition $T(1) = 1$.
\end{defi}

In what follows, if $(T_t)_{t \geq 0}$ is a semigroup, we denote $A$ its (negative) generator, i.e. $T_t = \exp(-tA)$.
In fact, we shall always call $A$ the negative generator and omit the term ``negative'' sometimes.
Note that if $(T_t)_{t \geq 0}$ is a submarkovian semigroup, then for $1 < p < \infty$, $(T_t)_{t \geq 0}$ is analytic on $L^p(\Omega)$, and thus, for $f \in L^p(\Omega)$, we have that $T_tf$ belongs to $D(A_p)$, the domain of the generator on $L^p(\Omega)$.
The semigroup $(T_t)_{t \geq 0}$ is typically not strongly continuous on $L^\infty(\Omega)$.
However, $A_\infty(f) = w^*-\lim_{h \to 0+} \frac1h (f - T_hf)$ has a $w^*$-dense domain $D(A_\infty) = \{ f \in L^\infty(\Omega) : \:  w^*-\lim_{h \to 0+} \frac1h (f - T_hf) \text{ exists} \}$.
A classical $w^*$-approximation of $f \in L^\infty(\Omega)$ by elements of $D(A_\infty)$ is $\frac1h \int_0^h T_tf dt \quad (h \to 0+)$.
We have the following Cauchy Schwarz type lemma for positive $L^\infty(\Omega)$ contractions, in particular for operators coming from a (sub)markovian semigroup.

\begin{lemma}
\label{lem-semigroup-CSU}
Let $T$ be a contraction on $L^\infty(\Omega)$ such that $Tf \geq 0$ for any $f \geq 0$.
Then for $f,g \in L^\infty(\Omega)$,
\[ |T(fg)(x)|^2 \leq T(|f|^2)(x)T(|g|^2)(x) \quad (\text{a.e. }x \in\Omega) . \]
\end{lemma}

\begin{proof}
For a measurable set $B \subset \Omega$ of finite positive measure, we put
\[ \phi_B : L^\infty(\Omega) \to \C , \: g \mapsto \frac{1}{\mu(B)} \int_B Tg d\mu . \]
Then $\phi_B$ is a well-defined positive linear functional.
By the Cauchy-Schwarz inequality for positive linear functionals on $C^*$-algebras, we have $|\phi_B(fg)|^2 \leq \phi_B(|f|^2) \phi_B(|g|^2)$.
Assume on the contrary that $|T(fg)(x)|^2 > T(|f|^2)(x)T(|g|^2)(x)$ on a set $B$ of positive measure.
Then passing to a smaller set of positive measure, we have $|T(fg)(x)|^2 \geq (1 + \epsilon) T(|f|^2)(x) T(|g|^2)(x)$ for some $\epsilon > 0$.
We can assume that $B$ has finite (positive) measure.
Consider first the case that one of the two functions $T(|f|^2),T(|g|^2)$ is equal to $0$ on $B$, say $T(|f|^2)$.
We have $0 \leq |\phi_B(|fg|)|^2 \leq \phi_B(|f|^2) \phi_B(|g|^2) = 0$, so $\phi_B(|fg|) = 0$.
Since $T(|fg|) \geq 0$, it follows $T(|fg|) = 0$ on $B$, and by positivity of $T$, also that $|T(fg)(x)| \leq T(|fg|)(x) = 0$ for $x \in B$.

Assume now that both functions $T(|f|^2)$ and $T(|g|^2)$ are strictly positive on $B$.
For $N \in \N$ a given number whose value we shall specify later, we decompose $[0,2\pi] = \bigcup_{k = 0}^{N-1} [2 \pi \frac{k}{N}, 2 \pi \frac{k+1}{N}]$.
Then there exists $k_0 \in \{1,\ldots,N-1\}$ and $B' \subset B$ of positive measure such that $|T(fg)(x)| = T(e^{i\theta_x}fg)(x)$ for all $x \in B'$ and $\theta_x \in [2\pi\frac{k_0}{N},2\pi\frac{k_0 + 1}{N}]$.
For a given $\eta > 0$ whose value we shall specify later, there is $N \in \N$ and $\theta = 2 \pi \frac{k_0 + \frac12}{N}$ such that
$| \Re T(e^{i\theta}fg)(x)| \geq ( 1 - \eta ) |T(fg)(x)|$ for $x \in B'$.
We assume for simplicity of notation that $B' = B$.
Let $\delta > 0$ be a number whose value we shall specify later.
For $n \in \Z$, put $W_n = [(1+\delta)^n,(1+\delta)^{n+1}]$ and $A_n = \left(T(|f|^2)\right)^{-1}(W_n) \cap B$.
Since $\bigcup_{n \in \Z} W_n = (0,\infty)$ and $T(|f|^2)$ is strictly positive on $B$, there exists $n_0$ such that $A_{n_0}$ is of positive measure.
We have $\esssup_{x \in A_{n_0}}T(|f|^2)(x) / \essinf_{x \in A_{n_0}}T(|f|^2)(x) \leq 1 + \delta$.
Consider in a similar manner $B_n = \left(T(|g|^2)\right)^{-1}(W_n) \cap A_{n_0}$, and again, there is some $n_1$ such that $B_{n_1}$ is of positive measure.
We have $\esssup_{x \in B_{n_1}}T(|h|^2)(x) / \essinf_{x \in B_{n_1}}T(|h|^2)(x) \leq 1 + \delta$ for both $h = f$ and $h = g$.
Now with $\theta$ as above,
\begin{align*}
|\phi_{B_{n_1}}(fg)|^2 & = \left| \frac{1}{\mu(B_{n_1})} \int_{B_{n_1}} T(fg)(x) d\mu(x) \right|^2 \\
& \geq \left[ \Re \frac{1}{\mu(B_{n_1})} \int_{B_{n_1}} T(e^{i\theta}fg)(x) d\mu(x) \right]^2 \\
& \geq (1 - \eta)^2 \left[\frac{1}{\mu(B_{n_1})} \int_{B_{n_1}} |T(fg)(x)| d\mu(x) \right]^2 \\
& \geq (1+ \epsilon)(1 - \eta)^2 \left[\frac{1}{\mu(B_{n_1})} \int_{B_{n_1}} \sqrt{T(|f|^2)(x) T(|g|^2)(x)} d\mu(x) \right]^2 \\
& \geq (1+ \epsilon)(1 - \eta)^2 \essinf_{x \in B_{n_1}} T(|f|^2)(x) \essinf_{x \in B_{n_1}} T(|g|^2)(x) \\
& \geq (1+ \epsilon)(1 - \eta)^2 \frac{1}{(1+\delta)^2} \esssup_{x \in B_{n_1}} T(|f|^2)(x) \esssup_{x \in B_{n_1}} T(|g|^2)(x) \\
& \geq (1+ \epsilon)(1 - \eta)^2 \frac{1}{(1+\delta)^2} \frac{1}{\mu(B_{n_1})} \int_{B_{n_1}} T(|f|^2)(x) d\mu(x) \frac{1}{\mu(B_{n_1})} \int_{B_{n_1}} T(|g|^2)(x) d\mu(x) \\
& = (1+ \epsilon)(1 - \eta)^2 \frac{1}{(1+\delta)^2} \phi_{B_{n_1}}(|f|^2) \phi_{B_{n_1}}(|g|^2).
\end{align*}
Choosing now for the given $\epsilon > 0$ the free parameters $\eta$ and $\delta$ sufficiently close to $0$, we get the desired contradiction with the Cauchy-Schwarz inequality for the positive linear functional $\phi_{B_{n_1}}$.
\end{proof}

\subsection{$\HI$ calculus}
\label{subsec-HI}

Let $\omega \in (0,\pi)$.
We define the sector $\Sigma_\omega = \{ z \in \C \backslash \{ 0 \} :\: |\arg z| < \omega \}$.

\begin{defi}
Let $X$ be a Banach space, $\omega \in (0,\pi)$ and $A : D(A) \subseteq X \to X$ an operator.
$A$ is called $\omega$-sectorial if
\begin{enumerate}
\item $A$ is closed and densely defined on $X$.
\item The spectrum $\sigma(A)$ is contained in $\overline{\Sigma_\omega}$.
\item For any $\omega' > \omega,$ we have $\sup_{\lambda \in \C \backslash \overline{\Sigma_{\omega'}}} \| \lambda (\lambda - A)^{-1} \| < \infty$.
\end{enumerate}
\end{defi}

If $X$ is reflexive, which will always be the case in this article and $A$ is $\omega$-sectorial, then $A$ admits a canonical decomposition
\begin{equation}
\label{equ-sectorial-space-decomposition}
 A = \begin{bmatrix} A_0 & 0 \\ 0 & 0 \end{bmatrix} : X = \overline{R(A)} \oplus \Ker(A) \to \overline{R(A)} \oplus \Ker(A)
\end{equation}
such that $A_0 : D(A_0) \subseteq \overline{R(A)} \to \overline{R(A)}$ is again $\omega$-sectorial and in addition injective and has dense range \cite[Proposition 15.2]{KW04}.
Here, $R(A)$ stands for the range of $A$ and $\Ker(A)$ for its kernel.
The operator $A_0$ is called the injective part of $A$.
For $\theta \in (0,\pi),$ let 
\[ \HI(\Sigma_\theta) = \bigl\{ f : \Sigma_\theta \to \C :\: f \text{ analytic and bounded} \bigr\} \] equipped with the uniform norm $\|f\|_{\infty,\theta} = \sup_{z \in \Sigma_\theta} |f(z)|.$
Let further 
\[ \HI_0(\Sigma_\theta) = \bigl\{ f \in \HI(\Sigma_\theta):\: \exists \: C ,\epsilon > 0 :\: |f(z)| \leq C \min(|z|^\epsilon,|z|^{-\epsilon}) \bigr\}.\]
For an $\omega$-sectorial operator $A$ and $\theta \in (\omega,\pi),$ one can define a functional calculus $\HI_0(\Sigma_\theta) \to B(X),\: f \mapsto f(A)$ extending the ad hoc rational calculus, by using the Cauchy integral formula
\begin{equation}
\label{equ-CIF}
f(A) = \frac{1}{2\pi i} \int_{\partial \Sigma_{\theta'}} f(\lambda) R(\lambda,A) d\lambda ,
\end{equation}
where $\theta' = \frac12 ( \omega + \theta)$ and $\partial \Sigma_{\theta'}$ is the boundary of a sector oriented counterclockwise.
If moreover, there exists a constant $C < \infty$ such that $\|f(A)\| \leq C \| f \|_{\infty,\theta}$ for any $f \in \HI_0(\Sigma_\theta)$, then $A$ is said to have a (bounded) $\HI(\Sigma_\theta)$ calculus.
If $X$ is reflexive and $A$ has a bounded $\HI(\Sigma_\theta)$ calculus, then so does $A_0$ and $f(A) = f(A_0) \oplus 0 : \overline{R(A)} \oplus \Ker(A) \to \overline{R(A)} \oplus \Ker(A)$ for $f \in \HI_0(\Sigma_\theta)$.
Moreover, the functional calculus defined for $f \in \HI_0(\Sigma_\theta)$ can be extended to a bounded Banach algebra homomorphism $\HI(\Sigma_\theta) \to B(\overline{R(A)}),\: f \mapsto f(A_0)$.

\begin{lemma}
\label{lem-HI-convergence}
Let $\omega \in (0,\pi)$ and $A$ be an $\omega$-sectorial operator on $X$ having an $\HI(\Sigma_\theta)$ calculus for some $\theta \in (\omega,\pi)$.
Assume that $A$ is injective and has dense range (otherwise take the injective part in what follows).
Let $(f_n)_n$ be a sequence in $\HI(\Sigma_\theta)$ such that $f_n(\lambda) \to f(\lambda)$ for any $\lambda \in \Sigma_\theta$ and $\sup_n \|f_n\|_{\infty,\theta} < \infty.$
Then for any $x \in X,$ $f(A)x = \lim_n f_n(A)x.$
\end{lemma}

\begin{proof}
See \cite[Theorem 9.6]{KW04} or \cite[Lemma 2.1]{CDMcIY}.
\end{proof}

In practice, if $A$ is $\omega$-sectorial, and one wants to show that $A$ has an $\HI(\Sigma_\theta)$ calculus, the most what one can hope for is $\theta > \omega$.
In order to deal with $\theta = \omega$, we have the following refinement of the $\HI(\Sigma_\theta)$ classes in this article.

Let $J > 0$ be a parameter.
Let $(\phi_n)_{n \in \Z}$ be dyadic partition of unity, that is, a sequence of $C^\infty$ functions on $\R$ such that

\begin{enumerate}
\item $\supp( \phi_0 )\subset [-1,1]$
\item $\supp ( \phi_1 )\subset [\frac12,2]$
\item $\phi_n(t) = \phi_1(2^{-n}t)$ for $n \geq 1$
\item $\phi_{-n}(t) = \phi_n(-t)$ for $n \geq 1$
\item $\sum_{n \in \Z} \phi_n(t) = 1$ for $t \in \R$.
\end{enumerate}

Then we let $B^J_{\infty,1} = \{ f  \in L^\infty(\R) : \: \|f\|_{B^J_{\infty,1}} = \sum_{n \in \Z} 2^{J|n|} \| f \ast \check{\phi}_n \|_\infty < \infty \}$.
This class is called a Besov space.

\begin{defi}
\label{defi-HI-Besov}
Let $\theta \in (0,\pi)$.
We let 
\[ \HI(\Sigma_\theta;J) = \left\{ f \in \HI(\Sigma_\theta) :\: f(e^{\pm i \theta} e^{(\cdot)}) \in B^J_{\infty,1} \right\} \]
equipped with the norm $\|f\|_{\HI(\Sigma_\theta;J)} = \|f\|_{\infty,\theta} + \|f(e^{i\theta}e^{(\cdot)})\|_{B^J_{\infty,1}} + \|f(e^{-i\theta} e^{(\cdot)})\|_{B^J_{\infty,1}}$.
\end{defi}
Hereby we note that a holomorphic bounded function $f$ on $\Sigma_\theta$ possesses almost everywhere non-tangential limits on $\partial\Sigma_\theta$, so that $f(e^{\pm i \theta}e^{\lambda})$ is well-defined almost everywhere.
Consider the horizontal strip in the complex plane $\Str_\theta = \{ z \in \C : \: |\Im z| < \theta \}$ and note that $\lambda \mapsto e^\lambda$ maps biholomorphically $\Str_\theta \to \Sigma_\theta$.
For technical reasons, it will be easier at several instances to work with functions defined on $\Str_\theta$ instead of $\Sigma_\theta$.
We therefore consider $\HI(\Str_\theta) = \{ f : \Str_\theta \to \C : \: f \text{ holomorphic and bounded}\}$ and $\HI(\Str_\theta;J) = \{ g \in \HI(\Str_\theta): \: g(\pm i\theta + \cdot) \in B^J_{\infty,1} \}$.

\begin{lemma}
\label{lem-properties-Besov}
Let $\theta \in (0,\pi)$ and $J > 0$.
The space $\HI(\Sigma_\theta;J)$ is a Banach algebra and $\HI(\Sigma_{\theta'}) \hookrightarrow \HI(\Sigma_\theta;J)$ for any $\theta' \in (\theta,\pi)$ is a dense injection.
\end{lemma}

\begin{proof}
We show first that $\HI(\Sigma_\theta;J)$ is a Banach space.
To this end, let $(f_n)_n$ be a Cauchy sequence in $\HI(\Sigma_\theta;J)$.
Since obviously $\HI(\Sigma_\theta;J) \hookrightarrow \HI(\Sigma_\theta)$, $(f_n)_n$ is a Cauchy sequence in $\HI(\Sigma_\theta)$ and has a limit $f \in \HI(\Sigma_\theta)$.
Moreover, since $B^J_{\infty,1}$ is a Banach space, $f_n(e^{\pm i\theta} e^{(\cdot)})$ has a limit $g(e^{\pm i \theta}e^{(\cdot)})$ in $B^J_{\infty,1}$.
Since $f(e^{\pm i \theta} e^{(\cdot)})$ is also a limit in $L^\infty(\partial\Sigma_\theta)$ of $f_n(e^{\pm i \theta} e^{(\cdot)})$, we have $g(e^{\pm i \theta}e^{\lambda}) = f(e^{\pm i \theta} e^{\lambda})$, $\lambda \in \R$.
Moreover, 
\[ \|f - f_n\|_{\HI(\Sigma_\theta;J)} = \|f-f_n\|_{\infty,\theta} + \|(f-f_n)(e^{\pm i\theta}e^{(\cdot)})\|_{B^J_{\infty,1}} \to 0 + 0. \]
Thus, $\HI(\Sigma_\theta;J)$ is a Banach space.
That $\HI(\Sigma_\theta;J)$ is also a Banach algebra follows from the fact that $\HI(\Sigma_\theta)$ and $B^J_{\infty,1}$ are Banach algebras and thus,
\begin{align*}
\|fg\|_{\HI(\Sigma_\theta;J)} & = \|fg\|_{\infty,\theta} + \|(fg)(e^{\pm i \theta}e^{(\cdot)})\|_{B^J_{\infty,1}} \\
&  \leq \|f\|_{\infty,\theta} \|g\|_{\infty,\theta} + C \|f(e^{\pm i \theta}e^{(\cdot)})\|_{B^J_{\infty,1}} \|g(e^{\pm i \theta}e^{(\cdot)})\|_{B^J_{\infty,1}} \\
& \leq C \|f\|_{\HI(\Sigma_\theta;J)} \|g\|_{\HI(\Sigma_\theta;J)} .
\end{align*}
For the statement of dense injection, we pass the problem from the sector $\Sigma_\theta$ to the strip $\Str_\theta$ and thus have to show that $\HI(\Str_{\theta'}) \hookrightarrow \HI(\Str_\theta;J)$ densely.
That this is indeed an injection follows from \cite[Remark 4.16]{KrPhD}.
For the density, we let $f \in \HI(\Str_\theta;J)$ be given and take $f_n = f \ast \sum_{k = -n}^n \check{\phi}_k$, where $(\phi_n)_n$ is a dyadic partition of unity as above.
Then according to \cite[Lemma 4.15 (2)]{KrPhD}, $f_n(\pm i \theta + (\cdot))$ converges to $f$ in $B^J_{\infty,1}$.
Moreover, according to \cite[p.~416]{GCMMST}, we have $\|g\|_{\HI(\Str_\theta)} \leq 2 \|g(\pm i \theta+ (\cdot))\|_{B^J_{\infty,1}}$ for $g \in \HI(\Str_\theta;J)$, so that $\|f_n - f\|_{\HI(\Str_\theta)} \leq 2 \|(f_n - f)(\pm i \theta + (\cdot)) \|_{B^J_{\infty,1}} \to 0$, too.
We finally infer that $\|f_n - f \|_{\HI(\Str_\theta;J)} = \|f_n - f\|_{\HI(\Str_\theta)} + \|(f_n - f)(\pm i \theta + (\cdot))\|_{B^J_{\infty,1}} \to 0$.
\end{proof}

Lemma \ref{lem-properties-Besov} enables us to define a bounded $\HI(\Sigma_\omega;J)$ calculus via the well-known $\HI(\Sigma_\theta)$ calculus.
Namely, we say that for $\omega \in (0,\pi)$ and $A$ an injective $\omega$-sectorial operator with dense range, $A$ has a bounded $\HI(\Sigma_\omega;J)$ calculus if 
\begin{equation}
\label{equ-definition-Besov-calculus}
 \|f(A)\| \leq C \|f\|_{\HI(\Sigma_\omega;J)} \quad (f \in \HI(\Sigma_\theta))
\end{equation}
for some/any $\theta \in (\omega,\pi)$.
In this case, by density of $\HI(\Sigma_\theta)$ in $\HI(\Sigma_\omega;J)$, we can define $f(A)$ for all $f \in \HI(\Sigma_\omega;J)$ and \eqref{equ-definition-Besov-calculus} extends to all $f \in \HI(\Sigma_\omega;J)$.
Clearly, $\HI(\Sigma_\omega;J) \to B(X), \: f \mapsto f(A)$ is then an algebra homomorphism.
The following lemma is a variant of \cite[Theorem 4.10]{CDMcIY}.

\begin{lemma}
\label{lem-HI-to-Besov}
Let $\omega \in (0,\pi)$ and $J > 0$.
Let $X$ be a Banach space and $A$ be an $\omega$-sectorial operator on $X$ which is injective and has dense range.
Assume that for any $\theta > \omega$, $A$ has an $\HI(\Sigma_\theta)$ calculus and that there is $C < \infty$ such that
\[ \|f(A)\| \leq C (\theta - \omega)^{-J} \|f\|_{\infty,\theta} \quad (f \in \HI(\Sigma_\theta),\: \theta > \omega) . \]
Then $A$ has a bounded $\HI(\Sigma_\omega;J)$ calculus, meaning that 
\[ \|f(A)\| \leq C' \|f\|_{\HI(\Sigma_\omega;J)} \quad (f \in \HI(\Sigma_\omega;J)) .\]
\end{lemma}

\begin{proof}
We let $B = \log(A)$, which is an $\omega$-strip-type operator in the sense of \cite{Haase}.
According to \cite[Proposition 5.3.3]{Haase}, the hypothesis of the lemma implies that for $\theta > \omega$ and $g : \Str_\theta \to \C$ analytic and bounded, we have
$\|g(B)\| \leq C ( \theta - \omega)^{-J} \|g\|_{\infty,\theta}$.
On the level of spectral multiplier functions, this can be simply seen as the correspondence $f \in \HI(\Sigma_\omega) \cong g(\lambda) = f(e^\lambda) \in \HI(\Str_\omega)$.
It suffices now to show that for $\theta > \omega$
\begin{equation}
\label{equ-lem-HI-to-Besov}
\|g(B)\| \leq C \|g\|_{\HI(\Str_\omega;J)} \quad (g \in \HI(\Str_\theta)).
\end{equation}
Let $(\phi_n)_{n \in \Z}$ be a dyadic partition of unity as above.
According to the Paley-Wiener theorem, see also \cite[Proof of Theorem 4.10]{CDMcIY}, $g \ast \check{\phi}_n$ is an entire function, and moreover, for $b > 0$,
\[ \sup_{x \in \R,\: |y - (\pm \omega)| \leq b} |g \ast \check{\phi}_n(x+iy)| \leq \exp(b2^{|n| + 1}) \|g\ast\check{\phi}_n\|_{L^\infty(\R \pm i \omega)} . \]
Choose now $b = 2^{-|n|}$.
Then using the maximum principle to bound 
\[ \sup_{x \in \R,\: |y| \leq \omega} |g\ast \check{\phi}_n(x+iy)| \leq \sup_{x \in \R,\: |y - (\pm \omega)| \leq 2^{-|n|}} |g\ast \check{\phi}_n(x+iy)|, \] we obtain
\[ \sup_{x \in \R,\: |y| \leq \omega + 2^{-|n|}} | g \ast \check{\phi}_n(x + iy)| \leq C \| g \ast \check{\phi}_n\|_{L^\infty(\R \pm i \omega)}. \]
Thus, $\|g\ast\check{\phi}_n(B)\| \leq C (2^{-|n|})^{-J} \|g \ast \check{\phi}_n\|_{\HI(\Str_{\omega + 2^{-|n|}})} \leq C (2^{-|n|})^{-J} \|g \ast \check{\phi}_n\|_{L^\infty(\R \pm i \omega)}$.
It suffices to check now that
\begin{equation}
\label{equ-2-lem-HI-to-Besov}
\sum_{n \in \Z} g \ast \check{\phi}_n = g \quad (\text{convergence in }\HI(\Str_{\theta'}))
\end{equation}
for $\theta' = \frac12(\theta + \omega)$.
Indeed, then by bounded $\HI(\Str_{\theta'})$ calculus of $B$, 
\[ \|g(B)\| \leq \sum_{n \in \Z} \|g\ast \check{\phi}_n(B)\| \lesssim \sum_{n \in \Z} 2^{|n|J} \|g\ast\check{\phi}_n\|_{L^\infty(\R \pm i \omega)} \lesssim \|g\|_{\HI(\Str_\omega;J)}, \] and so, \eqref{equ-lem-HI-to-Besov} follows.
To show \eqref{equ-2-lem-HI-to-Besov}, we note that
\[ \| g - g \ast \sum_{n = -N}^N \check{\phi}_n \|_{\HI(\Str_{\theta'})} \leq \| g - g \ast \sum_{n = -N}^N \check{\phi}_n \|_{L^\infty(\R \pm i \theta')} \leq \| g (\cdot \pm i \theta')- g (\cdot \pm i \theta') \ast \sum_{n = -N}^N \check{\phi}_n \|_{B^1_{\infty,1}} , \]
and the last quantity converges to $0$ according to \cite[Lemma 4.15 (2)]{KrPhD}.
\end{proof}

A typical $\HI(\Sigma_{\frac{\pi}{2}};J)$ function (i.e. not belonging to any $\HI(\Sigma_\theta)$ for $\theta > \frac{\pi}{2}$) is given in the following lemma.

\begin{lemma}
\label{lem-Besov-norm-regularized-semigroup}
Let $J > 0$, $\epsilon > 0$ and $t \geq 0$.
Let $m_t(\lambda) = (1 + \lambda)^{-J-\epsilon} e^{-t\lambda}$.
Then $m_t$ belongs to the class $\HI(\Sigma_{\frac{\pi}{2}};J)$ and we have the estimate
\[ \|m_t\|_{\HI(\Sigma_{\frac{\pi}{2}};J)} \leq C (1 + t)^{J + \epsilon} . \]
\end{lemma}

\begin{proof}
It is easy to check that $m_t$ is holomorphic on $\Sigma_{\frac{\pi}{2}} = \C_+$ and that $\|m_t\|_{\infty,\frac{\pi}{2}} \leq 1$.
It thus remains to estimate $\| m_t(\pm i e^{(\cdot)}) \|_{B^{J}_{\infty,1}}$.
Since $B^J_{\infty,1}$ is a Banach algebra, we decompose
\begin{align*}
\| m_t(i e^\lambda)\|_{B^J_{\infty,1}} & =  \| (1 + i e^\lambda)^{-J - \epsilon} \exp(-ite^{\lambda})\|_{B^{J}_{\infty,1}} \\
& \lesssim \left\|\frac{(1 + t e^\lambda)^{J + \epsilon}}{(1 + ie^\lambda)^{J + \epsilon}}\right\|_{B^J_{\infty,1}} \left\| (1 + t e^{\lambda} )^{-J - \epsilon} \exp(-ite^\lambda) \right\|_{B^J_{\infty,1}} \\
& \lesssim \left\|\frac{(1 + t \lambda)^{J + \epsilon}}{(1 + i\lambda)^{J + \epsilon}}\right\|_{\infty,\delta} \left\| (1 + t e^{\lambda} )^{-J - \epsilon} \exp(-ite^\lambda) \right\|_{B^J_{\infty,1}},
\end{align*}
since $\HI(\Str_\delta) \hookrightarrow B^J_{\infty,1}$ according to \cite[Remark 4.16]{KrPhD}, where $\delta > 0$ is a small auxiliary angle.
It is easy to estimate $ \left\|\frac{(1 + t \lambda)^{J + \epsilon}}{(1 + i\lambda)^{J + \epsilon}}\right\|_{\infty,\delta} \leq C (1 + t)^{J + \epsilon}$.
Moreover, $\left\| (1 + t e^{\lambda} )^{-J - \epsilon} \exp(-ite^\lambda) \right\|_{B^J_{\infty,1}} \leq C$ according to \cite[Proof of Lemma 3.9 (2)]{KrW3}.
The term $\|m_t(-ie^\lambda)\|_{B^J_{\infty,1}}$ is estimated in the same way.
\end{proof}

The following proposition is a variant of \cite[Theorem 4.4]{CDMcIY}.
Note however that we do not assume that $A_Y$ is $\omega$-sectorial but that this is a consequence of the proposition.
Moreover, we precise the dependence of the $\HI(\Sigma_\theta)$ calculus norm on the angle $\theta > \frac{\pi}{2}$.

\begin{prop}
\label{prop-bilinear-to-calculus}
Let $X$ and $Y$ be Banach spaces such that $X,Y \subseteq Z$ with a bigger Banach space $Z$ and similarly for the duals $X',Y' \subseteq \tilde{Z}$.
Assume that $Y$ is reflexive.
Assume that $X \cap Y$ is dense in $Y$ and that $X' \cap Y'$ is dense in $Y'$.
Assume furthermore that for $f \in X \cap Y$ and $g \in X' \cap Y'$, the duality brackets $\langle f , g \rangle_{X,X'} = \langle f , g \rangle_{Y,Y'}$ coincide.
Let $A$ be the (negative) generator of an analytic semigroup $(T_t)_t$ on $X$, i.e. $A$ is $\omega$-sectorial for some $\omega < \frac{\pi}{2}$.
Assume
\begin{equation}
\label{equ-1-prop-bilinear-to-calculus}
\int_0^\infty | \langle A T_tf, g \rangle | dt \leq C \|f\|_Y \|g\|_{Y'} \quad (f \in X \cap Y, \: g \in X' \cap Y').
\end{equation}
Then there exists a $\frac{\pi}{2}$-sectorial operator $A_Y$ on $Y$ such that for $J > 1$
\begin{equation}
\label{equ-4-prop-bilinear-to-calculus}
\|m(A_Y)\|_{Y \to Y} \leq C_J (\theta - \frac{\pi}{2})^{-J} \|m\|_{\infty,\theta} \quad ( m \in \HI_0(\Sigma_\theta), \: \theta \in (\frac{\pi}{2},\pi) )
\end{equation}
and $m(A_Y)f = m(A)f$ for any such $m$ and $f \in X \cap Y$.
\end{prop}

\begin{proof}
We place ourselves in the notation of \cite[Proof of Theorem 4.4]{CDMcIY} and put there $\psi(z) = z e^{-z}$ and for given $\epsilon > 0$ sufficiently small so that $A$ is $(\frac{\pi}{2} - \epsilon)$-sectorial, $\mu = \frac{\pi}{2} - \epsilon$, $\nu  = \frac{\pi}{2}$, $\eta = \frac{\pi}{2} + 3 \epsilon > 2 \nu - \mu = \frac{\pi}{2} + \epsilon$, and $\alpha = \frac{\pi}{2} + 2 \epsilon \in (2 \nu - \mu , \eta)$.
Pick a $b \in \HI_0(\Sigma_\eta)$.
Then one has according to \cite[(4.3)]{CDMcIY}
\[ b(\lambda) = \int_0^\infty (\beta^+(t) + \beta^-(t)) \psi\left(\frac{\lambda}{t}\right) \frac{dt}t \quad (\lambda \in \Sigma_\mu) \]
with $\beta^\pm$ defined via the Fourier transform
\[ (\beta^\pm_e)\hat{\phantom{i}}(t) = \frac12 \hat{\gamma}_e(t) \hat{b}_e(t) e^{\mp \alpha t} .\]
Here, the index $e$ stands for composition with the exponential function, so $h_e = h \circ \exp : \R \to \C$ for a function $h : (0, \infty) \to \C$.
Moreover, 
\[ \hat{\gamma}_e (t) = \frac{1}{\hat{\psi}_e(t) \cosh(\alpha t)} = \frac{1}{\Gamma(1-it) \cosh(\alpha t)} \] for our particular choice of $\psi$, see \cite[Example 4.8]{CDMcIY}.
Since both $\hat{\gamma}_e$ and $\hat{b}_e \cdot e^{\mp \alpha t}$ belong to the Schwartz class, also $\beta^\pm_e$ belongs to the Schwartz class and thus, 
\[ \int_0^\infty |\beta^\pm(t)| \frac{dt}{t} < \infty. \]
Therefore and since $\|R(\lambda,A)\| \leq C_\mu \frac{1}{|\lambda|}$ for $\lambda \in \partial \Sigma_\mu \backslash \{ 0 \}$, we can apply Fubini below and obtain for $f \in X$ and $g \in X'$,
\begin{align*}
\langle b(A) f , g \rangle & = \frac{1}{2\pi i} \int_{\partial\Sigma_\mu} \int_0^\infty (\beta^+(t) + \beta^-(t)) \psi\left(\frac{\lambda}{t}\right) \langle R(\lambda,A) f , g \rangle d\lambda \\
& = \int_0^\infty (\beta^+(t) + \beta^-(t)) \frac{1}{2 \pi i} \int_{\partial\Sigma_\mu} \psi\left(\frac{\lambda}{t}\right) \langle R(\lambda,A) f, g \rangle d\lambda \frac{dt}{t} \\
& = \int_0^\infty (\beta^+(t) + \beta^-(t)) \left\langle \psi\left(\frac{1}{t}A\right) f, g \right\rangle \frac{dt}{t} \\
& = \int_0^\infty (\beta^+(t) + \beta^-(t)) \langle AT_tf, g \rangle dt.
\end{align*}
Now by assumption, we have for $f \in X \cap Y$ and $g \in X' \cap Y'$
\begin{equation}
\label{equ-2-bilinear-to-calculus}
 | \langle b(A)f,g \rangle| \leq \int_0^\infty (|\beta^+(t)| + |\beta^-(t)|) |\langle AT_tf,g\rangle| dt  \leq C (\|\beta^+ \|_{L^\infty(\R_+)} + \|\beta^-\|_{L^\infty(\R_+)}) \|f\|_Y \|g\|_{Y'} .
\end{equation}
We estimate $\|\beta^\pm\|_{L^\infty(\R_+)}$.
\begin{align}
\|\beta^\pm\|_{L^\infty(\R_+)} & = \|\beta^\pm_e\|_{L^\infty(\R)} = \frac12 \|\gamma_e \ast b_e((\cdot) \pm i \alpha) \|_{L^\infty(\R)} \\
& \leq \frac12 \|\gamma_e\|_{L^1(\R)} \|b_e((\cdot) \pm i \alpha)\|_{L^\infty(\R)} \nonumber\\
& = \frac12 \|\gamma_e\|_{L^1(\R)} \|b\|_{\infty,\alpha} \label{equ-5-bilinear-to-calculus}.
\end{align}
The technical Lemma \ref{lem-technical-bilinear-to-calculus} below gives an estimate of $\|\gamma_e\|_{L^1(\R)}$, so that \eqref{equ-2-bilinear-to-calculus} shows that $b(A)$ extends to a bounded operator on $Y$ with $\|b(A)\|_{Y \to Y} \leq C_\eta \|b\|_{\infty,\eta}$.
Take now for $\mu \in \C \backslash \overline{\Sigma_\eta}$, $b(\lambda) = \lambda (\mu - \lambda)^{-1} \in \HI_0(\Sigma_\eta)$, so that $J(\mu) = \frac{1}{\mu}(\Id + b(A)) = R(\mu,A)$ extends to an operator on $Y$ with uniform norm bound for these $\mu$.
Since $R(\mu,A)$ is a resolvent, $J(\mu)$ is a pseudo-resolvent in the sense of \cite[Section 1.9]{Pazy}.
We claim that
\begin{equation}
\label{equ-3-bilinear-to-calculus}
\mu J(\mu)f \to f\text{ weakly in }Y \text{ as }\mu \to - \infty \quad(f \in Y).
\end{equation}
Indeed, according to \cite[Proposition 15.2]{KW04}, $\mu R(\mu,A) f \to f$ strongly in $X$ for any $f \in X$.
Thus, $\langle \mu R(\mu,A)f,g \rangle \to \langle f, g \rangle$ for any $f \in X\cap Y$ and $g \in X' \cap Y'$.
A $3\epsilon$ argument together with the uniform norm bound of $\mu J(\mu)$ in $B(Y)$ allows then to deduce \eqref{equ-3-bilinear-to-calculus} for $f \in X \cap Y$, and then for any $f \in Y$.
This implies that the null space $\Ker(J(\mu))$ which is independent of $\mu$ \cite[Lemma 9.2 p.~36]{Pazy} equals $\{ 0 \}$.
Indeed, if $f \in N(J(\mu))$, then $\mu J(\mu)f = 0$ for any $\mu$.
Letting $\mu \to - \infty$ together with \eqref{equ-3-bilinear-to-calculus} shows that $f = 0$.
Then proceeding as in \cite[Proof of Theorem 9.3 p.~37]{Pazy} allows to define an operator $A_Y : R(J(\mu)) \to Y$ where $R(J(\mu))$ stands for the range of $J(\mu)$, such that $R(\mu,A_Y) = J(\mu)$.
Then, since the resolvent bound \cite[(15.1)]{KW04}  of $\|\mu R(\mu,A_Y)\|$ is satisfied, \cite[15.2 Proposition c)]{KW04} allows with the reflexivity of $Y$ to deduce that $A_Y$ is densely defined and that $\overline{R(A_Y)} \oplus \Ker(A_Y) = Y$.
In particular, $A_Y$ is sectorial and since $\eta > \frac{\pi}{2}$ was arbitrary, $A_Y$ is $\frac{\pi}{2}$-sectorial.
Moreover, $R(\mu,A_Y)f= R(\mu,A)f$ for $f \in X \cap Y$, so that the Cauchy integral formula \eqref{equ-CIF} implies that for $b \in \HI_0(\Sigma_\eta)$, we have $b(A_Y)f = b(A)f$ for such $f$.
It remains to show the estimate \eqref{equ-4-prop-bilinear-to-calculus}, which follows from the technical Lemma \ref{lem-technical-bilinear-to-calculus} together with \eqref{equ-2-bilinear-to-calculus} and \eqref{equ-5-bilinear-to-calculus}.
\end{proof}

\begin{lemma}
\label{lem-technical-bilinear-to-calculus}
Let $\alpha = \frac{\pi}{2} + 2 \epsilon$ and $\displaystyle \hat{\gamma}_e(t) = \frac{1}{\Gamma(1-it) \cosh(\alpha t)}$.
Then for $\delta > 0$, we have the estimate
\[ \|\gamma_e\|_{L^1(\R)} \leq C_\delta \epsilon^{-(1+\delta)} . \]
\end{lemma}

\begin{proof}
Since the Fourier transform is bounded $W^{1,2}(\R) \to L^1(\R)$, it suffices to estimate $\|\hat{\gamma}_e\|_{L^2(\R)}$ and $\|\hat{\gamma}_e'\|_{L^2(\R)}$.
According to \cite[Example 4.8]{CDMcIY}, we have $|\Gamma(1-it)| \geq C e^{-\frac{\pi}{2} | t | }$, so that $|\hat{\gamma}_e(t)| \leq C e^{\frac{\pi}{2} |t| - \alpha |t|} = C e^{-2 \epsilon |t|}$.
We have $\left( \int_\R |e^{-2\epsilon t}|^2 dt \right)^{\frac12} \lesssim \epsilon^{-\frac12}$, so that $\|\hat{\gamma}_e\|_{L^2(\R)}$ is estimated.
For the derivative, we have
\begin{align}
| \hat{\gamma}_e'(t) | & = \left| - \frac{-i \Gamma'(1-it) \cosh(\alpha t) + \Gamma(1 - it) \alpha \sinh(\alpha t)}{\Gamma(1-it)^2 \cosh^2(\alpha t)} \right| \nonumber \\
& \leq |\Gamma'(1-it)| \frac{1}{|\Gamma(1-it)|^2 |\cosh(\alpha t)| } + \frac{|\Gamma(1-it)|^{-1} |\alpha| \, |\sinh(\alpha t)|}{|\cosh(\alpha t)|^2 } \nonumber \\
& \leq C | \Gamma'(1-it)| \exp\left(2 \frac{\pi}{2} |t| - \alpha |t|\right) + C \exp\left(\frac{\pi}{2} |t| - \alpha |t|\right).
\label{equ-proof-lem-technical-bilinear-to-calculus}
\end{align}
The second summand can be treated as above.
For the first summand, we have to find an estimate for $\Gamma'(1-it)$ with $|t| \geq 1 > \delta$.
We recall that $|\Gamma(x+iy)| \leq C e^{-\frac{\pi}{2} |y|} |y|^{x - \frac12}$ according to \cite[p.~15]{Leb}.
We write, according to the Cauchy integral formula,
\[ \Gamma'(1-it) = \frac{1}{2\pi i} \int_{\partial B(1-it,\delta)} \frac{\Gamma(z)}{(z-(1-it))^2} dz . \]
Here we have choosen as contour path a circle of radius $\delta$ as in the statement.
For $x + iy \in \partial B(1-it,\delta)$, we have $||y|-|t|| \leq \delta$ and $x \leq 1 + \delta$.
Therefore, since $x - \frac12 \geq 0$, $|y|^{x - \frac12} \leq (|t| + \delta)^{x - \frac12} \leq (|t| + \delta)^{\frac12 + \delta} \lesssim |t|^{\frac12 + \delta}$.
Moreover, $e^{-\frac{\pi}{2} | y | } \leq e^{-\frac{\pi}{2} (|t| - \delta) } \lesssim e^{-\frac{\pi}{2} |t|}$.
Thus,
\[ \sup_{x + i y \in \partial B (1 - it, \delta)} |\Gamma(x+iy)| \lesssim |t|^{\frac12 + \delta} e^{-\frac{\pi}{2} |t|} ,\]
which by the Cauchy integral formula implies
\[ |\Gamma'(1-it)| \leq C \delta \frac{1}{\delta^2} |t|^{\frac12 + \delta} e^{-\frac{\pi}{2} |t|} . \]
Going up, we estimate the first summand in \eqref{equ-proof-lem-technical-bilinear-to-calculus} by
\begin{align*}
 | \Gamma'(1-it)| \exp\left(2 \frac{\pi}{2} |t| - \alpha |t|\right)  & \lesssim_\delta |t|^{\frac12 + \delta} \exp\left((\frac{\pi}{2} - \alpha)|t|\right) \\
& = |t|^{\frac12 +  \delta} \exp \left(-2\epsilon |t|\right).
\end{align*}
Passing to $L^2(\R)$-norms, we obtain
\begin{align*}
\left(\int_\R |\Gamma'(1-it) \exp(2 \frac{\pi}{2}|t| - \alpha |t|)|^2 dt \right)^{\frac12} & \lesssim C + \left( \int_1^\infty |t|^{1 + 2 \delta} \exp(-4\epsilon |t|) dt \right)^{\frac12} \\
& = C + \left( \int_{4 \epsilon}^\infty \frac{|t|^{1 + 2 \delta}}{(4\epsilon)^{1+2\delta+1}} \exp(-|t|) dt \right)^{\frac12} \\
& \lesssim \left( \frac{1}{\epsilon^{2 + 2 \delta}} \right)^{\frac12} = \epsilon^{-1 -\delta}.
\end{align*}
This concludes the proof of the lemma and thus, also that of Proposition \ref{prop-bilinear-to-calculus}.
\end{proof}

In our Main Theorem \ref{thm-bilinear} on weighted $L^2$ functional calculus, we shall show the hypotheses of Proposition \ref{prop-bilinear-to-calculus} in the case $X = L^2(\Omega,\mu)$ and $Y = L^2(\Omega,wd\mu)$ with $w : \Omega \to (0,\infty)$ a certain weight, i.e. a positive measurable function.
In the next proposition, we spell out the most general functional calculus consequence in this situation.

\begin{prop}
\label{prop-bilinear-to-calculus-weight}
Let $(\Omega,\mu)$ be a $\sigma$-finite measure space and $A$ be a positive self-adjoint operator on $L^2(\Omega,\mu)$.
Let $w : \Omega \to (0,\infty)$ be a positive measurable function.
We consider the duality bracket $\langle f , g \rangle = \int_\Omega f(x) g(x) d\mu(x)$, so that $L^2(\Omega,\mu)$ is its own dual and $L^2(\Omega,w^{-1}d\mu)$ is the dual of $L^2(\Omega,w d\mu)$.
Assume that for any $f \in L^2(\Omega,\mu) \cap L^2(\Omega,w d\mu)$ and $g \in L^2(\Omega,\mu) \cap L^2(\Omega,w^{-1}d\mu)$, we have
\[ \int_0^\infty | \langle AT_t f, g \rangle| dt \leq C \|f\|_{L^2(\Omega,wd\mu)} \|g\|_{L^2(\Omega,w^{-1}d\mu)} . \]
Then for $J > 1$ and $m \in \HI(\Sigma_{\frac{\pi}{2}};J)$, $m(A)$ extends to a bounded operator on $L^2(\Omega,wd\mu)$ and
\[ \|m(A)\|_{L^2(\Omega,wd\mu) \to L^2(\Omega,wd\mu)} \lesssim_J |m(0)| + \|m\|_{\HI(\Sigma_{\frac{\pi}{2}};J)} .\]
\end{prop}

\begin{remark}
\label{rem-bilinear-to-calculus-weight}
In fact, we will even show the following form of Proposition \ref{prop-bilinear-to-calculus-weight} in the proof below.
Under the hypotheses of Proposition \ref{prop-bilinear-to-calculus-weight}, according to Proposition \ref{prop-bilinear-to-calculus}, there exists a $\frac{\pi}{2}$-sectorial operator $A_Y$ on $Y = L^2(\Omega,wd\mu)$ with $m(A_Y)f = m(A)f$ for $m \in \HI_0(\Sigma_{\frac{\pi}{2} + \epsilon})$ and $f \in L^2(\Omega,\mu) \cap L^2(\Omega,wd\mu)$.
Then according to \eqref{equ-sectorial-space-decomposition}, $L^2(\Omega,wd\mu)$ decomposes as $L^2(\Omega,wd\mu) = \overline{R(A_Y)} \oplus \Ker(A_Y)$.
The part $A_{Y,0}$ on $\overline{R(A_Y)}$ is injective and has dense range, so that $A_{Y,0}$ has a bounded $\HI(\Sigma_{\frac{\pi}{2}};J)$ calculus for any $J > 1$, according to Lemma \ref{lem-HI-to-Besov}.
Putting then, in accordance with \eqref{equ-sectorial-space-decomposition}, for $m \in \HI(\Sigma_{\frac{\pi}{2}};J)$
\begin{equation}
\label{equ-rem-bilinear-to-calculus-weight}
 m(A_Y) = \begin{bmatrix} m(A_{Y,0}) & 0 \\ 0 & m(0) \Id_{\Ker(A_Y)} \end{bmatrix} ,
\end{equation}
we have 
\begin{equation}
\label{equ-2-rem-bilinear-to-calculus-weight}
\|m(A_Y)\|_{L^2(\Omega,wd\mu) \to L^2(\Omega,wd\mu)} \lesssim_J |m(0)| + \|m\|_{\HI(\Sigma_{\frac{\pi}{2}};J)}
\end{equation}
and $m(A)f = m(A_Y)f$ for $f \in L^2(\Omega,\mu) \cap L^2(\Omega,wd\mu)$.
\end{remark}

\begin{proof}[of Proposition \ref{prop-bilinear-to-calculus-weight} and Remark \ref{rem-bilinear-to-calculus-weight}]
Consider the $\frac{\pi}{2}$-sectorial operator $A_Y$ on $Y = L^2(\Omega,wd\mu)$ from Proposition \ref{prop-bilinear-to-calculus} and its space decomposition $Y = \overline{R(A_Y)} \oplus \Ker(A_Y)$ from \eqref{equ-sectorial-space-decomposition}.
We define for $J > 1$ and $m \in \HI(\Sigma_{\frac{\pi}{2}};J)$ the operator $m(A_Y) : Y \to Y$ as in \eqref{equ-rem-bilinear-to-calculus-weight}.
The norm estimate \eqref{equ-2-rem-bilinear-to-calculus-weight} is immediate.
It only remains to check that $m(A)f = m(A_Y)f$ for $f \in L^2(\Omega,\mu) \cap L^2(\Omega,wd\mu)$.
We consider first $m \in \HI(\Sigma_\theta)$ for some $\theta > \frac{\pi}{2}$.
Write $f = f_1 \oplus f_2$ according to $Y = \overline{R(A_Y)} \oplus \Ker(A_Y)$ and $f = g_1 \oplus g_2$ according to $X := L^2(\Omega,\mu) = \overline{R(A)} \oplus \Ker(A)$.
Let $\rho_n(\lambda) = n (n + \lambda)^{-1} - (1 + n\lambda)^{-1} \in \HI_0(\Sigma_\theta)$ from \cite[9.4 Proposition (2)]{KW04}.
According to \cite[15.2 Proposition]{KW04}, $f_1 = Y-\lim_n \rho_n(A_Y)f = (X+Y)-\lim_n \rho_n(A)f = X-\lim_n \rho_n(A)f = g_1$, so that $f_1 = g_1$ and $f_2 = g_2$.
It suffices to check that $m(A_Y)f_1 = m(A)f_1$ and $m(A_Y)f_2 = m(A)f_2$.
We have according to Lemma \ref{lem-HI-convergence}, $m(A_Y)f_1 = Y-\lim_n (m\rho_n)(A_Y)f_1 = (X+Y)-\lim_n (m\rho_n)(A)f_1 = X-\lim_n (m\rho_n)(A)f_1 = m(A)f_1$.
Moreover, $m(A_Y)f_2 = m(0)f_2 = m(A)f_2$.
We have shown $m(A)f = m(A_Y)f$  for $m \in \HI(\Sigma_\theta)$.
For a more general $m \in \HI(\Sigma_{\frac{\pi}{2}};J)$, we pick an approximating sequence $m_n \in \HI(\Sigma_\theta)$, $m_n \to m$ in $\HI(\Sigma_{\frac{\pi}{2}};J)$ according to Lemma \ref{lem-properties-Besov}.
Then we have
\begin{align*}
m(A_Y)f & = m(A_Y)f_1 \oplus m(A_Y)f_2 = Y-\lim_n m_n(A_Y)f_1 \oplus m(0)f_2\\
& = (X+Y)-\lim_n m_n(A)f_1 \oplus m(0)f_2 = X-\lim_n m_n(A)f_1 \oplus m(A)f_2 = m(A)f. 
\end{align*}
\end{proof}

\subsection{Weights}
\label{subsec-weights}

\begin{defi}
Let $(\Omega,\mu)$ be a $\sigma$-finite measure space and $w :\Omega \to (0,\infty)$ a strictly positive measurable function.
Then $w$ is called a weight.
\end{defi}

If $T$ is a positive mapping on $L^\infty(\Omega)$ and $w : \Omega \to (0,\infty)$ is a weight, then we can define $T(w) : \Omega \to [0,\infty]$ to be a measurable function as follows.
Let $w^{(n)} = w \cdot 1_{w \leq n}$.
Then $w^{(n)}$ is an increasing sequence of $L^\infty$ functions, so $T\left(w^{(n)}\right)$ is an increasing sequence of measurable functions, which thus converges to a measurable function $T(w) : \Omega \to [0,\infty]$.
Note that if $T$ extends to a bounded operator on $L^1$ and $w \in L^1(\Omega)$, then $T(w)$ is unambiguously defined.
Also if $w \in L^1(\Omega) + L^\infty(\Omega)$, then $T(w)$ is unambiguously defined.

For practical purposes, we define the following cut-offs of a weight.

\begin{defi}
\label{defi-truncated-weight}
Let $w$ be a weight on $\Omega$ and $n \in \N$.
Then we define the cut-off weight
\[ w_n(x) = \begin{cases} w(x) & : \: w(x) \in [\frac1n,n] \\ \frac1n & : \: w(x) < \frac1n \\ n & :\: w(x) > n \end{cases}. \]
\end{defi}

\begin{defi}
Let $(T_t)_{t \geq 0}$ be a semigroup acting on $L^1(\Omega)$ and $w: \Omega \to (0,\infty)$ a weight.
Let $p \in (1,\infty)$.
We define the characteristic of $w$ as 
\[Q^A_p(w) = \sup_{t \geq 0} \esssup_{x \in \Omega} T_t(w)(x) (T_t(u)(x))^{p-1} ,\]
where $u = w^{-\frac{1}{p-1}}$.
If $p = 2$, this becomes $Q^A_2(w) = \sup_{t \geq 0} \esssup_{x \in \Omega} T_t(w)(x) T_t(w^{-1})(x)$.
The upper index $A$ stands for the (negative) generator $A$ of $(T_t)_{t \geq 0}$.
Moreover, we define $Q^A_p = \{ w : \Omega \to (0,\infty) :\: Q^A_p(w) < \infty \}$.
\end{defi}

\begin{lemma}
\label{lem-truncated-weight}
Let $w$ be a weight on $(\Omega,\mu)$ and $n \in \N$.
Let $(T_t)_{t \geq 0}$ be a positive semigroup acting contractively on $L^\infty(\Omega)$.
If $w$ belongs to $Q^A_2$, then the cut-off weight $w_n$ also belongs to $Q^A_2$ and we have $Q^A_2(w_n) \leq Q^A_2(w)$.
\end{lemma}

\begin{proof}
The idea of the proof is from \cite[Lemma 4]{Dah}.
We let $\chi_1 = 1_{w \leq n}$ and $\chi_2 = 1 - \chi_1 = 1_{w > n}$.
Furthermore, we define $\alpha_1(x) = T_t(\chi_1)(x)$, $\alpha_2(x) = T_t(\chi_2)(x)$, and moreover for $v : \Omega \to (0,\infty)$ positive, $T_t^1v(x) = T_t(v\chi_1)(x) \frac{1}{\alpha_1(x)}$, $T_t^2v(x) = T_t(v \chi_2)(x) \frac{1}{\alpha_2(x)}$.
Then we have
\[ T_tv(x) = \alpha_1(x) T_t^1v(x) + \alpha_2(x) T_t^2v(x) . \]
Here we note that $\alpha_i(x) = 0$, for $x$ belonging to a set of positive measure, implies $0 \leq T_t(v \chi_i)(x) = \lim_{m \to \infty} T_t(v 1_{v \leq m} \chi_i)(x) \leq \lim_{m \to \infty} m T_t(\chi_i)(x) = \lim_{m \to \infty} 0 = 0$, so that we can well define $T_t^iv(x)$ by putting it to equal $1$ for $\alpha_i(x) = 0$,  $i = 1,2$.
We shall show that for $w^n(x) = \min(w(x),n)$, we have
\begin{equation}
\label{equ-proof-lem-truncated-weight}
T_tw(x) T_t(w^{-1})(x) - T_t(w^n)(x)T_t((w^n)^{-1})(x) \geq 0.
\end{equation}
Indeed, we have $T_t^2w(x) T_t^2(w^{-1})(x) \geq 1$ when $\alpha_2(x) \neq 0$ according to Lemma \ref{lem-semigroup-CSU}.
Then
\begin{align*}
& T_tw(x)T_t(w^{-1})(x) - T_t(w^n)(x)T_t((w^n)^{-1})(x) \\
& = [\alpha_1(x) T_t^1w(x) + \alpha_2(x) T_t^2w(x)] [\alpha_1(x) T_t^1 (w^{-1})(x) + \alpha_2(x) T_t^2 (w^{-1})(x)] \\
& - [\alpha_1(x) T_t^1(w^n)(x) + \alpha_2(x) T_t^2(w^n)(x)] [\alpha_1(x) T_t^1 ((w^n)^{-1})(x) + \alpha_2(x) T_t^2 ((w_n)^{-1})(x)] \\
& = \alpha_1^2(x) [T_t^1w(x) T_t^1(w^{-1})(x) - T_t^1 w^n(x) T_t^1((w^n)^{-1})(x)] \\
& + \alpha_1(x)\alpha_2(x) [T_t^1w(x) T_t^2(w^{-1})(x) + T_t^2w(x) T_t^1(w^{-1})(x) - T_t^1 (w^n)(x) T_t^2((w^n)^{-1})(x) - T_t^2(w^n)(x) T_t^1((w^n)^{-1})(x)] \\
& + \alpha_2^2(x) [T_t^2w(x)T_t^2 (w^{-1})(x) - 1 ]\\
& = \alpha_1^2(x) [T_t^1w(x) T_t^1(w^{-1})(x) - T_t^1 w(x) T_t^1(w^{-1})(x)] \\
& + \alpha_1(x)\alpha_2(x) [T_t^1w(x) T_t^2(w^{-1})(x) + T_t^2w(x) T_t^1(w^{-1})(x) - T_t^2(\frac{1}{n}) T_t^1 w(x) - T_t^2(n) T_t^1(w^{-1})(x)] \\
& + \alpha_2^2(x) [T_t^2w(x)T_t^2 (w^{-1})(x) - 1 ]\\
& \geq \alpha_1(x)\alpha_2(x) [ T_t^1w(x) T_t^2(w^{-1} - \frac1n)(x)  + T_t^1(w^{-1})(x) T_t^2(w-n)(x) ] + \alpha_2^2(x) \cdot 0 \\
& = \alpha_1(x)\alpha_2(x) T_t^2 \left( (w^{-1}(\cdot) - \frac1n)T_t^1w(x) + (w(\cdot)-n) T_t^1(w^{-1})(x) \right)(x) \\
& = \alpha_1(x)\alpha_2(x) T_t^2 \left( \frac{w-n}{wn}(-\underbrace{T_t^1w(x)}_{\leq n} + wn \underbrace{T_t^1(w^{-1})(x)}_{\geq \frac1n}) \right)(x) \\
& \geq \alpha_1(x)\alpha_2(x) T_t^2 \left( \frac{w-n}{wn} (w-n) \right) \geq 0.
\end{align*}
We have thus shown \eqref{equ-proof-lem-truncated-weight}, which implies
$Q^A_2(w) \geq Q^A_2(w^n)$.
Now observe that $w_n = (((w^n)^{-1})^n)^{-1}$ and that $Q^A_2(v) = Q^A_2(v^{-1})$, so that
$Q^A_2(w) \geq Q^A_2(w^n) = Q^A_2((w^n)^{-1}) \geq Q^A_2(((w^n)^{-1})^n) = Q^A_2(w_n)$.
\end{proof}

\begin{remark}
\label{remark-classical-characteristic}
Suppose that $(\Omega,\dist,\mu)$ is a space of homogeneous type, i.e. a metric measure space such that $\mu(B(x,2r)) \leq C \mu(B(x,r))$ with uniform constant $C$ over all $x \in \Omega$ and $r > 0$, where $B(x,r) = \{ y \in \Omega : \: \dist(x,y) \leq r \}$ denotes a ball in $\Omega$.
This entails that $\Omega$ admits a doubling dimension $d \geq 0$ such that $\mu(B(x,\alpha r)) \leq C \alpha^d \mu(B(x,r))$ for $x \in \Omega$, $r > 0$ and $\alpha \geq 1$.
In the following cases, the semigroup characteristic can be compared to the classical characteristic defined in terms of means over balls, that is
\[ Q^{class}_2(w) = \sup_{B \text{ ball in }\Omega} \frac{1}{\mu(B)} \int_B w(y)d\mu(y) \frac{1}{\mu(B)} \int_B \frac{1}{w(y)}d\mu(y) . \]
Let $(T_t)_{t \geq 0}$ be a submarkovian semigroup acting on $L^2(\Omega)$.
Suppose that $T_t$ has an integral kernel $p_t(x,y)$.
\begin{enumerate}
\item If $p_t(x,y)$ satisfies (one-sided) Gaussian estimates \cite[(1.3)]{GrTe}: there exist $C,C_+ > 0$ such that
\begin{equation}
\label{equ-GE}
 p_t(x,y) \leq C \frac{1}{\mu(B(x,\sqrt{t}))} \exp(-C_+ \dist(x,y)^2 / t) \quad (t > 0, \: x , y \in \Omega),
\end{equation}
then there exists some $c < \infty$ such that for any weight $w : \Omega \to \R_+$, we have $Q^A_2(w) \leq c Q^{class}_2(w).$
\item If $p_t(x,y)$ satisfies lower Gaussian estimates: there exist $c,C_-  >0$ such that
\[ p_t(x,y) \geq c \frac{1}{\mu(B(x,\sqrt{t}))} \exp(-C_- \dist(x,y)^2 / t) \quad (t > 0, \: x , y \in \Omega),\]
then there exists some $c < \infty$ such that for any weight $w : \Omega \to \R_+$, we have $Q^{class}_2(w) \leq c Q^A_2(w).$
\end{enumerate}
Consequently, if the semigroup satisfies two-sided Gaussian estimates, then $Q^{class}_2(w) \cong Q^A_2(w)$.
\end{remark}

\begin{proof}
1. Let $x \in \Omega, \: t > 0$ and for $k \geq 0$, denote the ball centered at $x$, $B_k = B(x,2^{k+1}\sqrt{t})$ and also the annulus $A_k = B_k \backslash B_{k-1}$ ($A_k = B_k$ if $k = 0$).
Then with $d$ denoting a doubling dimension of $\Omega$, so that $\mu(B(x,2^{k+1}r)) \lesssim 2^{(k+1)d}\mu(B(x,r))$, we have
\begin{align*}
& T_tw(x) T_t(w^{-1})(x) = \int_\Omega p_t(x,y)w(y)d\mu(y) \int_\Omega p_t(x,y) w^{-1}(y) d\mu(y) \\
& = \sum_{k,l = 0}^\infty \int_{A_k} p_t(x,y) w(y)d\mu(y) \int_{A_l} p_t(x,y) w^{-1}(y) d\mu(y) \\
& \lesssim \sum_{k,l = 0}^\infty \int_{A_k} \frac{1}{\mu(B(x,\sqrt{t}))} \exp(-C_+\dist(x,y)^2/t) w(y)d\mu(y) \int_{A_l} \frac{1}{\mu(B(x,\sqrt{t}))} \exp(-C_+\dist(x,y)^2/t) w^{-1}(y)d\mu(y) \\
& \lesssim \sum_{k,l = 0}^\infty \int_{A_k} 2^{(k+1)d} \frac{1}{\mu(B(x,2^{k+1}\sqrt{t}))} \exp(-C_+ 2^{2k}) w(y) d\mu(y) \\
& \times \int_{A_l} 2^{(l+1)d} \frac{1}{\mu(B(x,2^{l+1}\sqrt{t}))} \exp(-C_+ 2^{2l}) w^{-1}(y) d\mu(y) \\
& \lesssim \sum_{k,l = 0}^\infty 2^{(k+1)d} \exp(-C_+ 2^{2k}) \frac{1}{\mu(B_k)} \int_{B_k} w(y) d\mu(y) 2^{(l+1)d} \exp(-C_+ 2^{2l}) \frac{1}{\mu(B_l)} \int_{B_l} w^{-1}(y) d\mu(y) \\
& \leq \sum_{k,l = 0}^\infty 2^{(k+l+2)d} \exp(-C_+ (2^{2k} + 2^{2l})) 2^{(2\max(k,l) - k-l)d} \frac{1}{\mu(B_{\max(k,l)})} \int_{B_{\max(k,l)}} w(y) d\mu(y) \\
& \times \frac{1}{\mu(B_{\max(k,l)})} \int_{B_{\max(k,l)}} w^{-1}(y) d\mu(y) \\
& \leq \sum_{k,l = 0}^\infty 2^{(k+l+2)d} \exp(-C_+(2^{2k} + 2^{2l})) 2^{(2\max(k,l)-k-l)d} Q^{class}_2(w) \\
& = c Q^{class}_2(w),
\end{align*}
where we note in the last step that the double series is clearly convergent. \\

2. We have
\begin{align*}
& T_tw(x) T_t(w^{-1})(x) \gtrsim \int_\Omega \frac{1}{\mu(B(x,\sqrt{t}))}\exp(-C_- \dist(x,y)^2/t) w(y)d\mu(y) \\
& \times
\int_\Omega \frac{1}{\mu(B(x,\sqrt{t}))}\exp(-C_- \dist(x,y)^2/t) w^{-1}(y)d\mu(y) \\
& \geq \exp(-C_-) \frac{1}{\mu(B(x,\sqrt{t}))} \int_{B(x,\sqrt{t})} w(y)d\mu(y) \exp(-C_-) \frac{1}{\mu(B(x,\sqrt{t}))} \int_{B(x,\sqrt{t})} w^{-1}(y)d\mu(y).
\end{align*}
Taking the supremum over all $t > 0$ and $x \in \Omega$ yields $Q^A_2(w) \gtrsim Q^{class}_2(w)$.
\end{proof}

\section{The Bellman function and its main properties}
\label{sec-Bellman}

For the proof of the Main Theorem \ref{thm-bilinear} on weighted $L^2$ functional calculus, we need the existence of a Bellman function from \cite{DoPe}.
Let us note $V$ the quadruplet
\[ V = (x, y, r, s) \in \mathbb{C} \times \mathbb{C} \times
   \mathbb{R}^{\ast}_+ \times \mathbb{R}^{\ast}_+ = : \mathbb{S}. \]
The variables $(x, y)$ will be $\mathbb{C}$--valued whereas the variables $r$
and $s$ are $\mathbb{R}$--valued and represent the weights. We introduce the
domain $\mathcal{D}_Q$
\[ \mathcal{D}_Q = \left\{ V \in \mathbb{S}: \hspace{1em} 1 \leqslant r
   s \leqslant Q \right\} . \]
We will often restrict our attention to truncated weights, that is given $0 <
\varepsilon < 1$, variables $r$ and $s$ bounded below and above
\[ \mathcal{D}^{\varepsilon}_Q = \left\{ V \in \mathcal{D}_Q :
   \hspace{1em} \varepsilon \leqslant r \leqslant \varepsilon^{- 1},
   \hspace{1em} \varepsilon \leqslant s \leqslant \varepsilon^{- 1} \right\} .
\]
Let $| \cdot |$ denote the standard norm in the complex plane and denote for
$x = x_1 + i x_2$ the complex derivative $\partial_x = \partial_{x_1} - i
\partial_{x_2}$.

\begin{lemma}[existence and properties of the Bellman function]
  \label{L: existence and properties of the Bellman function}There exists a
  function $B (V) = B_Q$ that is $\mathcal{C}^1$ on
  $\mathcal{D}^{\varepsilon}_Q$, and piecewise $\mathcal{C}^2$, with the
  estimate
  \begin{equation}
    B (V) \lesssim \frac{| x |^2}{r} + \frac{| y |^2}{s}
    \label{inequality-size}
  \end{equation}
  and on each subdomain where it is $\mathcal{C}^2$ there holds
  \begin{equation}
    d^2 B (V) \geqslant \frac{2}{Q} | d x | | d y | .
     \label{inequality-convexity}
  \end{equation}
  Whenever $V$ and $V_0$ are in the domain, the function has the property
  \begin{equation}
    B (V) - B (V_0) - d B (V_0) (V - V_0) \geqslant \frac{2}{Q} | x - x_0
    |  | y - y_0 | . \label{inequality-one leg convexity}
  \end{equation}
  Moreover, we have 
   \begin{equation}
    \partial_r B (V)\le 0 {\text{ and }}  \partial_s B (V) \le 0
 \label{derivative-sign},
  \end{equation}
and  the estimate
  \begin{equation}
    \Re [\partial_x B (V) x + \partial_y B (V) y + \partial_r B (V) r +
    \partial_s B (V) s] \gtrsim \frac{1}{Q} | x |  | y |
    \label{estimate-error} .
  \end{equation}
\end{lemma}

\begin{proof}
We shall use the function constructed in \cite{DoPe}, composed of linear combinations
of functions $B_1$ through $B_6$. The final function in \cite{DoPe} is of the form
$\sum^6_{i = 1} C_i B_i$ for some positive constants $C_i$ and for some
functions $B_i$ written below. The function $B$ here will have $C_1$ replaced
by $C > C_1$ to be determined. In order to have the last property
\eqref{estimate-error} we will increase the coefficient of $B_1$. Since $B_1$ is
convex and has the desired upper bounds, this will not change the assertions
made in \cite{DoPe} concerning the other properties of the function.

We recall that
\[ B_1 (V) = \frac{| x |^2}{r} + \frac{| y |^2}{s} \]
and calculate $\partial_x B_1 (V) x = \frac{2}{r} | x |^2$, $\partial_y B_1
(V) y = \frac{2}{s} | y |^2$, $\partial_r B_1 (V) r = - \frac{1}{r} | x |^2$
and $\partial_s B_1 (V) s = - \frac{1}{s} | y |^2$. Thus
\[ \Re [\partial_x B_1 (V) x + \partial_y B_1 (V) y + \partial_r B_1 (V)
   r + \partial_s B_1 (V) s] = \frac{| x |^2}{r} + \frac{| y |^2}{s} \geqslant
   2 \frac{| x |  | y |}{r s} \geqslant \frac{2 | x |  | y |}{Q} . \]
This is the main term that gives us the desired estimate. We check now that
the remaining parts of the function's derivatives stemming from $B_2$ through
$B_6$ are not too large.

We recall that
\[ B_2 (V) = \frac{| x |^2}{2 r - \frac{1}{s (N + 1)}} + \frac{| y |^2}{s}
   { \text{ and } } B_3 (V) = \frac{| x |^2}{r} + \frac{| y |^2}{2 s - \frac{1}{r (N
   + 1)}} \]
where
\[ N = N (r, s) = \frac{\sqrt{r s}}{\sqrt{Q}} \left( 1 - \frac{(r s)^{3 /
   2}}{128 Q^{3 / 2}} \right) = \frac{\sqrt{r s}}{\sqrt{Q}} - \frac{(r
   s)^2}{128 Q^2} . \]
Since $0 \leqslant N \leqslant 1$ and $0 \leqslant (r s - 1) \frac{1}{s} = r -
\frac{1}{s} \leqslant r$ we have $r \leqslant 2 r - \frac{1}{s (N + 1)}
\leqslant 2 r$. Thus
\[ \partial_x B_2 (V) x = \frac{2}{2 r - \frac{1}{s (N + 1)}} | x |^2
   \geqslant \frac{| x |^2}{r} . \]
As above $\partial_y B_2 (V) y = \frac{2}{s} | y |^2$. We calculate the
derivatives in $r$ and $s$. First, we observe that
\[ \partial_r N = \frac{1}{r} \frac{\sqrt{r s}}{\sqrt{Q}} \left( \frac{1}{2}
   - \frac{(r s)^{3 / 2}}{64 Q^{3 / 2}} \right) { \text{ and }} \partial_s N =
   \frac{1}{s} \frac{\sqrt{r s}}{\sqrt{Q}} \left( \frac{1}{2} - \frac{(r s)^{3
   / 2}}{64 Q^{3 / 2}} \right) \]
thus $0 \leqslant \partial_r N \leqslant \frac{1}{2 r}$ and $0 \leqslant
\partial_s N \leqslant \frac{1}{2 s}$. We calculate
\[ \partial_r B_2 (V) = - \frac{| x |^2 \left( 2 + \frac{1}{s (N + 1)^2}
   \partial_r N \right)}{\left( 2 r - \frac{1}{s (N + 1)} \right)^2}
   {\text{ and }} \partial_s B_2 (V) = - \frac{| y |^2}{s^2} - \frac{| x |^2 (N +
   1 + s \partial_s N)}{(2 r s (N + 1) - 1)^2} \]
thus $- \partial_r B_2 (V) r \leqslant \frac{| x |^2}{4 r} \frac{5}{2}$ and
with $1 \leqslant r s$ we have $- \partial_s B_2 (V) s \leqslant \frac{| y
|^2}{s} + \frac{| x |^2}{r} \frac{5}{2}$. Therefore
\[ \Re [\partial_x B_2 (V) x + \partial_y B_2 (V) y + \partial_r B_2
   (V) r + \partial_s B_2 (V) s] \geqslant \left( 1 - \frac{25}{8} \right)
   \frac{| x |^2}{r} + \frac{| y |^2}{s} \]
and similarly
\[ \Re [\partial_x B_3 (V) x + \partial_y B_3 (V) y + \partial_r B_3 (V)
   r + \partial_s B_3 (V) s] \geqslant \frac{| x |^2}{r} + \left( 1 -
   \frac{25}{8} \right) \frac{| y |^2}{s} . \]
We turn to $B_4$. Recall that
\[ H_4 (x, y, r, s, K) = \sup_{\alpha > 0} \beta (\alpha, x, y, r, s, K) =
   \sup_{\alpha > 0} \left( \frac{| x |^2}{r + \alpha K} + \frac{| y |^2}{s +
   \alpha^{- 1} K} \right) \]
and $B_4 (x, y, r, s) = H_4 (x, y, r, s, K (r, s))$ with
\[ K (r, s) = \frac{\sqrt{r s}}{\sqrt{Q}} \left( 1 - \frac{\sqrt{r s}}{8
   \sqrt{Q}} \right) . \]
Since $H_4$ is in $C^1$ we investigate the inequality on each piece
separately, whether the supremum is attained in $\alpha = 0 , \infty$
or $\alpha$ finite. If the supremum occurs at the boundary, the function $H_4$
simplifies to the expression $\frac{| x |^2}{r}$ or $\frac{| y |^2}{s}$ so in
this case there is nothing more to estimate. Let us assume the supremum is a
maximum and is attained at a finite strictly positive $\alpha' = \alpha (x, y,
r, s, K)$ so that $H_4 (x, y, r, s, K) = \beta (\alpha' (x, y, r, s, K), x, y,
r, s, K)$. Recall this happens if and only if $| x | s - | y | K > 0$ and $| y
| r - | x | K > 0$. Since $\partial_{\alpha} \beta (\alpha', x, y, r, s, K) =
0$ we obtain
\[ \partial_x B_4 (V) x = \frac{| x |^2}{r + \alpha' K} \geqslant 0 {\text{ and }}
   \partial_y B_4 (V) y = \frac{| y |^2}{s + \alpha'^{- 1} K} \geqslant 0. \]
We prefer not to take advantage of the subtle gain from these derivatives as
it is more easily had from $B_1$. Now we control the damage from the other
derivatives. To do so, let us resort to the explicit expression
\[ H_4 (x, y, r, s, K) = \frac{| x |^2 s - 2 | x |  | y | K + | y |^2 r}{r s
   - K^2} . \]
We get
\[ \partial_r H_4 (V, K) = - \frac{(| x | s - | y | K)^2}{(r s - K^2)^2}
   {\text{ and }} \partial_s H_4 (V, K) = - \frac{(| y | r - | x | K)^2}{(r s -
   K^2)^2} . \]
Since $K^2 \leqslant \frac{r s}{Q}$ we get $r s - K^2 \geqslant \left( 1 -
\frac{1}{Q} \right) r s$. We estimate $| \partial_r H_4 | \leqslant \frac{| x
|^2}{4 r^2}$ and $| \partial_s H_4 | \leqslant \frac{| y |^2}{4 s^2}$ for $Q$
large enough. Now
\[ \partial_K H_4 (V, K) = \frac{| x | s (2 K | x | - | y | r) + | y | r (2 K
   | y | - | x | s)}{(r s - K^2)^2} .\]
Let us estimate $| \partial_K H_4 | \leqslant 24 \frac{| x |  | y | r s}{(r
s)^2} = \frac{24 | x |  | y |}{r s}$. In order to deduce the remaining
derivative estimates for $B_4$ observe that
\[ \partial_r K = \frac{1}{r} \frac{\sqrt{r s}}{\sqrt{Q}} \left( \frac{1}{2}
   - \frac{\sqrt{r s}}{8 \sqrt{Q}} \right) {\text{ and }} \partial_s K =
   \frac{1}{s} \frac{\sqrt{r s}}{\sqrt{Q}} \left( \frac{1}{2} - \frac{\sqrt{r
   s}}{8 \sqrt{Q}} \right) \]
with $\frac{1}{8 r} \leqslant \partial_r K \leqslant \frac{1}{2 r}$ and \
$\frac{1}{8 s} \leqslant \partial_s K \leqslant \frac{1}{2 s}$ for large
enough $Q$. Thus
\[ | \partial_r B_4 (V) | r \leqslant \frac{| x |^2}{4 r} + \frac{24 | x |  |
   y |}{2 r s} {\text{ and }} | \partial_s B_4 (V) | s \leqslant \frac{| y |^2}{4
   s} + \frac{24 | x |  | y |}{2 r s} . \]
Now
\[ B_5 (V) = \frac{| x |^2}{2 r - \frac{1}{s (K + 1)}} + \frac{| y |^2}{s}
   {\text{ and }} B_6 (V) = \frac{| x |^2}{r} + \frac{| y |^2}{2 s - \frac{1}{r (K
   + 1)}} \]
and since $r \leqslant 2 r - \frac{1}{s (K + 1)} \leqslant 2 r$ we get
\[ \partial_x B_5 (V) x = \frac{2 | x |^2}{2 r - \frac{1}{s (K + 1)}}
   \geqslant \frac{| x |^2}{r} {\text{ and }} \partial_y B_5 (V) y = \frac{2 | y
   |^2}{s} . \]
Further $- \partial_r B_5 (V) r \leqslant \frac{| x |^2}{4 r} \frac{5}{2}$
and, using $1 \leqslant r s$ we get $- \partial_s B_5 (V) s \leqslant \frac{|
x |^2}{4 r} \frac{5}{2} + \frac{| y |^2}{s}$. Similarly, estimates hold for
$B_6$.

Summarizing, we see that none of the arising terms from $B_2$ through $B_6$
exceed $C \left( \frac{| x |^2}{r} + \frac{| y |^2}{s} \right)$ for a suitably
large $C$. Thus, a large enough coefficient before $B_1$ will arrange for us
the desired property in our function $B$. 

Now notice that the proof of \eqref{derivative-sign} is implicit in the considerations we just carried out and follow as an easy calculation, treating the pieces of $B$ separately. The least obvious estimate is that of $\partial_r B_4$ and $\partial_s B_4$ for the case when the extremum is attained for $0<\alpha < \infty$. For example, $\partial_r B_4(V)=\partial_r H_4(V,K) + \partial_K H_4(V,K) \partial_r K$. Here, we see from the considerations above that $\partial_r H_4(V,K)\le 0$ as well as $\partial_K H_4(V,K)\le 0$ and $\partial_r K\ge 0$, giving us the desired sign.
\end{proof}

Similarly to arguments detailed in \cite{DoPe}, 
we can obtain a regularised version of this Bellman function so that its main properties remain true:

\begin{lemma}[regularised Bellman function and its properties]
  \label{L: regularised Bellman function}Let $\varepsilon > 0$ given. Let $0 <
  \ell \leqslant \varepsilon / 2$. There exists a function $B_{\ell} (x, y, r,
  s)$ defined with domain
  \[ \mathcal{D}_Q^{\varepsilon, \ell} = \left\{ V \in
     \mathcal{D}_Q^{\varepsilon} ; \hspace{1em} | x | \geqslant \ell, | y |
     \geqslant \ell \right\} \subset \mathcal{D}_Q^{\varepsilon} \]
  such that for all $V_0, V \in \mathcal{D}_Q^{\varepsilon, \ell}$, we have
  \[ \hspace{1em} B_{\ell} (V) \lesssim (1 + \ell) \left( \frac{| x |^2}{r} +
     \frac{| y |^2}{s} \right), \]
  \begin{equation}
    \d_V^2 B_{\ell} (V) \geqslant \frac{2}{Q}  | d x | | d y |
    , \label{eq: concavity continuous}
  \end{equation}
  \begin{equation}
    B_{\ell} (V) - B_{\ell} (V_0) - d_V B_{\ell} (V_0) (V - V_0)
    \geqslant \frac{1}{Q} | \Delta x |  | \Delta y | = \frac{1}{Q} | x - x_0 |
    | y - y_0 | \label{eq: concavity discrete}.
  \end{equation}
  Further we have the estimates
  \[ \partial_r B_{\ell} (V)\le 0 {\text{ and }} \partial_s B_{\ell} (V)\le 0 \]
  and
  \[ \Re [\partial_x B_{\ell} (V) x + \partial_y B_{\ell} (V) y +
     \partial_r B_{\ell} (V) r + \partial_s B_{\ell} (V) s] \gtrsim
     \frac{1}{Q} | x |  | y | . \]
\end{lemma}

In addition, we have the following specific properties we formulate separately as lemmata.

%
%
%
\begin{lemma}
\label{lem-partial5-L1}
Let $Q > 1, \: \varepsilon > 0$
and $B$ be the Bellman function from Lemma \ref{L: existence and properties of the Bellman function} with domain $\mathcal{D}^{\varepsilon}_Q$.
Let $f,g \in L^1(\Omega) \cap L^\infty(\Omega)$ and $v_1,v_2 : \Omega \to \R$ measurable such that $1 \leq v_1 v_2 \leq Q$ and $\varepsilon \leq v_1,v_2 \leq \frac{1}{\varepsilon}$.
Then $\partial_r B(f,g,v_1,v_2), \partial_s B(f,g,v_1,v_2)$ belong to $L^1(\Omega)$ with $L^1$-norm bounded by $\|f\|_{L^2(\Omega)}^2 + \|g\|_{L^2(\Omega)}^2$ (times a constant depending on $\varepsilon$).
Moreover, $\partial_r B(f,g,v_1,v_2)$ and $\partial_s B(f,g,v_1,v_2)$ belong to $L^\infty(\Omega)$.
\end{lemma}

\begin{proof}
This follows via the calculations above. Let us treat the part $\partial_r B_4$ as an example. Above we have estimated $|\partial_r B_4 (V)| \le \frac{|x|^2}{4r^2}+\frac{24|x||y|}{2r^2s}$. Plugging in $f,g,v_1,v_2$ and integrating over $\Omega$, then using elementary estimates and the range of $v_1$ and $v_2$ we obtain an estimate $\|\partial_r B_4(f,g,v_1,v_2)\|_{L^1(\Omega)} \le c(\epsilon) (\|f\|_{L^2(\Omega)}^2+ \|g\|_{L^2(\Omega)}^2)$.
Similarly, we obtain $\partial_r B(f,g,v_1,v_2)$,  $\partial_s B(f,g,v_1,v_2) \in L^\infty(\Omega)$.
\end{proof}

\begin{lemma}
\label{lem-partial1-L2}
Let $Q > 1, \: \varepsilon > 0$
and $B$ be the Bellman function from Lemma \ref{L: existence and properties of the Bellman function} with domain $\mathcal{D}^{\varepsilon}_Q$.
Let $f,g \in L^1(\Omega) \cap L^\infty(\Omega)$ and $v_1,v_2 : \Omega \to \R$ measurable such that $1 \leq v_1 v_2 \leq Q$ and $\epsilon \leq v_1,v_2 \leq \frac{1}{\epsilon}$.
Then $\partial_k B(f,g,v_1,v_2)$ belongs to $L^2(\Omega)$ for $k = x_1,x_2,y_1,y_2$
with an estimate $\|\partial_k B(f,g,v_1,v_2) \|_2^2 \leq c(\varepsilon) (\|f\|_2^2 + \|g\|_2^2)$.
\end{lemma}

\begin{proof}
One shows this assertion for each part of $B$ separately. Since the different first derivatives in the real and imaginary parts of $x$ and $y$ are similar, we focus on $x_1$. We show the assertion for $B_4$ in the part of the domain where the supremum is obtained for $0<\alpha' <\infty$ since the other parts are easier. Recall that here $B_4(V)=H_4(x,y,r,s,K(r,s))$ with $H_4(x,y,r,s,K)=\beta(\alpha'(x,y,r,s,K),x,y,r,s,K)$ with 
$\beta(\alpha,x,y,r,s,K)=\frac{|x|^2}{r+\alpha K}+\frac{|y|^2}{s+\alpha^{-1} K}$. Note that $\partial_{\alpha} \beta$ at $\alpha'$ is zero because we have a critical point there. Thus, $|\partial_{x_1} B_4(V)|^2\le\frac{4|x_1|^2}{r^2}$, where we estimated the denominator $r+\alpha'(x,y,r,s,K(r,s)),x,y,r,s,K(r,s)\ge r$. Plugging in the given functions for the variables while using the range of $v_1$ and $v_2$, then integrating over $\Omega$ gives the desired result.
\end{proof}

\begin{lemma}
\label{lem-B-Lipschitz}
Let $B$ be the Bellman function from Lemma \ref{L: existence and properties of the Bellman function} with domain $\mathcal{D}^{\varepsilon}_Q$.
For $C > 0$ there exists $L = L(C,\varepsilon,Q) > 0$ with the following property.
Assume $x,y \in \R^2$ with $|x|,|y| \leq C$.
Moreover assume $1 \leq r_js_j \leq Q$ and $\varepsilon \leq r_j,s_j \leq \frac{1}{\varepsilon}$ for $j = 1,2$ such that any $(r,s)$ belonging to the line segment connecting $(r_1,s_1)$ and $(r_2,s_2)$ still satisfies $1 \leq rs \leq Q$.
Then
\begin{align*}
\MoveEqLeft
\left| \partial_rB(x,y,r_1,s_1) - \partial_rB(x,y,r_2,s_2) \right| , \left| \partial_s B(x,y,r_1,s_1) - \partial_s B(x,y,r_2,s_2) \right| \\
& \leq L (|x|^2 + |y|^2) (|r_1 - r_2| + |s_1 - s_2|) .
\end{align*}
\end{lemma}

\begin{proof}
We split the Bellman function into its six parts (see the proof of Lemma \ref{L: existence and properties of the Bellman function}) and majorize each of them.
For the first part, we have $B_1$ a $C^\infty$ function on $\mathcal{D}^{\varepsilon}_Q$.
Moreover, it is clearly of the form $B_1(x,y,r,s) = |x|^2 B_{1,1}(r,s) + |y|^2 B_{1,2}(r,s)$.
Thus, also $\partial_rB_1(x,y,r,s) = |x|^2 \partial_r B_{1,1}(r,s) + |y|^2 \partial_r B_{1,2}(r,s)$ and similar for $\partial_sB_1$.
Now we conclude for $B_1$ by the mean value theorem (note that we have assumed that the line segment between $(r_1,s_1)$ and $(r_2,s_2)$ lies in the domain $\mathcal{D}^{\varepsilon}_Q$), the Lipschitz constant popping up being clearly majorized by $L (|x|^2 + |y|^2)$.
The same argument works for the parts $B_2,B_3,B_5$ and $B_6$.

We turn to the most technical part $B_4$.
It is given by one of the three expressions $\frac{|x|^2}{r}$ (if $|y| r - |x| K > 0$ and $|x| s - |y| K \leq 0$), or $\frac{|y|^2}{s}$ (if $|y|r - |x| K \leq 0$ and $|x|s - |y| K > 0$) or $B_4(x,y,r,s) = H_4(x,y,r,s,K) = \frac{|x|^2 s - 2 |x| \, |y| K + |y|^2 r}{rs - K^2}$ (if both $|y|r - |x| K > 0$ and $|x|s - |y|K > 0$).
Indeed, the remaining case, both $|y|r - |x|K \leq 0$ and $|x|s - |y|K \leq 0$ can only occur if both $|x| = |y| = 0$ (see \cite{DoPe}) in which case $B_4(x,y,r,s) = 0$ for all $r,s$, so that $\partial_rB_4(x,y,r,s) = \partial_sB_4(x,y,r,s) = 0$ for all $r,s$ and the claimed Lipschitz estimates for $\partial_rB_4$ and $\partial_sB_4$ are trivially true.

We claim that the signs of the expressions $|x|s-|y|K(r,s)$ and $|y|r - |x|K(r,s)$ which determine the formula giving $B_4$ cannot change too often along the segment connecting $(r_1,s_1)$ and $(r_2,s_2)$.
Indeed, call $f(r,s) = |x|s-|y|K(r,s)$, $r(t) = r_1 + (r_2-r_1)t$, $s(t) = s_1 + (s_2 - s_1)t$ and $g(t) = f(r(t),s(t))$.
Then a longer yet elementary calculation reveals that $\frac{d^3}{dt^3}g(t) = \frac{1}{[r(t)s(t)]^{\frac52}} p(t)$, where $p$ is a polynomial of degree at most $3$.
Thus, $\frac{d^3}{dt^3}g(t)$ has at most $3$ zeros and by Rolle's theorem, $g(t)$ itself has at most $6$ zeros.
The same reasoning applies to show that also $|y|r - |x|K(r,s)$ has at most $6$ zeros.
Thus we can cut the segment connecting $(r_1,s_1)$ and $(r_2,r_2)$ into at most $6 \times 6 = 36$ subsegments such that on the interior of each of them, $|x|s-|y|K$ and $|y|r-|x|K$ keep the same sign.
We infer that on these subsegments, $B_4(x,y,r,s)$ is given by one of the three expressions $\frac{|x|^2}{r},\frac{|y|^2}{s}$ or $H_4$ in a constant manner.
Using the fact that these three expressions are $C^2$, that their second derivatives in $r,s$ are bounded by $C(\varepsilon,Q) (|x|^2 + |y|^2)$ and that $\partial_rB_4,\partial_sB_4$ as a whole are continuous, we can argue as in the case of $B_1$ above to deduce the Lipschitz estimates for $\partial_rB_4(x,y,r,s),\partial_sB_4(x,y,r,s)$ as claimed in the lemma, where the points $(r_1,s_1),(r_2,s_2)$ are replaced by the boundary points of the subsegments.
Adding the estimates for the at most $36$ subsegments yields the Lipschitz estimate for $B_4$ with the correct boundary points $(r_1,s_1),(r_2,s_2)$.
\end{proof}

\section{The positive $\frac{\pi}{2}$ angle result, markovian case}
\label{sec-bilinear}

This section is devoted to the proof of Theorem \ref{thm-1-intro} (see Theorem \ref{thm-bilinear} and Corollary \ref{cor-bilinear}) and Proposition \ref{prop-3-intro} (see Proposition \ref{prop-maximal-regularity}).
In the proof of the Main Theorem \ref{thm-bilinear} of this section, the functional $\mathcal{E}$ from the following Lemma \ref{lem-E-differentiable} involving the Bellman function from Section \ref{sec-Bellman} plays an important r\^{o}le.
After that, we shall define a technical condition assumed in the Main Theorem, that is the local diffusion (Definition \ref{defi-local-diffusion}).
Then we state the main results, Theorem \ref{thm-bilinear} and Corollary \ref{cor-bilinear} on weighted $L^2$ functional calculus.
In the subsequent remarks, we shall discuss the validity of their hypotheses in important examples.
Afterwards, we prove Theorem \ref{thm-bilinear}, where the crucial estimate on the derivative of the functional $\mathcal{E}$ is split into several Propositions and Lemmas.
Finally, we state and prove Proposition \ref{prop-maximal-regularity} on the angle reduction $\theta < \frac{\pi}{2}$ and compare in Remark \ref{rem-DSY-GY} our result with the literature \cite{DSY,GY}.

Recall for a markovian semigroup the domain $D(A_\infty) = \{ f \in L^\infty(\Omega) :\: w^* - \lim_{h \to 0+} \frac1h (f - T_hf) \text{ exists} \}$, and the derivatives $\partial_x = \partial_{x_1} - i \partial_{x_2}$ and $\partial_y = \partial_{y_1} - i \partial_{y_2}$.

\begin{lemma}
\label{lem-E-differentiable}
Let $B$ be the Bellman function from Lemma \ref{L: existence and properties of the Bellman function} with domain $\mathcal{D}^{\varepsilon}_Q$.
Let $(\Omega,\mu)$ be a $\sigma$-finite measure space, $(T_t)_{t \geq 0}$ a markovian semigroup on $\Omega$ and $f,g \in L^1(\Omega) \cap L^\infty(\Omega)$.
Let further $v,w \in L^\infty(\Omega)$ with $\varepsilon \leq v,w \leq \frac{1}{\varepsilon}$ and $2 \leq T_tv T_tw \leq Q/2$ for all $t$ belonging to some interval $(t_0,t_1) \subseteq (0,\infty)$.
Assume one of the following conditions.
\begin{enumerate}
\item The measure space is finite, i.e. $\mu(\Omega) < \infty$, or
\item $T_tv,T_tw$ belong to $D(A_\infty)$ for $t \in (t_0,t_1)$, e.g. $v,w \in D(A_\infty)$.
\end{enumerate}
Define the functional
\begin{equation}
\label{equ-E-functional}
\mathcal{E}(t) = \int_\Omega B(T_tf(x) , T_tg(x), T_tv(x), T_t w(x)) d\mu(x) \quad (t \in (t_0,t_1)) .
\end{equation}
Then $\mathcal{E}(t)$ is differentiable for $t \in (t_0,t_1)$, and we have
\begin{align}
\label{equ-E-calculated}
- \mathcal{E}'(t)&  = \Re [ \int_\Omega \partial_x B(T_tf, T_tg, T_tv, T_tw)A T_tf + \partial_y B(T_tf, T_tg,T_tv, T_tw)A T_tg \\
& + \partial_r B(T_tf, T_tg, T_tv, T_tw)AT_tv + \partial_s B(T_tf, T_tg, T_tv, T_tw)AT_tw d \mu]. \nonumber
\end{align}
\end{lemma}

\begin{proof}
The proofs for the two alternative assumptions are different.
Let us start with the case $\mu(\Omega) < \infty$.
Since $B$ is a $\mathcal{C}^1$ function, we have
\begin{align}
\frac1h & (B(T_{t+h}f,T_{t+h}g,T_{t+h}v,T_{t+h}w) - B(T_tf,T_tg,T_tv,T_tw)) \label{equ-1-proof-E-differentiable}\\
& = \Re \left[dB(T_tf,T_tg,T_tv,T_tw) \cdot \frac1h (T_{t+h}-T_t)(f,g,v,w) \right] \nonumber \\
& + \frac1h [o(T_{t+h}f-T_tf) + o(T_{t+h}g-T_tg) + o(T_{t+h}v-T_tv) + o(T_{t+h}w-T_tw)] . \nonumber
\end{align}
Here, we write $dB = (\partial_x B,\partial_yB,\partial_r B,\partial_s B)$ with values in $\C \times \C \times \R \times \R$.
According to Lemmas \ref{lem-partial5-L1} and \ref{lem-partial1-L2}, all the six components of $dB(T_tf,T_tg,T_tv,T_tw)$ belong to $L^2(\Omega)$.
On the other hand $\frac1h (T_{t+h}-T_t)f$ converges to $-AT_tf$ in $L^2(\Omega)$, since $T_tf$ belongs to the domain of $A$ in $L^2(\Omega)$, and similarly for $g,v,w$ in place of $f$.
Here, we use the crucial fact that $v,w \in L^\infty(\Omega) \subseteq L^2(\Omega)$ for $\mu(\Omega) < \infty$.
Integrating in \eqref{equ-1-proof-E-differentiable} over $\Omega$ and letting $h \to 0$, it only remains to show that
\begin{equation}
\label{equ-2-proof-E-differentiable}
\int_\Omega \frac1h [o(T_{t+h}f-T_tf) + o(T_{t+h}g-T_tg) + o(T_{t+h}v-T_tv) + o(T_{t+h}w-T_tw)] d\mu \to 0 \quad (h \to 0+).
\end{equation}
We treat the component with $f$, the others being entirely similar.
Let $h_n$ be any sequence converging to $0+$.
Since $\int_\Omega | \frac1{h_n} (T_{t+h_n}f-T_tf) | d\mu \to \int_\Omega |AT_tf| d\mu < \infty$, we have that $\int_\Omega \frac1{h_n} |o(T_{t+h_n}f-T_tf)| d\mu$ is a bounded sequence.
We take a subsequence $h_k'=h_{n_k}$ such that $\int_\Omega \frac1{h_k'} |o(T_{t+h_k'}f-T_tf)| d\mu \to \limsup_n \int_\Omega \frac1{h_n} |o(T_{t+h_n}f-T_tf)| d\mu$.
As $\frac{1}{h_k'} (T_{t+h_k'}f-T_tf)$ converges in $L^1$, it converges along a subsequence $h_l''$ pointwise a.e. and $L^1$ dominated to the finite value $-AT_tf$.
Thus, $T_{t+h_l''}f-T_tf$ converges pointwise a.e. to $0$, and therefore, leveraging the $o$-notation, $\frac{1}{h_l''} o(T_{t+h_l''}f-T_tf)$ converges pointwise a.e. to $0$.
This latter convergence is also $L^1$ dominated by the above.
Thus, by dominated convergence, we infer that $\int_\Omega \frac{1}{h_l''} |o(T_{t+h_l''}f-T_tf)| d\mu \to 0 = \limsup_n \int_\Omega \frac1{h_n} |o(T_{t+h_n}f-T_tf)| d\mu$.
Then \eqref{equ-2-proof-E-differentiable} follows, and the Lemma is proved in the case $\mu(\Omega) < \infty$.

We turn to the second assumption $T_tv,T_tw \in D(A_\infty)$.
We develop
\begin{align}
\frac1h & (B(T_{t+h}f,T_{t+h}g,T_{t+h}v,T_{t+h}w) - B(T_tf,T_tg,T_tv,T_tw)) \label{equ-3-proof-E-differentiable}\\
& = \frac1h (B(T_{t+h}f,T_{t+h}g,T_{t+h}v,T_{t+h}w) - B(T_tf,T_tg,T_{t+h}v,T_{t+h}w)) \nonumber \\
& + \frac1h (B(T_tf,T_tg,T_{t+h}v,T_{t+h}w) - B(T_tf,T_tg,T_tv,T_tw)) \nonumber \\
& = \Re \left[ (\partial_x B, \partial_y B)(T_tf,T_tg,T_{t+h}v,T_{t+h}w) \cdot \frac1h (T_{t+h}-T_t)(f,g) \right] \nonumber \\
& + \frac1h[ o(T_{t+h}f-T_tf) + o(T_{t+h}g - T_tg) ] \nonumber \\
& + \frac1h (B(T_tf,T_tg,T_{t+h}v,T_{t+h}w)-B(T_tf,T_tg,T_tv,T_{t+h}w)) \nonumber \\
& + \frac1h (B(T_tf,T_tg,T_tv,T_{t+h}w)-B(T_tf,T_tg,T_tv,T_tw)) \nonumber.
\end{align}
Note that since $T_tv \in D(A_\infty)$, the $w^*$ convergent term $\frac{1}{h}(T_{t+h}v-T_tv)$ is bounded in $L^\infty(\Omega)$, so that $\|T_{t+h}v-T_tv\|_\infty \leq Ch$.
Therefore, since $2 \leq T_tv T_tw \leq Q/2,$ we have that for $h$ sufficiently small, $1 \leq T_tv T_{t+h}w \leq Q$, and therefore the Bellman function in \eqref{equ-3-proof-E-differentiable} is evaluated everywhere in points of its domain.
We treat the first term on the right hand side of \eqref{equ-3-proof-E-differentiable}.
Since $T_{t+h}v \to T_tv$ and $T_{t+h}w \to T_tw$ uniformly as $h \to 0+$ and $|\partial_k B(T_tf,T_tg,T_{t+h}v,T_{t+h}w)| \leq c(\varepsilon) (|T_tf| + |T_tg|)$ for $k = x_1,x_2,y_1,y_2$ according to the proof of Lemma \ref{lem-partial1-L2}, we can apply dominated convergence to deduce $\partial_k B(T_tf,T_tg,T_{t+h}v,T_{t+h}w) \to \partial_k B(T_tf,T_tg,T_tv,T_tw)$ in $L^2(\Omega)$, $k = x_1,x_2,y_1,y_2$.
Moreover, $\frac1h (T_{t+h}f - T_tf) \to -AT_tf$ and $\frac1h(T_{t+h}g - T_tg) \to -AT_tg$ in $L^2(\Omega)$.
The $o$-expressions can be treated as in the case $\mu(\Omega) < \infty$.
For the moment, we have shown
\begin{align}
\label{equ-4-proof-E-differentiable}
\int_\Omega & \frac1h (B(T_{t+h}f,T_{t+h}g,T_{t+h}v,T_{t+h}w) - B(T_tf,T_tg,T_{t+h}v,T_{t+h}w)) d\mu \\
& \to - \int_\Omega \Re(\partial_x B,\partial_y B)(T_tf,T_tg,T_tv,T_tw) \cdot (AT_tf,AT_tg) d\mu. \nonumber
\end{align}
The second and third term on the right hand side of \eqref{equ-3-proof-E-differentiable} have to be treated differently, since we do not have $v,w \in L^1(\Omega)$ any more.
Let us treat the second term.
According to the mean value theorem in one dimension, for any $x \in \Omega$, there exists some $V=V(h,x)$ in between $T_tv(x)$ and $T_{t+h}v(x)$ such that
\begin{align*}
\frac1h & \left(B(T_tf(x),T_tg(x),T_{t+h}v(x),T_{t+h}w(x))-B(T_tf(x),T_tg(x),T_tv(x),T_{t+h}w(x)) \right) \\
& = \partial_r B(T_tf(x),T_tg(x),V(h,x),T_{t+h}w(x)) \frac1h (T_{t+h}v(x)-T_tv(x)) .
\end{align*}
According to Lemma \ref{lem-B-Lipschitz}, 
\begin{align*}
& | \partial_r B(T_tf(x),T_tg(x),V(h,x),T_{t+h}w(x)) - \partial_r B(T_tf(x),T_tg(x),T_tv(x), T_tw(x)) | \\
& \lesssim_{\|f\|_\infty,\|g\|_\infty} (|T_tf(x)|^2 + |T_tg(x)|^2) ( |T_{t+h}v(x) - T_tv(x)| + |T_{t+h}w(x) - T_tw(x)|).
\end{align*}
As $\|T_{t+h}v-T_tv\|_\infty, \|T_{t+h}w - T_tw\|_\infty \leq Ch \to 0$ and $\|T_tf\|_2^2 + \|T_tg\|_2^2 \leq \|f\|_2^2 + \|g\|_2^2 < \infty$, we deduce that $\partial_r B(T_tf,T_tg,V(h,\cdot),T_{t+h}w)$ converges to $\partial_r B(T_tf,T_tg,T_tv,T_tw)$ in $L^1(\Omega)$.
On the other hand, since $T_tv \in D(A_\infty)$, we have the $w^*$ convergence $\frac1h (T_{t+h}v-T_tv) \to -AT_tv$ in $L^\infty(\Omega)$.
Therefore, the integral of the product converges appropriately, that is,
\[ \int_\Omega \partial_r B(T_tf,T_tg,V(h,\cdot),T_{t+h}w) \frac1h (T_{t+h}v - T_tv) d\mu \to - \int_\Omega \partial_r B(T_tf,T_tg,T_tv,T_tw) AT_tv d\mu . \]
We have treated the second term in \eqref{equ-3-proof-E-differentiable}.
The third term can be treated in the same way.
Combining this with \eqref{equ-4-proof-E-differentiable}, we have shown the differentiation formula \eqref{equ-E-calculated}.
Thus, the lemma follows.
\end{proof}

In the Main Theorem \ref{thm-bilinear} where we establish the bilinear estimate which will yield the weighted $L^2$ functional calculus, we have to impose some technical condition on the markovian semigroup we consider.
One of them is the following.

\begin{defi}
\label{defi-local-diffusion}
Let $(\Omega,\mu)$ be a $\sigma$-finite measure space and $(T_t)_{t \geq 0}$ a markovian semigroup on $\Omega$.
We say that $(T_t)_{t \geq 0}$ satisfies \emph{local diffusion} if there exist $C,R > 0$ such that for all $w \in L^\infty(\Omega)$ with $w \geq 0$ and $0 \leq s \leq r \leq t$, we have $T_{t+s}w(x) \leq C T_{R(t+r)}w(x)$.
\end{defi}

We are now in a position to spell out the Main Theorem.

\begin{thm}
\label{thm-bilinear}
Let $(T_t)_{t \geq 0}$ be a markovian semigroup on $(\Omega,\mu)$.
Let $w$ be a weight on $\Omega$ with $Q^A_2(w) < \infty$.
Assume one of the following alternative conditions.
\begin{enumerate}
\item The measure space is finite, $\mu(\Omega) < \infty$, or
\item For any $t > 0$, $T_t$ maps $L^\infty(\Omega)$ into the domain $D(A_\infty)$ of the $w^*$ $L^\infty$ realization of $A$, or
\item $(T_t)_{t \geq 0}$ satisfies the local diffusion from Definition \ref{defi-local-diffusion}.
\end{enumerate}
In case 2. and 3. above, assume moreover that for any $v \in L^\infty(\Omega)$, $T_tv(x) \to v(x)$ as $t \to 0+$ $\mu$-almost everywhere.
Then we have with $C > 0$ independent of $w,f$ and $g$,
\[ \int_0^\infty | \spr{AT_t f}{T_t g} | dt \leq C Q^A_2(w) \|f\|_{L^2(\Omega,wd\mu)} \|g\|_{L^2(\Omega,w^{-1}d\mu)} \]
for any $f \in L^2(\Omega,wd\mu) \cap L^2(\Omega,d\mu)$ and $g \in L^2(\Omega,w^{-1}d\mu) \cap L^2(\Omega,d\mu)$.
\end{thm}

Theorem \ref{thm-bilinear} together with our preliminary work from Subsection \ref{subsec-HI} will immediately give the following corollary on functional calculus.

\begin{cor}
\label{cor-bilinear}
Let $(T_t)_{t \geq 0}$ be a markovian semigroup on $(\Omega,\mu)$.
Let $w$ be a weight on $\Omega$ such that $Q^A_2(w) < \infty$.
Assume one of the alternative conditions from Theorem \ref{thm-bilinear}:
\begin{enumerate}
\item The measure space is finite, $\mu(\Omega) < \infty$, or
\item For any $t > 0$, $T_t$ maps $L^\infty(\Omega)$ into the domain $D(A_\infty)$ of the $w^*$ $L^\infty$ realization of $A$, or
\item $(T_t)_{t \geq 0}$ satisfies the local diffusion from Definition \ref{defi-local-diffusion}.
\end{enumerate}
In case 2. and 3. above, assume moreover that for any $v \in L^\infty(\Omega)$, $T_tv(x) \to v(x)$ as $t \to 0+$ $\mu$-almost everywhere.
Then $A$ is $\frac{\pi}{2}$-sectorial on $L^2(\Omega,wd\mu)$.
For $J > 1$ there exists a constant $C_J$ only depending on $J$ such that for any $m \in \HI(\Sigma_{\frac{\pi}{2}};J)$, we have
\[ \|m(A)\|_{L^2(\Omega,wd\mu) \to L^2(\Omega,wd\mu)} \leq C_J Q^A_2(w) \left( |m(0)| + \|m\|_{\HI(\Sigma_{\frac{\pi}{2}};J)} \right) .\]
In particular, $(1 + A)^{-J} \exp(-tA)$ extends to a bounded operator on $L^2(\Omega,wd\mu)$ and we have
\[ \|(1+A)^{-J} \exp(-tA) \|_{L^2(\Omega,wd\mu) \to L^2(\Omega,wd\mu)} \leq C (1 + t)^{-J} \quad (t \geq 0) . \]
\end{cor}

\begin{proof}
The first part of the corollary is a direct consequence of Theorem \ref{thm-bilinear} together with Proposition \ref{prop-bilinear-to-calculus-weight}.
Note hereby that $\langle f , g \rangle = \int_\Omega f(x)g(x)d\mu(x)$, and $\int_0^\infty |\langle AT_t f , T_t g \rangle| dt = \int_0^\infty | \langle AT_{2t}f, g \rangle| dt = \frac12 \int_0^\infty |\langle AT_tf,g \rangle| dt$, due to self-adjointness and lattice positivity of $T_t$.
The last part then follows from Lemma \ref{lem-Besov-norm-regularized-semigroup}.
\end{proof}

Before proving Theorem \ref{thm-bilinear}, we discuss examples where its alternative technical conditions in the hypothesis are satisfied.

\begin{remark}
\label{rem-automatic-infty-domain}
In Theorem \ref{thm-bilinear}, Corollary \ref{cor-bilinear} and Lemma \ref{lem-E-differentiable}, one of the alternative hypotheses is that $T_tw$ belongs to $D(A_\infty)$ for $t > 0$.
Note that this is true for any $w \in L^\infty(\Omega)$ in case that $(T_t)_{t \geq 0}$ is the Gauss or Poisson semigroup over $\R^d$.
Moreover, the same holds true if $\Omega$ is a space of homogeneous type and $(T_t)_{t \geq 0}$ is a self-adjoint semigroup that has an integral kernel $p_t(x,y)$ satisfying (one-sided) Gaussian estimates \cite[(1.3)]{GrTe}: there exist $C,C_+ > 0$ such that
\[ p_t(x,y) \leq C \frac{1}{\mu(B(x,\sqrt{t}))} \exp(-C_+ \dist(x,y)^2 / t). \]
\end{remark}

\begin{proof}
By self-adjointness it suffices to check that $T_tf$ falls in the domain of $A$ realized over $L^1(\Omega)$ whenever $f \in L^1(\Omega)$ and $t > 0$.
For the Gauss semigroup, we have for the integral kernel $\partial_t p_t(x,y) = c_d (-\frac{d}{2} t^{-\frac{d}{2}-1} \exp(-|x-y|^2/4t) + t^{-\frac{d}{2}} \frac{|x-y|^2}{4t^2} \exp(-|x-y|^2/4t))$ which is integrable with respect to $y$ and moreover dominated uniformly in $t$ over compact intervals $\subseteq (0,\infty)$ by an integrable function.
It follows by standard integrability arguments that $-AT_tf = \partial_t T_tf \in L^1(\Omega)$.
The proof for the Poisson semigroup and that of Gaussian estimates goes along the same lines.
For the latter, one uses that the kernel of the semigroup is differentiable in time and satisfies also a Gaussian estimate \cite[(1.3)]{Gri}
\[ |\partial_t p_t(x,y)| \leq C t^{-1} \frac{1}{\mu(B(x,\sqrt{t}))} \exp(-C_+' \dist(x,y)^2/ t) .\]
This is a standard consequence from the fact that $t \mapsto p_t(x,y)$ extends to a holomorphic function on a sector $\Sigma_\theta$ with $\theta \in (0,\frac{\pi}{2})$ satisfying the same Gaussian estimate (with different constants) when $t$ is replaced by $|t|$ \cite[Theorem 7.2, (7.5)]{Ouh}, together with an application of the Cauchy integral formula $\partial_t p_t(x,y) = \frac{1}{2\pi i}\int_{B_{\C}(t,\epsilon t)} \frac{p_z(x,y)}{(z-t)^2} dz$.
\end{proof}

\begin{remark}
\label{rem-local-diffusion}
Let us comment on the local diffusion assumption 3. in Theorem \ref{thm-bilinear}.
Assume that $(\Omega,\dist,\mu)$ is a space of homogeneous type and the markovian semigroup $(T_t)_{t \geq 0}$ has an integral kernel $p_t(x,y)$ satisfying two-sided Gaussian estimates \cite[(1.3)]{GrTe}: there exist $c,C,C_+,C_- > 0$ such that
\[ c \frac{1}{\mu(B(x,\sqrt{t}))} \exp(-C_- \dist(x,y)^2 / t) \leq p_t(x,y) \leq C \frac{1}{\mu(B(x,\sqrt{t}))} \exp(-C_+ \dist(x,y)^2 / t) . \]
Then $(T_t)_{t \geq 0}$ satisfies local diffusion.
The same holds if the semigroup satisfies two-sided Poisson estimates \cite[(7)]{DuRo}: there exist $c,C,C_+,C_- > 0$ and a decreasing function $s : [0,\infty) \to [0,\infty)$ such that
\[ c \frac{1}{\mu(B(x,\sqrt{t}))} s(C_-\dist(x,y)^2 / t) \leq p_t(x,y) \leq C \frac{1}{\mu(B(x,\sqrt{t}))} s(C_+ \dist(x,y)^2 / t) . \]
In particular, local diffusion holds for the classical Gauss and Poisson semigroup on $\R^d$.
\end{remark}

\begin{proof}
We show the local diffusion under the more general Poisson estimates (put then $s(r) = \exp(-r)$ for Gaussian estimates).
Take $R = \frac{C_-}{C_+}$ and $0 \leq s \leq r \leq t$ as required in Definition \ref{defi-local-diffusion}.
Note first that $\mu(B(x,\sqrt{t+s})) \cong \mu(B(x,\sqrt{R(t+r)}))$ thanks to the doubling property.
Then
\begin{align*}
p_{t+s}(x,y) & \lesssim \frac{1}{\mu(B(x,\sqrt{t+s}))} s(C_+ \dist(x,y)^2 / t) \cong \frac{1}{\mu(B(x,\sqrt{R(t+r)}))} s(C_- \dist(x,y)^2 / (t C_-/C_+)) \\
& \lesssim \frac{1}{\mu(B(x,\sqrt{R(t+r)}))} s(C_- \dist(x,y)^2 / (R(t+r))) \lesssim p_{R(t+r)}(x,y).
\end{align*}
Integrating over $\Omega$, we obtain
\[ T_{t+s}w(x) = \int_\Omega p_{t+s}(x,y) w(y) d\mu(y) \lesssim \int_\Omega p_{R(t+r)}(x,y) w(y) d\mu(y) = T_{R(t+r)}w(x) \]
as we wished.
\end{proof}

\begin{remark}
\label{rem-automatic-pointwise-convergence}
In Theorem \ref{thm-bilinear} and Corollary \ref{cor-bilinear}, we assumed in two cases that for $v \in L^\infty(\Omega)$, we have pointwise convergence a.e. $T_tv(x) \to v(x)$ as $t \to 0+$.
\begin{enumerate}
\item
The pointwise convergence holds true when $\Omega$ is a space of homogeneous type and $(T_t)_{t \geq 0}$ is a markovian semigroup having an integral kernel $p_t(x,y)$ with (one-sided) Gaussian estimates \cite[(1.3)]{GrTe}
\[ p_t(x,y) \leq C \frac{1}{\mu(B(x,\sqrt{t}))} \exp(- C_+ \dist(x,y)^2 / t) . \]
More generally, the same holds true if the kernel has (one-sided) Poisson estimates \cite[(7)]{DuRo}
\[ p_t(x,y) \leq C \frac{1}{\mu(B(x,\sqrt{t}))} s(\dist(x,y)^2 / t) \]
with $s : [0, \infty) \to [0,\infty)$ decreasing such that $\sum_{k = 0}^\infty 2^{kd} s(2^{2k}) < \infty$, where $d$ denotes a doubling dimension of $\Omega$.
\item
On the other hand, if $\Omega$ is a locally compact separable metric measure space and $(T_t)_{t \geq 0}$ is moreover a Feller semigroup, meaning that $T_t$ maps the space $C_0(\Omega)$ of bounded continuous functions vanishing at $\infty$ into itself and
\[ \|T_t v - v\|_\infty \to 0 \quad (t \to 0+) \]
for any $v \in C_0(\Omega)$, then for a continuous bounded function $v \in C_b(\Omega)$, we have $T_tv(x) \to v(x)$ as $t \to 0+$ pointwise almost everywhere (even locally uniformly).
Consequently, going through the end of the proof of Theorem \ref{thm-bilinear} and taking into account that the cut-off $w_n$ remains continuous if $w$ is continuous, one sees that for Feller semigroups with one of the three alternative assumptions in Theorem \ref{thm-bilinear} and \emph{continuous weights}, Theorem \ref{thm-bilinear} and Corollary \ref{cor-bilinear} are valid.
We refer to \cite{primer} for classical examples of Feller semigroups.
\end{enumerate}
\end{remark}

\begin{proof}
1. Let $x_0 \in \Omega$ and $R > 0$.
We first claim that $T_t (1_{B(x_0,r)^c}v)(x) \to 0$ as $r \to \infty$ uniformly in $t \leq 1$ and $x \in B(x_0,R)$.
Indeed, we have for $r > R+1$ and $r' = r-R$
\begin{align*}
|T_t(1_{B(x_0,r)^c}v)(x)| & \leq \|v\|_\infty \int_{B(x_0,r)^c} p_t(x,y) d\mu(y) \\
& \leq \|v\|_\infty \int_{B(x,r')^c} p_t(x,y) d\mu(y) \\
& \lesssim \|v\|_\infty \int_{B(x,r')^c} \frac{1}{\mu(B(x,\sqrt{t}))} \exp(-C_+ \dist(x,y)^2/t) d\mu(y) \\
& \lesssim \|v\|_\infty \sum_{k \geq 0: \sqrt{t} 2^{k+1} \geq r'} \frac{1}{\mu(B(x,\sqrt{t}))} \int_{B(x,2^{k+1}\sqrt{t}) \backslash B(x,2^k \sqrt{t})} \exp(-C_+ 2^{2k}) d\mu(y) \\
& \lesssim \|v\|_\infty \sum_{k \geq 0: \sqrt{t} 2^{k+1} \geq r'} \frac{2^{(k+1)d}}{\mu(B(x,2^{k+1} \sqrt{t}))} \int_{B(x,2^{k+1} \sqrt{t})} \exp(-C_+ 2^{2k}) d\mu(y) \\
& \lesssim \|v\|_\infty \sum_{k \geq 0: 2^{k+1} \geq r'} 2^{(k+1)d} \exp(-C_+ 2^{2k}) \cdot 1 \\
& \to 0 \quad (r' \to \infty).
\end{align*}
Now $|T_tv(x) -v(x)| \leq |T_t(1_{B(x_0,r)}v)(x) - 1_{B(x_0,r)}v(x)| + |T_t(1_{B(x_0,r)^c}v)(x)|$.
For the second term on the right hand side, we use the proved uniform convergence as $r \to \infty$.
Let us turn to the first term.
It suffices to show that $T_t(1_{B(x_0,r)}v)(x) \to 1_{B(x,_0,r)}v(x)$ for $r > 0$ fixed, as $t \to 0+$ $\mu$-a.e.
To this end, let $f$ be a continuous approximation of $1_{B(x_0,r)}v$ such that $\|f-1_{B(x_0,r)}v\|_2 \leq \varepsilon$.
We have
\begin{align*}
|T_t(1_{B(x_0,r)}v) - 1_{B(x_0,r)}v| & \leq |T_t(1_{B(x_0,r)}v) - T_tf| + |T_tf-f| + |f - 1_{B(x_0,r)}v| \\
& \leq M((1_{B(x_0,r)}v) - f) + |T_tf -f| + |f - 1_{B(x_0,r)}v|,
\end{align*}
where $M$ denotes the maximal operator $Mg = \sup_{t > 0}|T_tg|$, which is bounded on $L^2(\Omega)$ since $(T_t)_{t \geq 0}$ is markovian.
Then for $c > 0$,
\begin{align*}
& \mu\{x \in \Omega : \: \limsup_{t \to 0} |T_t(1_{B(x_0,r)}v)(x)-1_{B(x_0,r)}v(x)| \geq c \} \\
& \leq \mu \{ x : \: \limsup_{t \to 0} |T_tf(x)-f(x)| \geq c/3 \}  \\
& + \mu \{ x : \: |f(x) - 1_{B(x_0,r)}v(x)| \geq c/3 \} + \mu \{ x : \: M(f-1_{B(x_0,r)}v)(x) \geq c/3 \} \\
& \leq \mu \{ x : \: \limsup_{t \to 0} |T_tf(x)-f(x)| \geq c/3 \} + \frac{\|f-1_{B(x_0,r)}v\|_2^2}{(c/3)^2} + \frac{ \|M(f-1_{B(x_0,r)}v)\|_2^2 }{(c/3)^2} \\
& \leq \mu \{ x : \: \limsup_{t \to 0} |T_tf(x)-f(x)| \geq c/3 \} + \frac{\|f-1_{B(x_0,r)}v\|_2^2}{(c/3)^2} + \|M\|_{2\to 2}^2\frac{ \|f-1_{B(x_0,r)}v\|_2^2 }{(c/3)^2}.
\end{align*}
The last two terms are small according to $\|f-1_{B(x_0,r)}v\|_2 \leq \varepsilon$.
The first term is $0$ since it is a classical fact that a markovian semigroup with Gaussian estimates satisfies $\lim_{t \to 0} |T_tf(x) -f(x)| = 0$ $\mu$-almost everywhere for a \emph{continuous} function $f$.
Thus 
\[ \mu\{x \in \Omega :\: \limsup_{t \to 0} |T_t(1_{B(x_0,r)}v)(x) - 1_{B(x_0,r)}v(x)| \geq 0 \} = 0\]
for any $c > 0$.
Letting $c \to 0$, we infer the claimed pointwise convergence almost everywhere.

The proof for Poisson estimates instead of Gaussian estimates goes along the same lines.

2. This is proved in \cite[Lemma 1.8]{primer}.
\end{proof}


%

\begin{proof}[of Theorem \ref{thm-bilinear}]
Let $n \in \N$ and put $w_n$ the truncation of the weight $w$ from Definition \ref{defi-truncated-weight}.
Suppose that we have shown the theorem with $w_n, \: w_n^{-1}$ and $Q^A_2(w_n)$ in place of $w$, $w^{-1}$ and $Q^A_2(w)$.
Then since $w_n \to w$ pointwise as $n \to \infty$, we get for $f \in L^2(\Omega,wd\mu) \cap L^2(\Omega,d\mu)$ that $\|f\|_{L^2(\Omega,wd\mu)} = \lim_{n \to \infty} \|f\|_{L^2(\Omega,w_nd\mu)}$ and for $g \in L^2(\Omega,d\mu) \cap L^2(\Omega,w^{-1}d\mu)$ that $\|g\|_{L^2(\Omega,w^{-1}d\mu)} = \lim_{n \to \infty} \|g\|_{L^2(\Omega,w_n^{-1}d\mu)}$.
Moreover, according to Lemma \ref{lem-truncated-weight}, $Q^A_2(w_n) \leq Q^A_2(w)$.
Thus, the theorem would follow with the weight $w$.
We therefore assume from now on without loss of generality that there is some $\epsilon > 0$ such that $\epsilon \leq w(x) \leq \frac{1}{2\epsilon}$ for almost every $x \in \Omega$.
We prove the theorem first under the alternative hypotheses 1 or 2, and indicate the adaptation for the alternative hypothesis 3 at the end of the proof.
Put $v = 2 w^{-1}$ so that $\epsilon \leq v,w \leq \frac{1}{\epsilon}$.
Note that according to lattice positivity of $T_t$, we have 

\[ \epsilon = T_t(\epsilon 1_\Omega)(x) \leq T_t(v)(x),T_t(w)(x) \leq T_t(\frac{1}{\epsilon} 1_{\Omega})(x) = \frac{1}{\epsilon}.\]

Moreover, according to Lemma \ref{lem-semigroup-CSU}, 
\[ 1 = |T_t(1)(x)|^2 = |T_t(w^{-\frac12}w^{\frac12})(x)|^2 \leq T_t(w^{-1})(x)T_t(w)(x) \leq Q^A_2(w) ,\]
so that $2 \leq T_tv T_tw \leq Q/2$, where we put $Q = 4Q^A_2(w)$.
Then we let $f,g \in L^1(\Omega) \cap L^\infty(\Omega)$ fixed for the rest of the proof and $B$ the Bellman function from Lemma \ref{L: existence and properties of the Bellman function}.
We define the functional
\[ \mathcal{E}(t) = \int_\Omega B(T_tf(x) , T_tg(x), T_t v(x), T_t w(x)) d\mu(x)\]
from Lemma \ref{lem-E-differentiable}.
Note that indeed $(T_tv(x),T_tw(x))$ falls in the range needed in Lemma \ref{lem-E-differentiable}, according to the above.
Moreover, under the alternative hypotheses 1 or 2, according to Lemma \ref{lem-E-differentiable}
\begin{align}
\label{equ-E-old-calculated}
- \mathcal{E}'(t)&  = \Re [ \int_\Omega \partial_x B(T_tf, T_tg, T_tv, T_tw)A T_tf + \partial_y B(T_tf, T_tg,T_tv, T_tw))A T_tg \\
& + \partial_r B(T_tf, T_tg, T_tv, T_tw)AT_tv + \partial_s B(T_tf, T_tg, T_tv, T_tw)AT_t w d \mu]. \nonumber
\end{align}
Now in order to have the bilinear estimate 
\[ \int_0^\infty | \spr{AT_t f}{T_tg} | dt \leq C Q \|f\|_{L^2(w d\mu)} \|g\|_{L^2(w^{-1} d\mu)},\]
it will suffice to show (see the end of the proof of Theorem \ref{thm-bilinear} after Proposition \ref{prop-markovian-estimate})
\begin{equation}
\label{equ-quantitative-monotonicity-version-2}
 - \mathcal{E}'(t) \geq \frac{c}{Q} | \spr{AT_t f}{T_tg} |.
\end{equation}
In view of the explicit expression of $\mathcal{E}'(t)$ from \eqref{equ-E-old-calculated}, it suffices to show that for any $f,g \in L^2(\Omega)$ and $v_1,v_2 \in L^\infty(\Omega)$, all belonging to the domain of $A$, with $\epsilon \leq v_1,v_2 \leq \frac{1}{\epsilon},$ and $1 \leq v_1(x)v_2(x) \leq Q$ for any $x \in \Omega,$ we have
\begin{align}
\label{equ-calculation-1}
\left| \int_\Omega Af g d\mu \right| & \leq c Q \Re\left[ \int_\Omega \partial_x B(f,g,v_1,v_2)Af + \partial_y B(f,g,v_1,v_2)Ag \right. \\
& \left. + \partial_r B(f,g,v_1,v_2)Av_1 + \partial_s B(f,g,v_1,v_2)Av_2 d\mu \right] \nonumber
\end{align}
(from this, go then up to \eqref{equ-quantitative-monotonicity-version-2} by replacing $f$ by $T_tf,$ $g$ by $T_tg,$ $v_1$ by $T_tv$ and $v_2$ by $T_tw.$)
For this in turn, it suffices to show that for any $T$ markovian (that is, $T = T_t$ for some $t \geq 0$ fixed), we have
\begin{align}
\label{equ-calculation-2}
\left| \int_\Omega (\Id - T)(f) g d\mu \right| & \leq c Q \Re \left[ \int_\Omega \partial_x B(f,g,v_1,v_2)(\Id - T)(f) + \partial_y B(f,g,v_1,v_2)(\Id - T)(g) \right. \\
& \left. + \partial_r B(f,g,v_1,v_2)(\Id - T)(v_1) + \partial_s B(f,g,v_1,v_2)(\Id - T)(v_2) d\mu \right] \nonumber
\end{align}
(from this, go then up to \eqref{equ-calculation-1} by replacing $T$ by $T_t,$ dividing by $t,$ and letting $t \to 0;$ note that $A f = \lim_{t \to 0} \frac1t (\Id - T_t)(f)$ and similarly for $g,v_1$ and $v_2$.)
The proof of \eqref{equ-calculation-2} requires several technical steps, where the simplest form of \eqref{equ-calculation-2} is proved in Proposition \ref{prop-toy-example-estimate} and the most general form is shown in Proposition \ref{prop-markovian-estimate}.
We proceed now to these steps, and conclude the proof of the theorem once Proposition \ref{prop-markovian-estimate} is shown.
\end{proof}

\begin{prop}
\label{prop-toy-example-estimate}
Consider a two-point measure space $\Omega = \{ a , b \}$ equipped with the measure $\nu_{a,b} = \delta_a + \delta_b$ and the (negative) generator $\mathcal{G} = \begin{pmatrix} 1 & - 1 \\ -1 & 1 \end{pmatrix}$ of a markovian semigroup (given by $T_t = \frac12 \begin{pmatrix} 1 + e^{-2t} & 1 - e^{-2t} \\ 1 - e^{-2t} & 1 + e^{-2t} \end{pmatrix}$ \cite{Kri1}).
Let $Q \geq 1$.
Then for any $f,g: \{a,b\} \to \C$ and $v_1,v_2 : \{a,b\} \to (0,\infty)$ with $1 \leq v_1(x) \cdot v_2(x) \leq Q$ (and $\epsilon \leq v_1(x),v_2(x) \leq \frac{1}{\epsilon}$), the estimate \eqref{equ-calculation-1} holds with a universal constant $C < \infty$, that is,
\begin{align}
\label{equ-prop-toy-example-estimate}
\left| \int_{\{a,b\}} \mathcal{G}f \cdot g d\nu_{a,b} \right| & \leq C Q \Re\left[ \int_{\{a,b\}} \partial_x B(f,g,v_1,v_2)\mathcal{G}f + \partial_y B(f,g,v_1,v_2) \mathcal{G}g \right. \\
& \left. + \partial_r B(f,g,v_1,v_2)\mathcal{G}v_1 + \partial_s B(f,g,v_1,v_2) \mathcal{G}v_2 d\nu_{a,b} \right]. \nonumber
\end{align}
\end{prop}

\begin{proof}
The key observation is that our Bellman function satisfies its one-leg convexity, which will retranslate into \eqref{equ-prop-toy-example-estimate}.
Namely, take some $W_1,W_2 \in D_Q^\epsilon$.
Then \eqref{inequality-one leg convexity} gives
\begin{align}
B(W_1) - B(W_2) - dB(W_2)\cdot(W_1 - W_2) \geq \frac{c}{Q} |W_{1,x} - W_{2,x}| \cdot |W_{1,y} - W_{2,y}| \label{equ-1-proof-prop-toy-example-estimate}
\intertext{and}
B(W_2) - B(W_1) - dB(W_1)\cdot(W_2 - W_1) \geq \frac{c}{Q} |W_{1,x} - W_{2,x}| \cdot |W_{1,y} - W_{2,y}|. \label{equ-2-proof-prop-toy-example-estimate}
\end{align}
Taking the sum of \eqref{equ-1-proof-prop-toy-example-estimate} and \eqref{equ-2-proof-prop-toy-example-estimate} yields
\begin{equation}
\label{equ-3-proof-prop-toy-example-estimate}
dB(W_2) \cdot(W_2 - W_1) - dB(W_1) \cdot(W_2-W_1) \geq \frac{2c}{Q} |W_{1,x} - W_{2,x}| \cdot |W_{1,y} - W_{2,y}|.
\end{equation}
On the other hand, \eqref{equ-prop-toy-example-estimate} can be rewritten as
\begin{align}
& |f(a)-f(b)| \cdot |g(a) - g(b)| \leq C Q \Re [ ( \partial_x B(f(a),g(a),v_1(a),v_2(a)) \label{equ-4-proof-prop-toy-example-estimate} \\
& - \partial_x B(f(b),g(b),v_1(b),v_2(b)) ) \cdot (f(a) - f(b)) \nonumber \\
& +
\left( \partial_y B(f(a),g(a),v_1(a),v_2(a)) - \partial_y B(f(b),g(b),v_1(b),v_2(b)) \right) \cdot (g(a) - g(b)) \nonumber \\
& +
\left(\partial_r B(f(a),g(a),v_1(a),v_2(a)) - \partial_r B(f(b),g(b),v_1(b),v_2(b)) \right) \cdot (v_1(a) - v_1(b)) \nonumber \\
& +
\left(\partial_s B(f(a),g(a),v_1(a),v_2(a)) - \partial_s B(f(b),g(b),v_1(b),v_2(b)) \right) \cdot (v_2(a) - v_2(b)) ] \nonumber,
\end{align}
where the right hand side can be rewritten as $CQ$ times
\begin{align*}
& dB(f(a),g(a),v_1(a),v_2(a))\cdot (f(a)-f(b),g(a)-g(b),v_1(a)-v_1(b),v_2(a)-v_2(b)) \\
& - dB(f(b),g(b),v_1(b),v_2(b)) \cdot (f(a)-f(b),g(a)-g(b),v_1(a)-v_1(b),v_2(a)-v_2(b)) .
\end{align*}
Note hereby, that we have done the identification $\C = \R^2$, so that e.g. $f(a)-f(b) = (\Re(f(a)-f(b)),\Im(f(a)-f(b))) \in \R^2$.
If we put $W_1 = (f(b),g(b),v_1(b),v_2(b))$ and $W_2 = (f(a),g(a),v_1(a),v_2(a))$, then \eqref{equ-3-proof-prop-toy-example-estimate} gives \eqref{equ-4-proof-prop-toy-example-estimate}, and thus, \eqref{equ-prop-toy-example-estimate}.
\end{proof}

It turns out that under certain technical conditions which we will precise below, \eqref{equ-calculation-2} to show becomes an average of \eqref{equ-prop-toy-example-estimate}.
To this end, we recall the Gelfand transform which is used in \cite{CaDr} in a similar context.
Namely, suppose that $(\Omega,\mu)$ is a finite measure space.
Denote $\hat{\Omega}$ the maximal ideal space of the commutative $C^*$-algebra $L^\infty(\Omega,\mu)$.
Then there is an isometric isomorphism, the Gelfand isomorphism $\mathcal{F} : L^\infty(\Omega,\mu) \to C(\hat{\Omega})$ satisfying $\mathcal{F}(1) = 1$, $\mathcal{F}(f\cdot g) = \mathcal{F}(f) \cdot \mathcal{F}(g)$ and $\mathcal{F}(\overline{f}) = \overline{\mathcal{F}(f)}$.
In particular, $\mathcal{F}$ is positivity preserving.
We shall write in short $\mathcal{F}(f) = \hat{f}$.
Since $\hat{\Omega}$ is a compact Hausdorff space, by the Riesz representation theorem, the measure $\mu$ is transported to some positive Radon measure $\hat{\mu}$ on $\hat{\Omega}$ such that 
\[ \int_\Omega f d\mu = \int_{\hat{\Omega}} \hat{f} d\hat{\mu} \]
for $f \in L^\infty(\Omega,\mu)$.
Moreover, every $f \in L^\infty(\hat{\Omega},\hat{\mu})$ has a representative in $C(\hat{\Omega})$, so that $L^\infty(\hat{\Omega},\hat{\mu})$ and $C(\hat{\Omega})$ coincide as Banach spaces.

\begin{lemma}
\label{lem-Gelfand-functional-calculus}
Suppose that $(\Omega,\mu)$ is a finite measure space, so that the above Gelfand isomorphism exists.
Let $f \in L^\infty(\Omega)$, $D = \{ z \in \C : \: |z| \leq \|f\|_\infty \}$ and $G : D \to \C$ a continuous function.
Then $\widehat{G(f)} = G(\hat{f})$.
Similarly, let $f,g,v_1,v_2 \in L^\infty(\Omega)$ such that $(f(x),g(x),v_1(x),v_2(x)) \in {\mathcal{D}}^\epsilon_Q$ for some $\epsilon > 0$ and $Q \geq 1$, and all $x \in \Omega$.
Here $\mathcal{D}^\epsilon_Q$ is the domain of the Bellman function from Lemma \ref{L: existence and properties of the Bellman function}.
Let $G : \mathcal{D}^\epsilon_Q \to \C$ be a continuous function.
Then $(\hat{f}(y),\hat{g}(y),\hat{v}_1(y),\hat{v}_2(y)) \in \mathcal{D}^\epsilon_Q$ for any $y \in \hat{\Omega}$ and $\widehat{G(f,g,v_1,v_2)} = G(\hat{f},\hat{g},\hat{v}_1,\hat{v}_2)$.
\end{lemma}

\begin{proof}
The first statement is standard.
Indeed, since $D$ is compact, $G$ can be approximated uniformly on $D$ by polynomials $P_n$ in $z,\overline{z}$.
Now the fact that the Gelfand isomorphism is multiplicative and involutive yields $\widehat{P_n(f)} = P_n(\hat{f})$.
Then 
\[\|\widehat{G(f)}-\widehat{P_n(f)}\|_{L^\infty(\hat{\Omega},\hat{\mu})} = \|G(f) - P_n(f)\|_{L^\infty(\Omega,\mu)} \leq \|G-P_n\|_{L^\infty(D)} \|f\|_{L^\infty(\Omega)} \to 0\text{ as }n \to \infty . \]
On the other hand, 
\[\|G(\hat{f}) - P_n(\hat{f})\|_{L^\infty(\hat{\Omega},\hat{\mu})} \leq \|G - P_n \|_{L^\infty(D)} \|f\|_{L^\infty(\Omega,\mu)} \to 0\text{ as }n \to \infty . \]
Combining the two convergences yields the first statement.

For the second statement, note first that $(f,g,v_1,v_2) \in D^\epsilon_Q$ restates as $\epsilon \leq v_1,v_2 \leq \frac{1}{\epsilon}$ and $1 \leq w \leq Q$ with $w = v_1 v_2$.
Since $\mathcal{F}$ is involutive, we deduce that also $\hat{v}_1,\hat{v}_2$ are real valued.
Since $\mathcal{F}$ is an isometry (on the $L^\infty$ level), we have $\hat{v}_1,\hat{v}_2 \leq \frac{1}{\epsilon}$ and $\hat{w} = \hat{v}_1 \hat{v}_2 \leq Q$.
For the lower estimates, note that $\hat{v}_k - \epsilon = \widehat{v_k - \epsilon} \geq 0$ since $\mathcal{F}$ is positivity preserving. Similarly, $\hat{w} \geq 1$.
Thus, $(\hat{f},\hat{g},\hat{v}_1,\hat{v}_2)$ takes its values in $\mathcal{D}^\epsilon_Q$, too.
Now take again a sequence of polynomials $P_n$ in 6 commuting variables $z_f,\overline{z_f},z_g,\overline{z_g},x_1,x_2$ where $z_f,z_g \in \C$ and $x_1,x_2 \in \R$, approximating $G$ uniformly on $\mathcal{D}^\epsilon_Q \cap \left( \Ran(f,g,v_1,v_2) \cup \Ran(\hat{f},\hat{g},\hat{v}_1,\hat{v}_2) \right)$.
We have $\widehat{P_n(f,g,v_1,v_2)} = P_n(\hat{f},\hat{g},\hat{v}_1,\hat{v}_2)$, $\widehat{G(f,g,v_1,v_2)} = \lim_n \widehat{P_n(f,g,v_1,v_2)}$ and $G(\hat{f},\hat{g},\hat{v}_1,\hat{v}_2) = \lim_n P_n(\hat{f},\hat{g},\hat{v}_1,\hat{v}_2)$.
Thus, the second statement follows as was done for the first one.
\end{proof}

We will need the following proposition from \cite[Lemma 30]{CaDr}.

\begin{prop}
\label{prop-Gelfand-transform}
Let $(\Omega,\mu)$ be a finite measure space and $T$ a submarkovian operator on $(\Omega,\mu)$ in the sense of Definition \ref{defi-submarkovian}.
Then there exists a positive symmetric Radon measure $m_T$ on $\hat{\Omega} \times \hat{\Omega}$ such that
\[ \int_\Omega T f(x) g(x) d\mu(x) = \int_{\hat{\Omega} \times \hat{\Omega}} \hat{f}(x) \hat{g}(y) dm_T(x,y) \quad (f,g \in L^\infty(\Omega)) . \]
Here, by a symmetric measure we mean that $dm_T(x,y) = dm_T(y,x)$.
\end{prop}

\begin{prop}
\label{prop-submarkovian-finite-measure-estimate}
Let $(\Omega,\mu)$ be a finite measure space and $T$ a submarkovian operator on $(\Omega,\mu)$.
Then for any $f,g \in L^\infty(\Omega)$ and real valued $v_1,v_2 \in L^\infty(\Omega)$ with $1 \leq v_1(x) \cdot v_2(x) \leq Q$ and $\epsilon \leq v_1(x),v_2(x) \leq \frac{1}{\epsilon}$, the estimate \eqref{equ-calculation-2} holds with a universal constant $C < \infty$, that is,
\begin{align}
\label{equ-prop-submarkovian-finite-measure-estimate}
\left| \int_{\Omega} (\Id - T)f \cdot g d\mu \right| & \leq C Q \Re\left[ \int_{\Omega} \partial_x B(f,g,v_1,v_2)(\Id - T)f \right.\\
&+ \partial_y B(f,g,v_1,v_2) (\Id - T)g \nonumber \\
& \left. + \partial_r B(f,g,v_1,v_2) (\Id - T)v_1 + \partial_s B(f,g,v_1,v_2) (\Id - T)v_2 d\mu \right]. \nonumber
\end{align}
\end{prop}

\begin{proof}
The idea of the proof, stemming from \cite{CaDr}, is to use the representation of $T$ by the symmetric measure $m_T$ from Proposition \ref{prop-Gelfand-transform} and then to integrate over the estimate obtained in Proposition \ref{prop-toy-example-estimate}.
First, we decompose 
\begin{equation}
\label{equ-1-prop-submarkovian}
 \int_\Omega (\Id - T)f \cdot g d\mu = \int_{\Omega} (\Id - T(1) \Id)f \cdot g d\mu + \int_\Omega (T(1)\Id-T)f \cdot g d\mu.
\end{equation}
The first integral can be considered as an error term and will be controlled through the estimate \eqref{estimate-error}.
We start with the second integral.
Note that according to Proposition \ref{prop-Gelfand-transform},
\begin{align*}
\int_\Omega (T(1)\Id - T)f \cdot g d\mu & = \int_\Omega T(1)(x)f(x) g(x) d\mu(x) - \int_\Omega Tf(x) g(x) d\mu(x) \\
& = \int_{\hat{\Omega} \times \hat{\Omega}} 1 \cdot \hat{f}(y) \hat{g}(y) dm_T(x,y) - \int_{\hat{\Omega} \times \hat{\Omega}} \hat{f}(x) \hat{g}(y) dm_T(x,y) \\
& = \frac12 \int_{\hat{\Omega} \times \hat{\Omega}} \hat{f}(x) \hat{g}(x) + \hat{f}(y) \hat{g}(y) dm_T(x,y) \\
& - \frac12 \int_{\hat{\Omega} \times \hat{\Omega}} \hat{f}(x) \hat{g}(y) + \hat{f}(y) \hat{g}(x) dm_T(x,y) \\
& = \frac12 \int_{\hat{\Omega} \times \hat{\Omega}} (\hat{f}(x) - \hat{g}(x)) (\hat{f}(y) - \hat{g}(y)) dm_T(x,y) \\
&  = \frac12 \int_{\hat{\Omega} \times \hat{\Omega}} \left( \int_{\{x,y\}} \mathcal{G}\hat{f} \cdot \hat{g} d\nu_{x,y}\right) dm_T(x,y).
\end{align*}
Here we have used the symmetry of the measure $m_T$.
On the other hand, according to Lemma \ref{lem-Gelfand-functional-calculus} and Proposition \ref{prop-Gelfand-transform},
\begin{align*}
\Re & \int_\Omega \partial_x B(f,g,v_1,v_2) (T(1)\Id - T)f + \partial_y B(f,g,v_1,v_2)(T(1)\Id - T)g \\
& + \partial_r B(f,g,v_1,v_2)(T(1)\Id - T)v_1 + \partial_s B(f,g,v_1,v_2)(T(1)\Id - T)v_2 d\mu \\
& = \Re \int_{\hat{\Omega} \times \hat{\Omega}} \partial_x B(\hat{f},\hat{g},\hat{v}_1,\hat{v}_2)(x)\hat{f}(x) -  \partial_x B(\hat{f},\hat{g},\hat{v}_1,\hat{v}_2)(y)\hat{f}(x) \\
& +  \partial_y B(\hat{f},\hat{g},\hat{v}_1,\hat{v}_2)(x)\hat{g}(x) - \partial_y B(\hat{f},\hat{g},\hat{v}_1,\hat{v}_2)(y)\hat{g}(x) \\
& + \partial_r B(\hat{f},\hat{g},\hat{v}_1,\hat{v}_2)(x)\hat{v}_1(x) - \partial_r B(\hat{f},\hat{g},\hat{v}_1,\hat{v}_2)(y)\hat{v}_1(x) \\
& + \partial_s B(\hat{f},\hat{g},\hat{v}_1,\hat{v}_2)(x)\hat{v}_2(x) - \partial_s B(\hat{f},\hat{g},\hat{v}_1,\hat{v}_2)(y)\hat{v}_2(x) dm_T(x,y) \\
& = \frac12 \Re \int_{\hat{\Omega} \times \hat{\Omega}} \int_{\{x,y\}}  \partial_x B(\hat{f},\hat{g},\hat{v}_1,\hat{v}_2) \mathcal{G}\hat{f}  + \partial_y B(\hat{f},\hat{g},\hat{v}_1,\hat{v}_2)\mathcal{G}\hat{g} \\
& + \partial_r B(\hat{f},\hat{g},\hat{v}_1,\hat{v}_2) \mathcal{G}\hat{v}_1 + \partial_s B(\hat{f},\hat{g},\hat{v}_1,\hat{v}_2) \mathcal{G}\hat{v}_2
d\nu_{x,y} dm_T(x,y).
\end{align*}
Noting that $\epsilon \leq \hat{v}_1,\hat{v}_2 \leq \frac{1}{\epsilon}$ and that $1 \leq \hat{v}_1 \hat{v}_2 \leq Q$, we deduce from Proposition \ref{prop-toy-example-estimate} together with an integration over $\hat{\Omega} \times \hat{\Omega}$, and the above that
\begin{align}
& \left| \int_\Omega (T(1)\Id - T)f\cdot g d\mu \right| \leq CQ \Re \int_\Omega \partial_x B(f,g,v_1,v_2) (T(1)\Id - T)f \label{equ-2-prop-submarkovian} \\
&  + \partial_y B(f,g,v_1,v_2)(T(1)\Id - T)g \nonumber \\
& + \partial_r B(f,g,v_1,v_2)(T(1)\Id - T)v_1 + \partial_s B(f,g,v_1,v_2)(T(1)\Id - T)v_2 d\mu. \nonumber
\end{align}
In other words, the second integral of the right hand side in \eqref{equ-1-prop-submarkovian} is estimated.
We proceed to the first integral of the right hand side in \eqref{equ-1-prop-submarkovian}.
According to \eqref{estimate-error}, we have the pointwise estimate
\begin{align*}
\MoveEqLeft \Re\left[ \partial_x B(f,g,v_1,v_2)\cdot f + \partial_y B(f,g,v_1,v_2)\cdot g + \partial_r B(f,g,v_1,v_2) \cdot v_1 + \partial_s B(f,g,v_1,v_2)\cdot v_2 \right] \\
& \geq \frac{c}{Q} |f\cdot g|.
\end{align*}
Multiplying this inequality pointwise with the positive function $1 - T(1)$ and then integrating over $\Omega$ yields
\begin{align}
\Re & \left[ \int_\Omega \partial_x B(f,g,v_1,v_2)\cdot (1 - T(1))f + \partial_y B(f,g,v_1,v_2)\cdot(1 - T(1)) g  \right.\label{equ-3-prop-submarkovian} \\
& \left.+ \partial_r B(f,g,v_1,v_2) \cdot (1 - T(1))v_1 + \partial_s B(f,g,v_1,v_2) \cdot (1 - T(1)) v_2d\mu\right] \nonumber \\
& \geq \frac{c}{Q} \int_\Omega |(1 - T(1))f\cdot g| d\mu \geq \frac{c}{Q} \left| \int_\Omega (1-T(1))f \cdot g d\mu \right|. \nonumber
\end{align}
Summing the above estimates \eqref{equ-2-prop-submarkovian} and \eqref{equ-3-prop-submarkovian} for the first and the second integral in \eqref{equ-1-prop-submarkovian}, we obtain
\begin{align*}
\left| \int_\Omega (\Id - T)f g d\mu \right| & \leq \left| \int_\Omega (\Id - T(1)\Id)fgd\mu \right| + \left| \int_\Omega (T(1)\Id -T)fg d\mu \right| \\
& \leq CQ \Re \int_\Omega \partial_x B(f,g,v_1,v_2) (\Id - T)f + \partial_y B(f,g,v_1,v_2)(\Id - T)g \\
& + \partial_r B(f,g,v_1,v_2)(\Id - T)v_1 + \partial_s B(f,g,v_1,v_2)(\Id - T)v_2 d\mu.
\end{align*}
Thus, we proved \eqref{equ-prop-submarkovian-finite-measure-estimate}.
\end{proof}

We are now in the position to prove \eqref{equ-calculation-2} in its general form, and thus to conclude the proof of Theorem \ref{thm-bilinear}.

\begin{prop}
\label{prop-markovian-estimate}
Let $(\Omega,\mu)$ be a $\sigma$-finite measure space and $T$ a submarkovian operator on $(\Omega,\mu)$.
Then for any $f,g \in L^1(\Omega) \cap L^\infty(\Omega)$ and real valued $v_1,v_2 \in L^\infty(\Omega)$ with $1 \leq v_1(x) \cdot v_2(x) \leq Q$ and $\epsilon \leq v_1(x),v_2(x) \leq \frac{1}{\epsilon}$, the estimate \eqref{equ-calculation-2} holds, that is
\begin{align}
\left| \int_{\Omega} (\Id - T)f \cdot g d\mu \right| & \leq C Q \Re\left[ \int_{\Omega} \partial_x B(f,g,v_1,v_2)(\Id - T)f \right. \label{equ-prop-markovian-estimate}\\
&+ \partial_y B(f,g,v_1,v_2) (\Id - T)g \nonumber \\
& \left. + \partial_r B(f,g,v_1,v_2) (\Id - T)v_1 + \partial_s B(f,g,v_1,v_2) (\Id - T)v_2 d\mu \right]. \nonumber
\end{align}
\end{prop}

\begin{proof}
Since $\Omega$ is $\sigma$-finite, we can write $\Omega = \bigcup_{n \in \N} \Omega_n$ with $\Omega_n$ an increasing sequence of measurable sets of finite measure.
Let $f,g \in L^1(\Omega) \cap L^\infty(\Omega)$ and put $f_n = 1_{\Omega_n} f$ and similarly $g_n,v_{1,n},v_{2,n}$.
We note that $T_n = 1_{\Omega_n} T 1_{\Omega_n}$ is again a submarkovian operator, acting on $(\Omega_n,\mu|_{\Omega_n})$ of finite measure.
Also, $(\Id - T_n)f_n = 1_{\Omega_n} (\Id-T)f_n$ and similarly for $g$, $v_1$ and $v_2$.
Thus, according to Proposition \ref{prop-submarkovian-finite-measure-estimate}, the estimate \eqref{equ-prop-markovian-estimate} holds with $f,g,v_1$ and $v_2$ replaced by $f_n,g_n,v_{1,n}$ and $v_{2,n}$ and $\Omega$ replaced by $\Omega_n$.
We have pointwise and $L^2$ dominated convergence of $f_n\to f$ and $g_n\to g$, so these are convergences in $L^2$.
Thus, $\int_{\Omega_n} (\Id - T)f_n g_n d\mu \to \int_{\Omega} (\Id - T)fg d\mu$.
We argue similarly for the right hand side of \eqref{equ-prop-markovian-estimate}, noting that $\partial_k B(f_n,g_n,v_{1,n},v_{2,n})$ converges to $\partial_k B(f,g,v_1,v_2)$ in $L^2$ for $k = x,y$ using Lemma \ref{lem-partial1-L2} and in $L^1$ for $k = r,s$ using Lemma \ref{lem-partial5-L1}, together with the following reasoning:
$v_{1,n}$ converges to $v_1$ weak$*$ in $L^\infty(\Omega)$.
Then, writing $w= \partial_r B(f,g,v_1,v_2) \in L^1(\Omega)$, we have
\begin{align*}
\langle 1_{\Omega_n} (\Id - T)(1_{\Omega_n} v_1),w \rangle_{L^\infty,L^1} & = \langle v_1, 1_{\Omega_n}(\Id - T)(1_{\Omega_n} w) \rangle \\
& \to \langle v_1, (\Id - T)w \rangle = \langle (\Id - T)v_1 , w \rangle,
\end{align*}
since a subsequence of $(\Id - T)(1_{\Omega_n}w)$  converges pointwise and dominated in $L^1(\Omega)$ to $(\Id - T)(w)$.
Use the same reasoning for $v_{2,n}$ and $\partial_s B(f,g,v_1,v_2)$.
\end{proof}

\begin{proof}[- conclusion of Theorem \ref{thm-bilinear}]
According to Proposition \ref{prop-markovian-estimate}, estimate \eqref{equ-calculation-2} holds for $f,g \in L^1(\Omega) \cap L^\infty(\Omega)$ and $T = T_t$.
Dividing by $t$ and letting $t \to 0+$ yields \eqref{equ-calculation-1}.
Since $T_t$ leaves $L^1(\Omega) \cap L^\infty(\Omega)$ invariant, we deduce that for these $f$ and $g$,
\[ |\langle AT_tf,T_tg \rangle| \leq-  CQ \mathcal{E}'(t). \]
Integrating over $t \in (0,\infty)$ yields in view of the lower and upper estimate of the Bellman function that
\begin{align*}
\int_0^\infty | \langle AT_tf,T_tg \rangle| dt & \leq CQ( \liminf_{r \to 0+}\mathcal{E}(r) - \limsup_{s \to \infty} \mathcal{E}(s) )\leq CQ \liminf_{r \to 0+}\mathcal{E}(r) \\
& \lesssim Q \left(\liminf_{r \to 0+} \int_\Omega |T_rf|^2 (T_r(w^{-1}))^{-1} d\mu + \int_\Omega |T_rg|^2 (T_rw)^{-1} d\mu \right) \\
& \lesssim Q^A_2(w) \left(\|f\|_{L^2(\Omega,wd\mu)}^2 + \|g\|_{L^2(\Omega,w^{-1}d\mu)}^2 \right).
\end{align*}
Here we have used the assumption $\lim_{t \to 0} T_{t}(w^{\pm 1})(x) = w^{\pm 1}(x)$ almost everywhere  in the last estimate in the case of alternative assumption 2.
In case of alternative assumption 1., we have $\lim_{n \to \infty} T_{t_n}(w^{\pm 1})(x) = w^{\pm 1}(x)$ almost everywhere along a sequence $t_n \to 0+$, since then $T_t(w^{\pm 1}) \to w^{\pm 1}$ in $L^1(\Omega)$, which allows still to conclude as above.
Changing $f \leadsto \lambda f$ and $g \leadsto \lambda^{-1} g$ and optimizing in $\lambda > 0$ yields the desired bilinear estimate
\[ \int_0^\infty | \langle AT_tf, T_tg \rangle| dt \leq CQ \|f\|_{L^2(\Omega,wd\mu)} \|g\|_{L^2(\Omega,w^{-1}d\mu)}. \]
To replace the restriction $f,g \in L^1(\Omega,\mu)\cap L^\infty(\Omega,\mu)$ by the requirement $f \in L^2(\Omega,\mu) \cap L^2(\Omega,wd\mu)$ and $g \in L^2(\Omega,\mu) \cap L^2(\Omega,w^{-1}d\mu)$, we use a standard approximation $f_n,g_n \in L^1(\Omega,\mu) \cap L^\infty(\Omega,\mu)$ with $|f_n|\leq |f|$ and $|g_n|\leq |g|$, and pointwise convergence $f_n \to f$ and $g_n \to g$.
Then we have $AT_tf = \lim_n AT_tf_n$ and $T_tg = \lim_n T_tg_n$ in $L^2(\Omega)$, so that $\lim_n \langle AT_tf_n, T_tg_n \rangle = \langle AT_t f, T_tg \rangle$.
Then by dominated convergence, for $0 < t_0 < t_1 < \infty$,
\begin{align*}
\int_{t_0}^{t_1} | \langle AT_tf, T_tg \rangle | dt & = \lim_n \int_{t_0}^{t_1}  | \langle AT_t f_n, T_t g_n \rangle| dt \leq CQ \lim_n \|f_n\|_{L^2(\Omega,w d\mu)} \|g_n\|_{L^2(\Omega,w^{-1} d\mu)} \\
& = CQ \|f\|_{L^2(\Omega, w d\mu)} \|g\|_{L^2(\Omega, w^{-1} d\mu)}.
\end{align*}
We can now let $t_0 \to 0+$ and $t_1 \to \infty$ to conclude.

Let us now indicate the adaptation to the case of the alternative hypothesis 3.
In this case, in order to have the functional $\E(t)$ differentiable, we have to modify the involved weight functions.
That is, we let for given $h > 0$ the weights
\[ v = \frac{1}{h} \int_0^h T_s (w^{-1}) ds , \quad \tilde{w} = \frac{1}{h} \int_0^h T_rw dr . \]
Then by classical semigroup theory, $v$ and $\tilde{w}$ belong to $D(A_\infty)$.
Moreover, $\epsilon \leq v,\tilde{w} \leq \frac{1}{\epsilon}$.
We now leverage the local diffusion property to enframe $T_tv \cdot T_t\tilde{w}$.
Namely, we have, as soon as $h \leq t$,
\begin{align*}
T_tv \cdot T_t\tilde{w} & = \frac{1}{h} \int_0^h T_{t+s}(w^{-1}) ds \cdot \frac{1}{h} \int_0^h T_{t+r}w dr \\
& = \frac{1}{h^2} \int_0^h \int_0^h T_{t+s}(w^{-1}) T_{t+r}w ds dr \\
& = \frac{1}{h^2} \int_0^h \int_0^r T_{t+s}(w^{-1}) T_{t+r}w ds dr + \frac{1}{h^2} \int_0^h \int_0^s T_{t+s}(w^{-1}) T_{t+r}w dr ds \\
& \lesssim \frac{1}{h^2} \int_0^h \int_0^r T_{R(t+r)}(w^{-1}) T_{R(t+r)}w ds dr + \frac{1}{h^2} \int_0^h \int_0^s T_{R(t+s)}(w^{-1}) T_{R(t+s)}w dr ds \\
& \leq Q^A_2(w) \frac{1}{h^2} \int_0^h \int_0^r 1 ds dr + Q^A_2(w) \frac{1}{h^2} \int_0^h \int_0^s 1  dr ds \\
& = Q^A_2(w).
\end{align*}
In the same manner, we estimate from below
\begin{align*}
T_tv \cdot T_t\tilde{w} & = \frac{1}{h^2} \int_0^h \int_0^r T_{t+s}(w^{-1}) T_{t+r}w ds dr + \frac{1}{h^2} \int_0^h \int_0^s T_{t+s}(w^{-1}) T_{t+r}w dr ds \\
& \gtrsim \frac{1}{h^2} \int_0^h \int_0^r T_{\frac1R(t+s)}(w^{-1}) T_{\frac1R(t+s)}w ds dr + \frac{1}{h^2} \int_0^h \int_0^s T_{\frac1R(t+r)}(w^{-1}) T_{\frac1R(t+r)}w dr ds \\
& \geq \frac{1}{h^2} \int_0^h \int_0^r 1 ds dr + \frac{1}{h^2} \int_0^h \int_0^s 1 dr ds \\
& = 1.
\end{align*}
Thus, defining $Q = C Q^A_2(w)$ and multiplying $v$ with a constant, as we did in the case of the alternative hypotheses 1 and 2, we get $2 \leq T_tv T_t{\tilde w} \leq Q/2$ for $h \leq t$.
We define the functional
\[ \mathcal{E}(t) = \int_\Omega B(T_tf,T_tg,T_tv,T_t\tilde{w}) d\mu \quad (t \geq h) \]
and deduce from Lemma \ref{lem-E-differentiable} that $\mathcal{E}$ is differentiable with derivative given by \eqref{equ-E-calculated}.
Going along the same lines as in the case of the alternative hypotheses 1 and 2, we deduce that
\[ - \mathcal{E}(t) \geq \frac{c}{Q} | \langle AT_tf, T_tg \rangle|  \quad (t \geq h) \]
and thus that
\[ \int_h^\infty | \langle AT_tf, T_tg \rangle | dt \leq C Q \left( \|T_hf \|_{L^2(\Omega,(T_h v)^{-1} d\mu)}^2 + \|T_h g\|_{L^2(\Omega,(T_h \tilde{w})^{-1} d\mu)}^2 \right) . \]
Letting $h \to 0+$, the left hand side converges to $\int_0^\infty | \langle AT_tg, T_tg \rangle| dt$.
Moreover, according to the assumption of Theorem \ref{thm-bilinear}, $T_hv$ converges pointwise almost everywhere to $w^{-1}$.
Also $|T_hf|^2 \to |f|^2$ in $L^1(\Omega)$, so converges pointwise almost everywhere and $L^1$ dominated, along a sequence $h = h_k \to 0$.
It follows that $\|T_{h_k}f\|^2_{L^2(\Omega,(T_{h_k}v)^{-1}d\mu)} \to \|f\|^2_{L^2(\Omega,wd\mu)}$.
In the same manner, there exists a subsequence $h_l' =h_{k_l}$ such that $\|T_{h_l'} g\|_{L^2(\Omega,(T_{h_l'} \tilde{w})^{-1} d\mu)} \to \|g\|_{L^2(\Omega,w^{-1} d\mu)}$.
We infer that
\[ \int_0^\infty | \langle AT_tf, T_tg \rangle| dt \leq CQ \left( \|f\|_{L^2(\Omega,w d\mu)}^2 + \|g\|_{L^2(\Omega,w^{-1}d\mu)}^2 \right). \]
We conclude the proof with the optimization $f \leadsto \lambda f$ and $g \leadsto \lambda^{-1} g$ and the approximation of $L^2$ functions by $L^1\cap L^\infty$ functions as we did in the case of alternative hypothesis 1 and 2 above.
\end{proof}

\begin{prop}
\label{prop-maximal-regularity}
Let $(T_t)_{t \geq 0}$ be a markovian semigroup on $(\Omega,\mu)$.
Assume one of the following alternative conditions.
\begin{enumerate}
\item The measure space is finite, $\mu(\Omega) < \infty$, or
\item For any $t > 0$, $T_t$ maps $L^\infty(\Omega)$ into the domain $D(A_\infty)$ of the $w^*$ $L^\infty$ realization of $A$, or
\item $(T_t)_{t \geq 0}$ satisfies the local diffusion from Definition \ref{defi-local-diffusion}.
\end{enumerate}
In case 2. and 3. above, assume moreover that for any $v \in L^\infty(\Omega)$, $T_tv(x) \to v(x)$ as $t \to 0+$ $\mu$-almost everywhere.
Assume that the weight $w$ satisfies $w^\delta \in Q^A_2$ for some $\delta > 1$.
Then $A$ has an $\HI(\Sigma_\theta)$ calculus on $L^2(\Omega,wd\mu)$ for some $\theta < \frac{\pi}{2}$ and in particular, the analytic semigroup $T_z$ extends boundedly to $L^2(\Omega,wd\mu)$ for $|\arg z| < \frac{\pi}{2} - \theta$, and $A$ has maximal regularity on $L^2(\Omega,wd\mu)$.
\end{prop}

\begin{proof}
Corollary \ref{cor-bilinear} yields that $A$ is $\frac{\pi}{2}$-sectorial on $L^2(\Omega,w^\delta d\mu)$ and has an $\HI(\Sigma_\sigma)$ calculus for any $\sigma > \frac{\pi}{2}$.
Moreover, by self-adjointness, $A$ is $0$-sectorial on $L^2(\Omega,\mu)$ and has an $\HI(\Sigma_\sigma)$ calculus for any $\sigma > 0$.
Note that the spaces interpolate (complex) and we have $L^2(\Omega,w d\mu) = [L^2(\Omega,\mu),L^2(\Omega,w^\delta d\mu)]_{\frac1\delta}$ \cite[5.5.3 Theorem]{BeL}.
By Stein's interpolation \cite{Ste56}, $A$ is then $(\frac1\delta \cdot \frac{\pi}{2})$-sectorial on $L^2(\Omega,w d\mu)$.
Moreover, by complex interpolation, $A$ has an $\HI(\Sigma_\sigma)$ calculus on $L^2(\Omega,wd\mu)$ for any $\sigma > \max(0,\frac{\pi}{2}) = \frac{\pi}{2}$.
The angle $\frac{\pi}{2}$ can then be reduced to $\theta_0 = \frac1\delta \cdot \frac{\pi}{2} < \frac{\pi}{2}$ e.g. by the method of imaginary powers \cite[Proof of Proposition 5.8]{JLMX}.
Then $A$ has an $\HI(\Sigma_\theta)$ calculus on $L^2(\Omega,wd\mu)$ for any $\theta > \theta_0$.
It is well-known that this implies that on $L^2(\Omega,wd\mu)$, the analytic semigroup angle is bigger or equal to $\frac{\pi}{2} - \theta$ and that $A$ has maximal regularity \cite{KW04,dS}.
\end{proof}

\begin{remark}
\label{rem-power-characteristic}
Note that if $\Omega = \R^n$ and the characteristic $Q^A_2$ is equivalent to the spatial characteristic $Q^{class}_2$ (see Remark \ref{remark-classical-characteristic}), then according to \cite[Theorem 2.7, p. 399]{GCR}, a weight $w \in Q^A_2$ already satisfies $w^\delta \in Q^A_2$ for some $\delta > 1$.
Thus, under the hypotheses of Corollary \ref{cor-bilinear}, Proposition \ref{prop-maximal-regularity} yields that for any $Q^A_2$ weight $w$, $A$ has an $\HI(\Sigma_\theta)$ calculus on $L^2(\Omega,wd\mu)$ for some $\theta < \frac{\pi}{2}$, and thus maximal regularity.
\end{remark}


\begin{remark}
\label{rem-DSY-GY}
Assume that $(\Omega, \dist, \mu)$ is a space of homogeneous type and that the semigroup $(T_t)_{t \geq 0}$ satisfies Gaussian estimates \eqref{equ-GE} and is self-adjoint on $L^2(\Omega)$.
In \cite[Theorem 3.2]{DSY} \cite[Theorem 4.2]{GY}, a functional calculus for the (negative) generator $A$ on weighted $L^p$ spaces is proved.
If $(T_t)_{t \geq 0}$ is in addition markovian, then one can compare these results to ours.
On the one hand, the results in \cite{DSY,GY} are stronger in respect that $L^p$ spaces with exponents $p \in (r_0,\infty)$ (for a certain $r_0 \in [1,2)$) different from $2$ are allowed and that non-holomorphic spectral multipliers $m$ defined on $[0,\infty)$ and satisfying a so-called H\"ormander condition
\[ \|m\|_{\mathcal{H}^s} = \sup_{t  > 0} \|\eta m(t\cdot)\|_{W^\infty_s(\R)} < \infty \]
are admitted.
Here, $s > \frac{d}{2}$, where $d$ is a doubling dimension of $\Omega$, $W^\infty_s(\R)$ stands for the usual Sobolev space and $\eta$ is any $C^\infty_c(0,\infty)$ function different from $0$.
Note that any $\HI(\Sigma_\theta)$ function for any $\theta \in (0,\pi)$ satisfies the H\"ormander condition and $\|m\|_{\mathcal{H}^s} \lesssim_{\theta,s} \|m\|_{\infty,\theta}$.
On the other hand, our result, Corollary \ref{cor-bilinear} is stronger in respect that $Q^A_2$ weights are admitted, whereas in \cite{DSY,GY}, one has to take the smaller class of weights belonging to $Q^{class}_{2/r_0} \subset Q^{class}_{2} \subset Q^A_2$ (cf. Remark \ref{remark-classical-characteristic}).
Indeed, Corollary \ref{cor-bilinear} applies according to Remarks \ref{rem-automatic-infty-domain} and \ref{rem-automatic-pointwise-convergence}.

Note that in our setting, no H\"ormander calculus result on weighted $L^2$ can be available in general, see Theorem \ref{thm-no-Hormander}. \end{remark}

\section{The positive $\frac{\pi}{2}$ angle result, submarkovian case}
\label{sec-submarkovian}

In Theorem \ref{thm-bilinear}, we assumed that $(T_t)_{t \geq 0}$ is a markovian semigroup, so that $T_t(1) = 1$.
It turns out that there is also a version of that theorem for submarkovian semigroups.
So the main objective of this section is to prove Theorem \ref{thm-bilinear-submarkovian} and Corollary \ref{cor-bilinear-submarkovian}.
Let $(T_t)_{t \geq 0}$ be a submarkovian semigroup on $(\Omega,\mu)$.
Define $\Omega' = \Omega \cup \{ \infty \}$ with some exterior cemetery point $\infty \not\in \Omega$.
Define moreover the measure $\mu'(A) = \mu(A \cap \Omega) + \delta_A(\infty)$ on $\Omega'$.
Then we put for $f' = f'|_{\Omega} + f'(\infty) \delta_\infty \in L^\infty(\Omega')$ and $t \geq 0$
\begin{equation}
\label{equ-extension-submarkovian-markovian}
S_t(f')(x) = \begin{cases} T_t(f'|_{\Omega})(x) + f'(\infty)(1-T_t(1))(x) & : \: x \in \Omega \\ f'(\infty) & : \: x = \infty. \end{cases} \end{equation}
It is easy to check that $S_t$ is a positive semigroup\footnote{\thefootnote. We do not need any continuity assumption of $t \mapsto S_t$.}
on $L^\infty(\Omega')$ and that moreover, $S_t(1) = 1$.
Thus $S_t$ are contractions on $L^\infty(\Omega')$.
Note however that $S_t$ is in general no longer self-adjoint or even defined on $L^p(\Omega')$ for $p < \infty$.
If $w : \Omega \to (0,\infty)$ is a weight, we define the characteristic associated with $S_t$ by
\[ \tilde{Q}^A_2(w) = \sup_{t > 0} \esssup_{x \in \Omega'} S_t(w')(x) S_t(w'^{-1})(x), \]
where $w'(x) = w(x)$ for $x \in \Omega$ and $w'(\infty) = 1$.
Note that even if $w$ has support only in $\Omega$, the characteristic $\tilde{Q}^A_2(w)$ is in general larger than $Q^A_2(w)$.
As in the markovian case, we will need the following lemma on differentiability of the Bellman functional.

\begin{lemma}
\label{lem-E-differentiable-submarkovian}
Let $B$ be the Bellman function from Lemma \ref{L: existence and properties of the Bellman function} with domain $\mathcal{D}^{\varepsilon}_Q$.
Let $(\Omega,\mu)$ be a $\sigma$-finite measure space, $(T_t)_{t \geq 0}$ a submarkovian semigroup on $\Omega$ and $f,g \in L^1(\Omega) \cap L^\infty(\Omega)$.
Consider the one point addition $\Omega' = \Omega \cup \{ \infty \}$ and the amplified semigroup $(S_t)_{t \geq 0}$ associated with $(T_t)_{t \geq 0}$, as above.
Let further $v,w \in L^\infty(\Omega')$ with $\varepsilon \leq v,w \leq \frac{1}{\varepsilon}$ and $2 \leq S_tv S_tw \leq Q/2$ for all $t$ belonging to some interval $(t_0,t_1) \subseteq (0,\infty)$.
Assume one of the following conditions.
\begin{enumerate}
\item The measure space is finite, i.e. $\mu(\Omega) < \infty$, or
\item $T_t(v|_\Omega),T_t(w|_\Omega),T_t(1_\Omega)$ belong to $D(A_\infty)$ for $t \in (t_0,t_1)$.
\end{enumerate}
Define the functional
\[ \mathcal{E}(t) = \int_\Omega B(T_tf(x) , T_tg(x), S_tv(x), S_t w(x)) d\mu(x) \quad (t \in (t_0,t_1)) .\]
Then $\mathcal{E}(t)$ is differentiable for $t \in (t_0,t_1)$, and we have
\begin{align}
\label{equ-E-calculated-submarkovian}
- \mathcal{E}'(t)&  = \Re [ \int_\Omega \partial_x B(T_tf, T_tg, S_tv, S_tw)A T_tf + \partial_y B(T_tf, T_tg,S_tv, S_tw)A T_tg \\
& + \partial_r B(T_tf, T_tg, S_tv, S_tw)(AT_t(v|_\Omega) - v(\infty) AT_t1) \nonumber \\
& + \partial_s B(T_tf, T_tg, S_tv, S_tw)(AT_t(w|_\Omega) - w(\infty) AT_t1) d \mu]. \nonumber
\end{align}
\end{lemma}

\begin{proof}
Observe that
\begin{align*}
 \frac{1}{h} (S_{t+h}(v) - S_tv) & = 1_\Omega \left\{ \frac1h (T_{t+h}(v|_\Omega) - T_t(v|_\Omega)) + v(\infty) \frac1h (T_t1 - T_{t+h}1) \right\} \\
& \to 1_\Omega \left\{ - AT_t(v|_\Omega) + v(\infty) AT_t(1)  \right\} \quad (h \to 0+),
\end{align*}
and the same formula for $w$ in place of $v$.
This convergence holds in $L^1(\Omega)$ under the first assumption, and in $L^\infty(\Omega)$ with respect to $w^*$ topology under the second assumption.
In particular, under the second assumption, we have $\|S_{t+h}v - S_tv\|_{L^\infty(\Omega)} \leq Ch \to 0 \quad (h \to 0+)$.
Now the rest of the proof goes the same lines as that of Lemma \ref{lem-E-differentiable}.
\end{proof}

\begin{thm}
\label{thm-bilinear-submarkovian}
Let $(\Omega,\mu)$ be a $\sigma$-finite measure space and $(T_t)_{t \geq 0}$ a submarkovian semigroup.
Assume one of the following alternative conditions.
\begin{enumerate}
\item The measure space is finite, $\mu(\Omega) < \infty$, or
\item For any $t > 0$, $T_t$ maps $L^\infty(\Omega)$ into the domain $D(A_\infty)$ of the $w^*$ $L^\infty$ realization of $A$, or
\item The amplified semigroup $(S_t)_{t \geq 0}$ on $\Omega'$ satisfies the local diffusion from Definition \ref{defi-local-diffusion} and $1_\Omega \in D(A_\infty)$.
\end{enumerate}
In case 2. and 3. above, assume moreover that for any $v \in L^\infty(\Omega)$, $T_tv(x) \to v(x)$ as $t \to 0+$ $\mu$-almost everywhere.
Let moreover $w :\Omega \to (0,\infty)$ be a weight with $\tilde{Q}^A_2(w) < \infty$.
Then there exists a constant $C < \infty$ such that for any $f \in L^2(\Omega,\mu) \cap L^2(\Omega,w d\mu)$ and $g \in L^2(\Omega,\mu) \cap L^2(\Omega, w^{-1} d\mu)$, we have
\[ \int_0^\infty | \langle AT_t f , T_t g \rangle| dt \leq C \tilde{Q}^A_2(w) \| f \|_{L^2(\Omega,wd\mu)} \| g \|_{L^2(\Omega,w^{-1} d\mu)} . \]
\end{thm}

\begin{proof}
We proceed in a similar manner to Theorem \ref{thm-bilinear}.
We let $w'$ be the extended weight $w'(x) = w(x)$ if $x \in \Omega$ and $w'(\infty) = 1$.
Again, due to Lemma \ref{lem-truncated-weight}, it suffices to assume that the weight $w'$ satisfies $\epsilon \leq w' \leq \frac{1}{2 \epsilon}$ and thanks to positivity of $S_t$ and the fact that $S_t(1) = 1$, we will then also have $\epsilon \leq S_t(w'),S_t(w'^{-1}) \leq \frac{1}{2 \epsilon}$.
Assume the alternative assumption 1. or 2. above.
Put $v = 2(w')^{-1}$ and $Q = 4 \tilde{Q}^A_2(w)$, so that $2 \leq S_tv S_tw' \leq Q/2$.
Then for $f,g$ as in the theorem, we define the functional
\[ \mathcal{E}(t) = \int_{\Omega} B(T_tf, T_tg, S_tv, S_t(w')) d\mu . \]
As in the proof of Theorem \ref{thm-bilinear}, it will suffice to prove for $f,g \in L^2(\Omega)$ and $t > 0$ that
\begin{equation}
\label{equ-1-proof-thm-bilinear-submarkovian}
\left| \int_\Omega A T_t f \cdot T_tg d\mu \right| \leq - CQ  \mathcal{E}'(t) .
\end{equation}
Indeed, then we use 
\begin{align*}
\liminf_{r \to 0+} \mathcal{E}(r) & \leq C \liminf_{r \to 0+} \|T_rf\|^2_{L^2(\Omega,S_r(v)^{-1}d\mu)} + \|T_rg\|^2_{L^2(\Omega,S_r(w')^{-1}d\mu)} \\
& \leq C \left( \|f\|^2_{L^2(\Omega,w d\mu)} + \|g\|^2_{L^2(\Omega,w^{-1} d\mu)} \right),
\end{align*}
since $S_r(v)(x) \to v(x) + 0$ and $S_r(w')(x) \to w(x) + 0$ for a.e. $x \in \Omega$.
We have according to \eqref{equ-E-calculated-submarkovian}
\begin{align*}
- \mathcal{E}'(t) & = \Re \int_\Omega \partial_x B(T_tf, T_tg, S_tv,S_t(w')) AT_tf + \partial_y  B(T_tf, T_tg, S_tv,S_t(w')) AT_tg  \\
& + \partial_r B(T_tf,T_tg, S_tv, S_t(w')) (AT_t(v|_\Omega) - v(\infty)AT_t1) \\
& + \partial_s B(T_tf,T_tg, S_tv, S_t(w')) (AT_t(w) - w(\infty)AT_t(1)) d\mu,
\end{align*}
Replace first, as in the proof of Theorem \ref{thm-bilinear}, $T_tf$, $T_tg$, $S_tv$, $S_t(w')$ by generic $f,g \subseteq L^2(\Omega)$ and $v_1,v_2 \in L^\infty(\Omega')$ with $1 \leq v_1(x)v_2(x) \leq Q$ and $\epsilon \leq v_1(x),v_2(x) \leq \frac{1}{\epsilon}$.
Then replacing $A$ by $\frac{1}{t}(\Id - T_t)$, it will then suffice to show for $f,g \in L^\infty(\Omega) \cap L^1(\Omega)$,
\begin{align}
\label{equ-2-proof-thm-bilinear-submarkovian}
\left| \int_\Omega (\Id - T_t)f \cdot g d\mu \right| & \leq CQ \Re \int_\Omega \partial_x B(f, g, v_1, v_2) (\Id - T_t)f + \partial_y  B(f, g, v_1, v_2) (\Id - T_t)g \\
& + \partial_r B(f,g, v_1, v_2) (\Id - S_t)v_1 + \partial_s B(f,g, v_1, v_2) (\Id - S_t)v_2 d\mu . \nonumber
\end{align}
According to \eqref{equ-extension-submarkovian-markovian}, we decompose the right hand side of \eqref{equ-2-proof-thm-bilinear-submarkovian} into
\begin{align*}
& CQ \Re \int_\Omega \partial_x B(f, g, v_1, v_2) (\Id - T_t)f + \partial_y  B(f, g, v_1, v_2) (\Id - T_t)g \\
& + \partial_r B(f,g, v_1, v_2) (\Id - T_t)v_1 + \partial_s B(f,g, v_1, v_2) (\Id - T_t)v_2 d\mu \\
& + CQ \int_{\Omega} \partial_r B(f,g,v_1,v_2) v(\infty)(T_t(1) - 1) + \partial_s B(f,g,v_1,v_2) w(\infty)(T_t(1) - 1) d\mu .
\end{align*}
Note that the first term is indeed minorised by $\left| \langle (\Id - T_t) f , g \rangle \right|$, according to Proposition \ref{prop-markovian-estimate} (note that we had allowed $T$ to be a submarkovian operator in this proposition).
The second term is positive, since $T_t(1) - 1 \leq 0$ and $\partial_r B, \partial_s B \leq 0$, according to property \eqref{derivative-sign}.
Thus, \eqref{equ-2-proof-thm-bilinear-submarkovian} and \eqref{equ-1-proof-thm-bilinear-submarkovian} are shown, and Theorem \ref{thm-bilinear-submarkovian} follows in the case of the alternative assumptions 1 or 2.

In case of the alternative assumption 3., we proceed as in the proof of Theorem \ref{thm-bilinear}.
We define on $\Omega$, $v = \frac1h \int_0^h S_s(w')^{-1} ds  =  \frac1h \int_0^h T_s(w^{-1}) + 1 - T_s(1) ds$ and $\tilde{w} = \frac1h \int_0^h S_r(w') dr = \frac1h \int_0^h T_rw + 1 - T_r(1) dr$.
Note that $v|_\Omega ,\tilde{w}|_\Omega$ belong to $D(A_\infty)$ since $1 \in D(A_\infty)$ by assumption.
As in the proof of Theorem \ref{thm-bilinear}, the local diffusion property implies that
$c 1 \leq S_t v S_t(\tilde{w}) \leq C \tilde{Q}^A_2(w)$.
The rest of the proof goes along the same lines as the end of the proof of Theorem \ref{thm-bilinear}, with the modified functional $\mathcal{E}(t) = \int_\Omega B(T_tf,T_tg, S_tv, S_t\tilde{w}) d\mu$ as above.
\end{proof}

\begin{cor}
\label{cor-bilinear-submarkovian}
Let $(T_t)_{t \geq 0}$ be a submarkovian semigroup on some $\sigma$-finite measure space $(\Omega,\mu)$.
Assume one of the following alternative conditions.
\begin{enumerate}
\item The measure space is finite, $\mu(\Omega) < \infty$, or
\item For any $t > 0$, $T_t$ maps $L^\infty(\Omega)$ into the domain $D(A_\infty)$ of the $w^*$ $L^\infty$ realization of $A$, or
\item The amplified semigroup $(S_t)_{t \geq 0}$ on $\Omega'$ satisfies the local diffusion from Definition \ref{defi-local-diffusion} and $1_\Omega \in D(A_\infty)$.
\end{enumerate}
In case 2. and 3. above, assume moreover that for any $v \in L^\infty(\Omega)$, $T_tv(x) \to v(x)$ as $t \to 0+$ $\mu$-almost everywhere.
Let $J > 1$.
Then there exists a constant $C_J$ depending only on $J$ such that for any weight $w : \Omega \to (0,\infty)$ with $\tilde{Q}^A_2(w) < \infty$, we have
\[ \|m(A)\|_{L^2(\Omega,wd\mu) \to L^2(\Omega,wd\mu)} \leq C_J \tilde{Q}^A_2(w) \left( |m(0)| + \|m\|_{\HI(\Sigma_{\frac{\pi}{2}};J)} \right) . \]
\end{cor}

\begin{proof}
Copy the proof of Corollary \ref{cor-bilinear}, Theorem \ref{thm-bilinear-submarkovian} replacing Theorem \ref{thm-bilinear}.
\end{proof}

\begin{prop}
\label{prop-submarkovian-maximal-regularity}
Let $(T_t)_{t \geq 0}$ be a submarkovian semigroup on $(\Omega,\mu)$.
Assume one of the following alternative conditions.
\begin{enumerate}
\item The measure space is finite, $\mu(\Omega) < \infty$, or
\item For any $t > 0$, $T_t$ maps $L^\infty(\Omega)$ into the domain $D(A_\infty)$ of the $w^*$ $L^\infty$ realization of $A$, or
\item The amplified semigroup $(S_t)_{t \geq 0}$ on $\Omega'$ satisfies the local diffusion from Definition \ref{defi-local-diffusion} and $1_\Omega \in D(A_\infty)$.
\end{enumerate}
In case 2. and 3. above, assume moreover that for any $v \in L^\infty(\Omega)$, $T_tv(x) \to v(x)$ as $t \to 0+$ $\mu$-almost everywhere.
Assume that the weight $w$ satisfies $w^\delta \in \tilde{Q}^A_2$ for some $\delta > 1$.
Then $A$ has an $\HI(\Sigma_\theta)$ calculus on $L^2(\Omega,wd\mu)$ for some $\theta < \frac{\pi}{2}$ and in particular, the analytic semigroup $T_z$ extends boundedly to $L^2(\Omega,wd\mu)$ for $|\arg z| < \frac{\pi}{2} - \theta$, and $A$ has maximal regularity on $L^2(\Omega,wd\mu)$.
\end{prop}

\begin{proof}
The proof is the same as that of Proposition \ref{prop-maximal-regularity}.
\end{proof}

\section{Negative results: tensor powers of the two-point semigroup}
\label{sec-negative-results}

Our result from Corollary \ref{cor-bilinear} showed that any markovian semigroup satisfying technical conditions has an $\HI(\Sigma_\theta)$ calculus on weighted $L^2$ space for any $\theta > \frac{\pi}{2}$ and any $Q^A_2$ weight $w$, and that moreover, the dependence of the norm of this $\HI$ calculus is linear in the $Q^A_2(w)$ constant.
The question arises whether the angle $\theta$ can be lowered in this result.
In this section, we show the partial negative result in Theorem \ref{thm-no-Hormander} below.
We recall here the definition of the H\"ormander spectral multiplier class for a parameter $s > 0$:
\begin{equation}
\label{equ-def-Hormander}
\mathcal{H}^s = \left\{ m \in L^1_{\mathrm{loc}}(\R_+) : \: \|m\|_{\mathcal{H}^s} = \sup_{t > 0} \| \eta m(t \cdot)\|_{W^\infty_s(\R)} < \infty \right\},
\end{equation}
where $W^2_\infty(\R)$ stands for the usual Sobolev space and $\eta$ is any non-zero cut-off function from $C^\infty_c(0,\infty)$.
Note that $\mathcal{H}^s$ functional calculus is related to $H^\infty(\Sigma_\theta)$ functional calculus for angles $\theta \to 0$ according to \cite[Theorem 4.10]{CDMcIY}.
Namely, if a sectorial operator $A$ has a $\mathcal{H}^s$ functional calculus for some fixed $s > 0$, then it has a $\HI(\Sigma_\theta)$ calculus for any $\theta \in (0, \pi)$, and for any $s' > s$ there is a constant $C > 0$ such that
\begin{equation}
\label{equ-Hor-410}
\|m(A)\| \leq C \theta^{-s'} \|m\|_{H^\infty(\Sigma_\theta)} \quad (m \in H^\infty_0(\Sigma_\theta),\: \theta \in (0,\pi)).
\end{equation}
Conversely, \eqref{equ-Hor-410} implies that $A$ has a $\mathcal{H}^s$ calculus for any $s > s'$.
Thus, Theorem \ref{thm-no-Hormander} below can be read as a failure of the weighted functional calculus from Corollary \ref{cor-bilinear} when the angle $\theta$ is close to $0$.

\begin{thm}
\label{thm-no-Hormander}
There exists a markovian semigroup $T_t = \exp(-tA)$ on a probability space and a $Q^A_2$ weight $w$ such that $A$ does not have a H\"ormander $\mathcal{H}^s$ calculus for any $s > 0$ on weighted $L^2(w)$ space, that is, for no $s > 0$ and no $C > 0$ the estimate
\begin{equation}
\label{equ-1-thm-no-Hormander}
\|m(A)\|_{L^2(\Omega,w d\mu) \to L^2(\Omega,w d\mu)} \leq C \left( \|m\|_{\mathcal{H}^s} + |m(0)| \right) \quad (m \in \mathcal{H}^s)
\end{equation}
holds.
In fact, \eqref{equ-1-thm-no-Hormander} does not even hold for $m(\lambda) = \exp(-\lambda z)$ with $z \in \Sigma_{\frac{\pi}{2}}$.
\end{thm}

We remind that the above theorem is in contrast with the positive result in Remark \ref{rem-DSY-GY} for self-adjoint semigroups on spaces of homogeneous type satisfying Gaussian estimates and a restricted weight class.

The semigroup exhibiting the counter-example for the statement \eqref{equ-1-thm-no-Hormander} is based on the two-point semigroup that we have already encountered in Proposition \ref{prop-toy-example-estimate}, together with a tensor power extension of the semigroup.
So we consider a two point space $\Omega_0 = \{ a , b \}$ equipped with counting measure $\mu_0 = \delta_a + \delta_b$.
Consider moreover the operator 
\[ \Gc = \begin{bmatrix} 1 & -1 \\ -1 & 1 \end{bmatrix} \co L^2(\Omega_0,\mu_0) \to L^2(\Omega_0,\mu_0) \]
which generates the markovian semigroup
\begin{equation}
\label{equ-two-point-semigroup}
\exp(-t\Gc) = \frac12 \begin{bmatrix} 1 + e^{-2t} & 1 - e^{-2t} \\ 1 - e^{-2t} & 1 + e^{-2t} \end{bmatrix}.
\end{equation}

\begin{lemma}
\label{lem-tensor-power-extension}
For any $n \in \N$, the above semigroup admits a tensor power extension to a markovian semigroup in the following way.
We let $w_1 =(u_1,v_1),\ldots,w_n = (u_n,v_n)$ be weights on $\Omega_0$.
Then we put
\begin{align*}
\Omega & = \underset{n \text{ factors}}{\underbrace{\Omega_0 \times \Omega_0 \times \ldots \times \Omega_0}} = \Omega_0^n, \\
\mu & = \underset{n \text{ factors}}{\underbrace{\mu_0 \otimes \mu_0 \otimes \ldots \otimes \mu_0}} = \mu_0^{\otimes n}, \\
T_t & =  \underset{n \text{ factors}}{\underbrace{e^{-t\Gc} \otimes e^{-t\Gc} \otimes \ldots \otimes e^{-t\Gc}}}, \\
w & = w_1 \otimes w_2 \otimes \ldots \otimes w_n,
\end{align*}
where $T_t(\sum_k f_1^{(k)} \otimes \ldots \otimes f_n^{(k)}) = \sum_k (e^{-t\Gc} (f_1^{(k)})) \otimes \ldots \otimes (e^{-t\Gc}(f_n^{(k)}))$ and $w(x_1,\ldots,x_n) = w_1(x_1) \cdot \ldots \cdot w_n(x_n)$.
We have that $T_t$ is a markovian semigroup on $(\Omega,\mu)$.
Moreover,
\begin{align}
\|T_z\|_{L^2(\Omega,w d\mu) \to L^2(\Omega,w d\mu)}& \geq \prod_{k = 1}^n \|e^{-z\Gc}\|_{L^2(\Omega_0,w_kd\mu_0) \to L^2(\Omega_0,w_k d\mu_0)}, \label{equ-1-prop-tensor-power-extension}\\
Q^A_2(w) & = \prod_{k = 1}^n Q^{\Gc}_2(w_k) = \prod_{k = 1}^n \frac14 \left(2 + \frac{u_k}{v_k} + \frac{v_k}{u_k} \right), \label{equ-2-prop-tensor-power-extension}
\end{align}
where $Q^A_2$ stands for the weight characteristics with respect to the markovian semigroup $T_t$.
\end{lemma}

\begin{proof}
It is easy to check that $t \mapsto T_t$ satisfies the semigroup property.
Since all of the four entries of $\exp(-t\Gc)$ in \eqref{equ-two-point-semigroup} are positive for any $t \geq 0$, all the entries of $T_t$ are positive too, and $T_t$ is positive.
Moreover, $T_t(1 \otimes 1 \otimes \ldots \otimes 1) = e^{-t\Gc}(1) \otimes \ldots \otimes e^{-t\Gc}(1) = 1 \otimes \ldots \otimes 1$, so that $T_t$ is $L^\infty$ contractive.
Self-adjointness of $T_t$ is again easy, so that $T_t$ is $L^1$ contractive and finally $L^p$ contractive for all $p \in [1,\infty]$.
We infer that $T_t$ is a markovian semigroup on $(\Omega,\mu)$.
For the two claimed estimates, we observe that for a normalised function $f_k \in L^2(\Omega,w_k d\mu_0)$ such that $\|e^{-z\Gc}\|_{L^2(w_k) \to L^2(w_k)} = \|e^{-z\Gc}f_k\|_{L^2(w_k)}$, we have
\[ \|T_z(f_1 \otimes \ldots \otimes f_k)\|_{L^2(w)} = \prod_{k = 1}^n \|e^{-z\Gc}f_k\|_{L^2(w_k)} = \prod_{k=1}^n \|e^{-z\Gc}\|_{L^2(w_k)} \]
and $\|f_1 \otimes \ldots \otimes f_k\|_{L^2(w)} = \prod_{k = 1}^n \|f_k\|_{L^2(w_k)} = 1$.
Thus, \eqref{equ-1-prop-tensor-power-extension} follows.
For \eqref{equ-2-prop-tensor-power-extension}, we note
\begin{align*}
Q^A_2(w) & = \sup_{t > 0} \|T_tw T_t(w^{-1})\|_\infty \\
& = \sup_{t > 0} \|e^{-t\Gc}w_1 \otimes \ldots \otimes e^{-t\Gc} w_n \cdot e^{-t\Gc}(w_1^{-1}) \otimes \ldots \otimes e^{-t\Gc}(w_n^{-1}) \|_\infty \\
& = \sup_{t > 0} \prod_{k = 1}^n \|e^{-t\Gc}(w_k) e^{-t\Gc}(w_k^{-1})\|_\infty .
\end{align*}
We claim that each of the $n$ $L^\infty$ norms above attains its supremum for $t = \infty$, so that we can swap $\sup_{t > 0}$ and $\prod_{k = 1}^n$ above and thus deduce the first equality in \eqref{equ-2-prop-tensor-power-extension}.
Indeed, if $w_0 = (u,v) \in \R^2$ with $u,v > 0$, then $e^{-t\Gc}w_0 = \frac12 \left(u+v+e^{-2t}(u-v),u+v+e^{-2t}(v-u)\right)$, so that the first of the two coordinates of $e^{-t\Gc}w_0 \cdot e^{-t\Gc}(w_0^{-1})$ equals
\begin{align*}
& \frac14 \left(u+v+e^{-2t}(u-v)\right)\left(\frac{1}{u} + \frac{1}{v} + e^{-2t}\left(\frac{1}{u} - \frac{1}{v}\right)\right) \\
& =  \frac14 \left(2 + \frac{u}{v} + \frac{v}{u} + e^{-2t}(u-v)\left(\frac{1}{u} + \frac{1}{v}\right) + e^{-2t}\left(\frac{1}{u} - \frac{1}{v}\right)(u+v) + e^{-4t}(u-v)\left(\frac{1}{u} - \frac{1}{v}\right)\right) \\
& = \frac14 \left( 2 + \frac{u}{v} + \frac{v}{u} + e^{-4t}\underset{\leq 0}{\underbrace{\left(2 - \frac{u}{v} - \frac{v}{u}\right)}} \right).
\end{align*}
The last quantity clearly attains its sup at $t = \infty$ where its value is $\frac14 (2 + \frac{u}{v} + \frac{v}{u})$.
Now the same calculation with exchanged roles of $u$ and $v$ works for the second coordinate.
We deduce the first equality in \eqref{equ-2-prop-tensor-power-extension}, and in fact also the second equality.
\end{proof}

The tensor power extension of Lemma \ref{lem-tensor-power-extension} will be used to bootstrap lower estimates for the two-point semigroup on weighted $L^2$ space, to a lower estimate of the same kind, but with a better constant.
We shall now establish such lower bounds on $\|e^{-z \Gc}\|_{L^2(w) \to L^2(w)}$ (in terms of $Q^{\Gc}_2(w) = \frac14 \left( 2 + \frac{u}{v} + \frac{v}{u} \right)$ ).

\begin{prop}
\label{prop-two-point-heart-estimate}
Consider the markovian semigroup $e^{-t\Gc}$ on $(\Omega_0,\mu_0)$ from \eqref{equ-two-point-semigroup}.
Let $w_0 = (1,v^2)$ be a weight on $\Omega_0$ with $v = 1 + \epsi$ for some small $\epsi > 0$.
Let $z \in \C_+ \backslash [0,\infty)$.
Then we have the following asymptotic formula for weighted norm of the analytic semigroup:
\begin{equation}
\label{equ-prop-two-point-heart-estimate}
\| T_z \|_{L^2(\Omega_0,w_0 d\mu_0) \to L^2(\Omega_0,w_0 d\mu_0)} = 1 +\frac{1}{16}  \left(4 |1 - \gamma|^2 + \frac12 d_\gamma\right) \epsi^2 + o(\epsi^2).
\end{equation}
where $\gamma = e^{-2z}$ and $\displaystyle d_\gamma = \frac{|1-\gamma|^4 + (1 - |\gamma|^2)^2 + 4 \Im(\gamma)^2}{(1-|\gamma|^2)^2}$.
\end{prop}

\begin{proof}
First we rewrite the norm on weighted $L^2$ space into a norm on unweighted $L^2$ space.
To this end, consider the multiplication operator $M_{\sqrt{w_0}} \co L^2(\Omega_0,\mu_0) \to L^2(\Omega_0,\mu_0), \: (f_1,f_2) \mapsto (f_1,v f_2)$.
Then for any operator $T$ on $L^2(\Omega_0,w_0 d\mu_0)$, we have
\begin{align*}
\| T \|_{L^2(\Omega_0,w_0 d\mu_0) \to L^2(\Omega_0,w_0 d\mu_0)} & = \sup_{f \neq 0} \frac{\| T f\|_{L^2(\Omega_0,w_0 d\mu_0)}}{\|f\|_{L^2(\Omega_0,w_0 d\mu_0)}} \\ 
& = \sup_{f \neq 0} \frac{ \| M_{\sqrt{w_0}} T f \|_{L^2(\Omega_0,\mu_0)}}{\|M_{\sqrt{w_0}} f\|_{L^2(\Omega_0,\mu_0)}} \\
& = \sup_{f \neq 0} \frac{ \| M_{\sqrt{w_0}} T (M_{\sqrt{w_0}})^{-1} f \|_{L^2(\Omega_0,\mu_0)}}{\|f\|_{L^2(\Omega_0,\mu_0)}}, 
\end{align*}
so that $\|T\|_{L^2(\Omega_0,w_0 d\mu_0) \to L^2(\Omega_0,w_0 d\mu_0)} = \| M_{\sqrt{w_0}} T (M_{\sqrt{w_0}})^{-1} f \|_{L^2(\Omega_0,\mu_0) \to L^2(\Omega_0,\mu_0)}.$
Write $S = M_{\sqrt{w_0}} T (M_{\sqrt{w_0}})^{-1}$.
The norm of $S$ on unweighted space is well-known to be $\sqrt{m}$, where $m$ denotes the maximal eigenvalue of $S^*S$.
We have with $T = T_z$ and $\gamma = e^{-2z}$
\[ S = \begin{bmatrix} 1 & 0 \\ 0 & v \end{bmatrix} \cdot \frac12 \begin{bmatrix} 1 + \gamma & 1 - \gamma \\ 1 - \gamma & 1 + \gamma \end{bmatrix} \begin{bmatrix} 1 & 0 \\ 0 & \frac{1}{v} \end{bmatrix} = \frac12 \begin{bmatrix} 1 + \gamma & \frac{1}{v} ( 1 - \gamma ) \\ v ( 1 - \gamma ) & 1 + \gamma \end{bmatrix} ,\]
so
\[ S^* S = \frac14 \begin{bmatrix} |1+\gamma|^2 + v^2 |1-\gamma|^2 & \beta \\ \ovl{\beta} & | 1 + \gamma |^2 + \frac{1}{v^2} | 1 - \gamma |^2 \end{bmatrix}, \]
with $\beta = \frac{1}{v} (1 + \ovl{\gamma})(1-\gamma) + v (1 - \ovl{\gamma})(1+ \gamma) = (v+\frac{1}{v})(1-|\gamma|^2) + 2i(v-\frac{1}{v}) \Im(\gamma)$.
The eigenvalues of a positive matrix $\begin{pmatrix} \alpha & \beta \\ \ovl{\beta} & \delta \end{pmatrix}$ are
\[  \frac12 \left(\alpha + \delta \pm \sqrt{ ( \alpha - \delta)^2 + 4 |\beta|^2 }\right), \]
so that we obtain with choice of sign ``$+$'' here that
\begin{align*}
& \|S^* S\|_{L^2(\Omega_0,\mu_0) \to L^2(\Omega_0,\mu_0)} \\
& = \frac18 \left( 2 |1 + \gamma|^2 + (v^2 + \frac{1}{v^2}) |1 - \gamma|^2 + \sqrt{ (v^2 - \frac{1}{v^2})^2 |1 - \gamma|^4 + 4 (v + \frac{1}{v})^2(1 - |\gamma|^2)^2 + 4 \cdot 4 (v-\frac{1}{v})^2 \Im(\gamma)^2 } \right) \\
& = \frac18 \left( 2 |1+\gamma|^2 + (2 + 4\epsi^2 + o(\epsi^2))| 1 - \gamma|^2 \right. \\
& \left. + \sqrt{(16 \epsi^2 + o(\epsi^2))|1 - \gamma|^4 + 4(4+4 \epsi^2 + o(\epsi^2))(1- |\gamma|^2)^2 + 16\cdot (4 \epsi^2  + o(\epsi^2)) \Im(\gamma)^2 } \right) \\
& = \frac18 \left( 2 |1 + \gamma|^2 + 2 |1 - \gamma|^2 + (4\epsi^2 + o(\epsi^2))|1-\gamma|^2 \right. \\
& \left. + \sqrt{16 (1 - |\gamma|^2)^2 + (\epsi^2 + o(\epsi^2))(16 |1 - \gamma|^4 + 16 ( 1 - |\gamma|^2)^2 + 16 \cdot 4 \Im(\gamma)^2) } \right),
\end{align*}
where we have used the following elementary Taylor series expansions (recall $v = 1 + \epsi$)
\begin{align*}
v^2 + \frac{1}{v^2} & = 2 + 4 \epsi^2 + o(\epsi^2), \\
(v^2 - \frac{1}{v^2})^2 & = 16 \epsi^2 + o(\epsi^2), \\
(v + \frac{1}{v})^2 & = 4 + 4 \epsi^2 + o(\epsi^2), \\
(v - \frac{1}{v})^2 & = 4 \epsi^2 + o(\epsi^2).
\end{align*}
Write in short 
\[d_\gamma = \frac{|1-\gamma|^4 + (1 - |\gamma|^2)^2 + 4 \Im(\gamma)^2}{(1-|\gamma|^2)^2} .\]
Then the above calculation continues with
\begin{align*}
\|S^*S\| & = \frac18 \left( 2 |1 + \gamma|^2 + 2 |1 - \gamma|^2 + (4\epsi^2 + o(\epsi^2))|1-\gamma|^2  + 4 (1-|\gamma|^2) \sqrt{1 + (\epsi^2 + o(\epsi^2))d_\gamma} \right) \\
& = \frac18 \left( 2 |1 + \gamma|^2 + 2 |1 - \gamma|^2 + (4 \epsi^2 + o(\epsi^2))|1-\gamma|^2 + 4(1 - |\gamma|^2)\left(1 + \frac12 d_\gamma (\epsi^2 + o(\epsi^2))\right) \right) \\
& = 1 + \frac18 \left(4 |1 - \gamma|^2 + \frac12 d_\gamma\right)(\epsi^2 + o(\epsi^2)),
\end{align*}
where we have used that $2 | 1 + \gamma|^2 + 2 | 1 - \gamma|^2 + 4(1-|\gamma|^2) = 8$.
Take now the square root of $\|S^*S\|$ and use (again) the asymptotics $\sqrt{1+x} = 1 + \frac12 x + o(x)$ to obtain \eqref{equ-prop-two-point-heart-estimate}.
\end{proof}

We now use \eqref{equ-prop-two-point-heart-estimate} to produce a lower estimate taylored for an application to Theorem \ref{thm-no-Hormander}.

\begin{lemma}
\label{lem-arctan31}
Consider again the two-point semigroup from \eqref{equ-two-point-semigroup}, and a weight $w = (1,v^2)$ with $v = 1 + \epsi$.
Then we have for $z = r e^{i \phi}$ with $r > 0$ and $\phi \in \left( - \frac{\pi}{2}, \frac{\pi}{2} \right)$,
\begin{equation}
\label{equ-1-lem-arctan31}
\|e^{-z \Gc}\|_{L^2(w) \to L^2(w)} = 1 + \frac{1}{32}\left(1 + \tan^2(\phi) + o_r(1)\right) \epsi^2 +o_\epsi(\epsi^2). 
\end{equation}
\end{lemma}

\begin{proof}
Since we shall have both asymptotics in $r$ and $\epsi$, we distinguish the little $o$ notations $o_r$ and $o_\epsi$.
In the course of the proof we shall pick $z = re^{i\phi}$, choose $r$ sufficiently close to $0$ and use \eqref{equ-prop-two-point-heart-estimate}.
Namely, we have for $r$ close to $0$, $\gamma = e^{-2 r e^{i\phi}} = 1 - 2r e^{i\phi} + o_r(r)$.
Then $|1 - \gamma| = 2r + o_r(r)$ and $1 - |\gamma|^2 = 1 - |1 - 2r e^{i\phi} + o_r(r)|^2 = 1 - (1 - 4r \cos(\phi) + o_r(r) ) = 4r \cos(\phi) + o_r(r)$, as well as $\Im(\gamma) = -2r \sin(\phi) + o_r(r)$.
This yields
\begin{align*}
d_\gamma & = \frac{16r^4 + 16 r^2 \cos^2(\phi) + 16 r^2 \sin^2(\phi) + o_r(r^2)}{16 r^2 \cos^2(\phi) + o_r(r^2)} \\
& = \frac{r^2}{ \cos^2(\phi)} + 1 + \tan^2(\phi) + o_r(1) = 1 + \tan^2(\phi) + o_r(1).
\end{align*}
Thus, in view of \eqref{equ-prop-two-point-heart-estimate}, we have
\begin{align*}
\|e^{-z \Gc}\|_{L^2(w) \to L^2(w)} & = 1 + \frac{1}{16}\left(16 r^2 + \frac12 \cdot\left( 1 + \tan^2(\phi) \right)  + o_r(1)\right) \epsi^2 + o_\epsi(\epsi^2) \\
& = 1 + \frac{1}{32}\left(1 + \tan^2(\phi) + o_r(1)\right) \epsi^2 +o_\epsi(\epsi^2). 
\end{align*}
\end{proof}

Putting the above intermediate results together, we are now in a position to prove the main result of this section.

\begin{proof}[of Theorem \ref{thm-no-Hormander}]
We take the spectral multiplier $m(\lambda) = \exp(-\lambda z)$ with $z  = re^{i\phi}$, where $r > 0$ (resp. $\phi \in (0,\frac{\pi}{2})$) is sufficiently close to $0$ (resp. to $\frac{\pi}{2}$) to be determined later.
The counterexample will be the direct sum of tensor powers of the two-point semigroup as in Lemma \ref{lem-tensor-power-extension}.
We pick a sequence of weights $w_n = \left(1, (1 + \epsi_n)^2 \right)$.
We note first that there exists a constant $C > 0$ such that for any $\epsi \geq 0$, we have $Q^{\Gc}_2(w) \leq 1 + C \epsi^2$ for the weight $w = \left(1 , (1 + \epsi)^2 \right)$.
Indeed, we already know this if $\epsi \leq \epsi_0$ for a certain $0 < \epsi_0 \ll 1$,  from the asymptotics $Q^{\Gc}_2(w) = 1 + \epsi^2 + o(\epsi^2)$.
Then for $\epsi \geq \epsi_0$, we have $Q^{\Gc}_2(w) = \frac14 \left( 2 + (1 + \epsi)^2 + \frac{1}{(1 + \epsi)^2}\right) \leq \frac14 \left( 2 + 1 + 2 \epsi + \epsi^2 + 1 \right) \leq  \frac14 \left( 2 + 1 + \frac{2}{\epsi_0} \epsi^2 + \epsi^2 + 1 \right) \leq 1 + \frac14\left(\frac{2}{\epsi_0} + 1 \right) \epsi^2$.
From this, we deduce for the tensor power weight $w = w_1 \otimes \ldots \otimes w_n$ from Lemma \ref{lem-tensor-power-extension} that 
\[\log(Q^A_2(w)) \leq \sum_{k = 1}^n \log(1 + C \epsi_k^2) \leq \sum_{k = 1}^n C \epsi_k^2 .
\]
Now assume that for a given $N \in \N$ the sequence $(\epsi_k)_{k \in \N} = (\epsi^{(N)}_k)_{k \in \N}$ satisfies $\epsi^{(N)}_k = \begin{cases} \frac{1}{\sqrt{N}} & k \leq N \\ 0 & k > N \end{cases}$.
Then according to the above, the associated weight $w^{(N)} = w_1^{(N)} \otimes \ldots \otimes w_N^{(N)}$ satisfies 
\begin{equation}
\label{equ-3-proof-thm-no-Hormander}
Q^A_2(w) \leq e^{C \sum_{k = 1}^N (\epsi_k^{(N)})^2} \leq Q,
\end{equation}
where $Q$ can be chosen independent of $N$.
Take now the semigroup $T_t = T_t^{(N)}$ associated to $\Omega_0^N$, $\mu_0^{\otimes N}$ and $w^{(N)}$ as in Lemma \ref{lem-tensor-power-extension}.
We estimate with this lemma together with \eqref{equ-1-lem-arctan31}
\[ \|T_z^{(N)}\|_{L^2(w^{(N)}) \to L^2(w^{(N)})} \geq \prod_{k = 1}^N \left(1 + \frac{1}{32} (  1 + \tan^2(\phi) + o_r(1) ) \left(\epsi_k^{(N)}\right)^2 + o_\epsi\left( \left(\epsi_k^{(N)}\right)^2 \right) \right). \]
Choose $r$ sufficiently close to $0$ to have $o_r(1) \geq -1$ here above.
Moreover, for given $z = re^{i\phi}$, choose $N$ so large that $\tan^2(\phi) \left(\epsi_k^{(N)}\right)^2 \leq 1$, i.e. $N \geq \tan^2(\phi)$, and that $o_\epsi((\epsi_k^{(N)})^2)$ is in force.
Then we obtain

\begin{align}
\log\left(\|T_z^{(N)}\|_{L^2(w^{(N)}) \to L^2(w^{(N)})} \right) & \geq \sum_{k = 1}^N \log \left(1 + \frac{1}{32} \tan^2(\phi) \left( \epsi_k^{(N)} \right)^2 + o_\epsi\left( \left( \epsi_k^{(N)} \right)^2 \right) \right) \nonumber \\
& \gtrsim \sum_{k = 1}^N \frac{1}{32} \tan^2(\phi) \left(\epsi_k^{(N)}\right)^2 \nonumber \\
& = \frac{1}{32} \tan^2(\phi). \label{equ-1-proof-thm-no-Hormander}
\end{align}

Take now the direct sum $\Omega = \bigsqcup_{N \in \N} \Omega_0^N$ equipped with the sum measure $\mu = \bigoplus_{N \in \N} \frac{1}{2^{2N}} \mu_0^{\otimes N}$.
Note that each $\mu_0^{\otimes N}$ has total mass $2^N$, so that $\mu$ is a probability measure.
Take moreover the weight $w = \bigoplus_{N \in \N} w^{(N)}$, that is, on each $N$-component, we have  $w|_{\Omega_0^N} = w^{(N)}$.
We also take the direct sum semigroup $T_tf(x_N) = T_t^{(N)}(f|_{\Omega_0^N})(x_N)$, where $T_t^{(N)}$ is as above and $f \co \bigsqcup_{N \in \N} \Omega_0^N \to \C, \: x_N \mapsto f(x_N)$.
It is not hard to check that $(T_t)_t$ is again markovian and that with respect to this semigroup, we have $Q^A_2(w) \leq \sup_{N \in \N} Q^{A^{(N)}}_2(w^{(N)}) \leq Q < \infty$ according to \eqref{equ-3-proof-thm-no-Hormander}.
Moreover, for given $z = r e^{i \phi}$ with fixed $r$ sufficiently close to $0$ as above, but $\phi$ varying and approaching $\frac{\pi}{2}$, we have for $N \geq \tan^2(\phi)$
\begin{equation}
\label{equ-2-proof-thm-no-Hormander}
\|T_z\|_{L^2(w) \to L^2(w)} \geq \|T_z^{(N)}\|_{L^2(w^{(N)}) \to L^2(w^{(N)})} \overset{\eqref{equ-1-proof-thm-no-Hormander}}{\geq} e^{c \tan^2(\phi)},
\end{equation}
where $c$ does not depend on $\phi$.

Now if this markovian semigroup $T_t$ had a weighted H\"ormander calculus on $L^2(\Omega,w d\mu)$, we would have
$ \|T_z\|_{L^2(w) \to L^2(w)} \leq C(w) \|\lambda \mapsto e^{- \lambda z} \|_{\mathcal{H}^s} $, and the last quantity is bounded by $(\frac{\pi}{2} - |\phi|)^{-s} \cong \left|\tan(\phi)\right|^s$ according to \cite[Lemma 3.9 (1)]{KrW3}.
But no inequality $e^{c \tan^2(\phi)} \lesssim \left| \tan(\phi) \right|^s$ can hold for all $\phi \in \left( - \frac{\pi}{2}, \frac{\pi}{2} \right)$,
so that we get a contradiction from \eqref{equ-2-proof-thm-no-Hormander}.
\end{proof}

\begin{remark}
Comparing Corollary \ref{cor-bilinear} and Theorem \ref{thm-no-Hormander}, the question arises if for a certain $\theta$ between $0$ and $\frac{\pi}{2}$ there is a $H^\infty(\Sigma_\theta)$ calculus result for any markovian semigroup on $L^2(\Omega,wd\mu)$ and any $Q^A_2$ weight $w$.
We say that $\theta \in (0,\pi)$ is a universal angle for weighted $L^2$ calculus if for any markovian semigroup satisfying the technical hypotheses of Corollary \ref{cor-bilinear} and any $Q^A_2$ weight $w$, there is a constant $C_w > 0$ such that
\[ \|m(A)\|_{L^2(\Omega,w d\mu) \to L^2(\Omega,w d\mu)} \leq C_w \|m\|_{H^\infty(\Sigma_\theta)} \quad (m \in H^\infty_0(\Sigma_\theta)) . \]
According to Corollary \ref{cor-bilinear}, any $\theta > \frac{\pi}{2}$ is a universal angle for weighted $L^2$ calculus.
We conjecture that no angle $\theta < \frac{\pi}{2}$ is a universal angle.
\end{remark}

\subsubsection*{Acknowledgment}
The second author acknowledges support by the grant ANR-17-CE40-0021 of the French National Research Agency ANR (project Front) and by the grant ANR-18-CE40-0021 (project HASCON). All authors acknowledge support by the European Research Council grant DLV-862402 (CHRiSHarMa). The third author acknowledges support by the Alexander von Humboldt Stiftung.


\vspace{0.2cm}
\footnotesize{
\noindent Komla Domelevo\\
Institut f{\"u}r Mathematik\\
Universit{\"a}t W{\"u}rzburg\\
Emil-Fischer-Stra{\ss}e 40\\
97074 W{\"u}rzburg\\
komla.domelevo@mathematik.uni-wuerzburg.de\\

\noindent Christoph Kriegler\\
Laboratoire de Math\'ematiques Blaise Pascal (UMR 6620)\\
Universit\'e Clermont Auvergne\\
63 000 Clermont-Ferrand, France \\
URL: \href{http://math.univ-bpclermont.fr/~kriegler/indexenglish.html}{http://math.univ-bpclermont.fr/{\raise.17ex\hbox{$\scriptstyle\sim$}}\hspace{-0.1cm} kriegler/indexenglish.html}\\
christoph.kriegler@uca.fr\\

\noindent Stefanie Petermichl\\
Institut f{\"u}r Mathematik\\
Universit{\"a}t W{\"u}rzburg\\
Emil-Fischer-Stra{\ss}e 40\\
97074 W{\"u}rzburg\\
stefanie.petermichl@mathematik.uni-wuerzburg.de
}


\begin{thebibliography}{CDMcIY}

\bibitem{AFLM}
C. Arhancet, S. Fackler and C. Le Merdy.
\newblock Isometric dilations and $\HI$ calculus for bounded analytic semigroups and Ritt operators.
\newblock Trans. Amer. Math. Soc. 369 (2017), no. 10, 6899--6933.

\bibitem{AHLLMT}
P. Auscher, S. Hofmann, M. Lacey, J. Lewis, A. McIntosh and P. Tchamitchian.
\newblock The solution of Kato's conjectures.
\newblock C. R. Acad. Sci. Paris S\'er. I Math. 332 (2001), no. 7, 601--606.

\bibitem{AuTc}
P. Auscher and P. Tchamitchian.
\newblock Square root problem for divergence operators and related topics.
\newblock  Ast\'erisque No. 249 (1998), viii+172 pp.


\bibitem{BeL}
J. Bergh and J. L\"ofstr\"om.
\newblock Interpolation spaces. An introduction.
\newblock Grundlehren der Mathematischen Wissenschaften, No. 223. Springer-Verlag, Berlin-New York, 1976. x+207 pp.

\bibitem{BlKu}
S. Blunck and P. Kunstmann.
\newblock Calder\'on-Zygmund theory for non-integral operators and the $\HI$ functional calculus.
\newblock Rev. Mat. Iberoamericana 19 (2003), no. 3, 919--942.

\bibitem{primer}
B. B\"ottcher, R. Schilling and J. Wang.
\newblock L\'evy matters. III. Chapter 1.
\newblock L\'evy-type processes: construction, approximation and sample path properties.
Lecture Notes in Mathematics, 2099. L\'evy Matters. Springer, Cham, 2013. xviii+199 pp.

\bibitem{CaDr}
A. Carbonaro and O. Dragi\v{c}evi\'{c}.
\newblock Functional calculus for generators of symmetric contraction semigroups.
\newblock Duke Math. J. 166 (2017), no. 5, 937--974. 

\bibitem{CoWe}
R. Coifman and G. Weiss.
\newblock Some examples of transference methods in harmonic analysis.
\newblock Symposia Mathematica, Vol. XXII (Convegno sull'Analisi Armonica e Spazi di Funzioni su Gruppi Localmente Compatti, INDAM, Rome, 1976), pp. 33--45. Academic Press, London, 1977.

\bibitem{Cow}
M. Cowling.
\newblock Harmonic analysis on semigroups.
\newblock Ann. of Math. (2) 117 (1983), no. 2, 267--283.

\bibitem{CDMcIY}
M. Cowling, I. Doust, A. McIntosh and A.Yagi.
\newblock Banach space operators with a bounded $H^\infty$ functional calculus.
\newblock J. Austral. Math. Soc. 60 (1996), 51--89.

\bibitem{Dah}
K. Dahmani.
\newblock Sharp dimension free bound for the Bakry-Riesz vector.
\newblock preprint on https://arxiv.org/abs/1611.07696

\bibitem{DDHPV}
R. Denk, G. Dore, M. Hieber, J. Pr\"uss and A. Venni.
\newblock New thoughts on old results of R. T. Seeley.
\newblock Math. Ann. 328 (2004), no. 4, 545--583.

\bibitem{dS}
L. de Simon.
\newblock Un'applicazione della teoria degli integrali singolari allo studio delle equazioni differenziali lineari astratte del primo ordine.
\newblock Rend. Sem. Mat. Univ. Padova 34 (1964), 205--223.

\bibitem{DoPe}
K. Domelevo and S. Petermichl.
\newblock Differential Subordination under change of law.
\newblock preprint on https://arXiv.org/abs/1604.01606


\bibitem{Duo}
X. T. Duong.
\newblock $\HI$ functional calculus of second order elliptic partial differential operators on $L^p$ spaces. 
\newblock Miniconference on Operators in Analysis (Sydney, 1989), 91--102,
Proc. Centre Math. Anal. Austral. Nat. Univ., 24, Austral. Nat. Univ., Canberra, 1990.

\bibitem{DuRo}
X. T. Duong and D. Robinson.
\newblock Semigroup kernels, Poisson bounds, and holomorphic functional calculus.
\newblock J. Funct. Anal. 142 (1996), no. 1, 89--128.

\bibitem{DuOS}
X. T. Duong, E. M. Ouhabaz and A. Sikora.
\newblock Plancherel-type estimates and sharp spectral multipliers.
\newblock J. Funct. Anal. 196 (2002), no. 2, 443--485.

\bibitem{DSY}
X. T. Duong, A. Sikora and L. Yan.
\newblock Weighted norm inequalities, Gaussian bounds and sharp spectral multipliers.
\newblock Journal of Functional Analysis 260 (2011), 1106--1131.

\bibitem{Fen}
G. Fendler.
\newblock On dilations and transference for continuous one-parameter semigroups of positive contractions on $\mathcal{L}^p$-spaces.
\newblock Ann. Univ. Sarav. Ser. Math. 9 (1998), no. 1, iv+97 pp.

\bibitem{GCMMST}
J. Garc\'ia-Cuerva, G. Mauceri, S. Meda, P. Sj\"ogren and J. L. Torrea.
\newblock Functional calculus for the Ornstein-Uhlenbeck operator.
\newblock J. Funct. Anal. 183 (2001), no. 2, 413--450.

\bibitem{GCR}
J. Garc\'ia-Cuerva and J. L. Rubio de Francia.
\newblock Weighted norm inequalities and related topics.
\newblock North-Holland Mathematics Studies, 116. Notas de Matem\'atica [Mathematical Notes], 104. North-Holland Publishing Co., Amsterdam, 1985. x+604 pp.

\bibitem{Gig}
Y. Giga.
\newblock Domains of fractional powers of the Stokes operator in $L_r$ spaces.
\newblock Arch. Rational Mech. Anal. 89 (1985), no. 3, 251--265.

\bibitem{GY}
R. Gong and L. Yan.
\newblock Littlewood-Paley and spectral multipliers on weighted $L^p$ spaces.
\newblock J. Geom. Anal. 24 (2014), no. 2, 873--900.

\bibitem{Gri}
A. Grigor'yan.
\newblock Gaussian upper bounds for the heat kernel and for its derivatives on a Riemannian manifold.
\newblock Classical and modern potential theory and applications (Chateau de Bonas, 1993), 237--252,
NATO Adv. Sci. Inst. Ser. C Math. Phys. Sci., 430, Kluwer Acad. Publ., Dordrecht, 1994.

\bibitem{GrTe}
A. Grigor'yan and A. Telcs.
\newblock Two-sided estimates of heat kernels on metric measure spaces.
\newblock Ann. Probab. 40 (2012), no. 3, 1212--1284.

\bibitem{HaLe}
B. Haak and C. Le Merdy.
\newblock $\alpha$-admissibility of observation and control operators.
\newblock Houston J. Math. 31 (2005), no. 4, 1153--1167.

\bibitem{HaOu}
B. Haak and E. M. Ouhabaz.
\newblock Exact observability, square functions and spectral theory.
\newblock J. Funct. Anal. 262 (2012), no. 6, 2903--2927.

\bibitem{Haase}
M. Haase.
\newblock The functional calculus for sectorial operators.
\newblock Operator Theory: Advances and Applications, 169. Birkh\"auser Verlag, Basel (2006). xiv+392 pp.

\bibitem{HiPr}
M. Hieber and J. Pr\"uss.
\newblock Functional calculi for linear operators in vector-valued $L^p$-spaces via the transference principle.
\newblock Adv. Differential Equations 3 (1998), no. 6, 847--872.

\bibitem{HvNP}
T. Hyt\"onen, J. van Neerven and P. Portal.
\newblock Conical square function estimates in UMD Banach spaces and applications to $\HI$-functional calculi.
\newblock J. Anal. Math. 106 (2008), 317--351.

\bibitem{JLMX}
M. Junge, C. Le Merdy and Q. Xu.
\newblock $\HI$ functional calculus and square functions on noncommutative $L^p$-spaces.
\newblock Ast\'erisque No. 305 (2006), vi+138 pp.

\bibitem{KaWe1}
N. Kalton and L. Weis.
\newblock The $\HI$-calculus and sums of closed operators.
\newblock  Math. Ann. 321 (2001), no. 2, 319--345.

\bibitem{KaWe2}
N. Kalton and L. Weis.
\newblock The $\HI$-Functional Calculus and Square Function Estimates.
\newblock Preprint on https://arxiv.org/abs/1411.0472v2

\bibitem{KaKuWe}
N. Kalton, P. Kunstmann and L. Weis.
\newblock Perturbation and interpolation theorems for the $\HI$-calculus with applications to differential operators.
\newblock Math. Ann. 336 (2006), no. 4, 747--801.

\bibitem{Kri1}
C. Kriegler.
\newblock Analyticity angle for non-commutative diffusion semigroups.
\newblock J. Lond. Math. Soc. (2) 83 (2011), no. 1, 168--186.

\bibitem{KrPhD}
C. Kriegler.
\newblock Spectral multipliers, $R$-bounded homomorphisms and analytic diffusion semigroups.
\newblock online available at https://hal.archives-ouvertes.fr/tel-00461310v1

\bibitem{KrW3}
C. Kriegler and L. Weis.
\newblock Spectral multiplier theorems via $\HI$ calculus and $R$-bounds.
\newblock to appear in Math. Z.

\bibitem{KW04}
P. Kunstmann and L. Weis.
\newblock Maximal $L_p$-regularity for parabolic equations, Fourier multiplier theorems and $\HI$-functional calculus.
\newblock Functional analytic methods for evolution equations, 65--311, 
Lecture Notes in Math., 1855 (2004) Springer, Berlin.

\bibitem{LaLaMe}
F. Lancien, G. Lancien and C. Le Merdy.
\newblock A joint functional calculus for sectorial operators with commuting resolvents.
\newblock Proc. London Math. Soc. (3) 77 (1998), no. 2, 387--414.

\bibitem{LaMe}
G. Lancien and C. Le Merdy.
\newblock A generalized $\HI$ functional calculus for operators on subspaces of $L^p$ and application to maximal regularity.
\newblock Illinois J. Math. 42 (1998), no. 3, 470--480.

\bibitem{Leb}
N. N. Lebedev.
\newblock Special functions and their applications.
\newblock Revised English edition. Translated and edited by Richard A. Silverman Prentice-Hall, Inc., Englewood Cliffs, N.J. (1965) xii+308 pp.

\bibitem{LeM1}
C. Le Merdy.
\newblock $\HI$-functional calculus and applications to maximal regularity.
\newblock Publ. Math. UFR Sci. Tech. Besan\c{c}on 16 (1998), 41--77.

\bibitem{LeM2}
C. Le Merdy.
\newblock The Weiss conjecture for bounded analytic semigroups.
\newblock J. London Math. Soc. (2) 67 (2003), no. 3, 715--738.

\bibitem{Med}
S. Meda.
\newblock A general multiplier theorem.
\newblock Proc. Amer. Math. Soc. 110 (1990), no. 3, 639--647.

\bibitem{vNP}
J. van Neerven and P. Portal.
\newblock The Weyl calculus for group generators satisfying the canonical commutation relations.
\newblock Preprint on https://arxiv.org/abs/1806.00980

\bibitem{Ouh}
E. M. Ouhabaz.
\newblock Analysis of heat equations on domains.
\newblock London Mathematical Society Monographs Series, 31. Princeton University Press, Princeton, NJ, (2005). xiv+284 pp.

\bibitem{Pazy}
A. Pazy.
\newblock Semigroups of linear operators and applications to partial differential equations.
\newblock Applied Mathematical Sciences, 44. Springer-Verlag, New York, (1983). viii+279 pp.


\bibitem{PrSi}
J. Pr\"uss and G. Simonett.
\newblock $\HI$-calculus for the sum of non-commuting operators.
\newblock Trans. Amer. Math. Soc. 359 (2007), no. 8, 3549--3565.

\bibitem{Ste56}
E. M. Stein.
\newblock Interpolation of linear operators.
\newblock Trans. Amer. Math. Soc. 83 (1956), 482--492.

\bibitem{Ste70}
E. M. Stein.
\newblock Topics in harmonic analysis related to the Littlewood-Paley theory.
\newblock Annals of Mathematics Studies, No. 63 Princeton University Press, Princeton, N.J.; University of Tokyo Press, Tokyo (1970) viii+146 pp.

\bibitem{Weis}
L. Weis.
\newblock A new approach to maximal $L_p$-regularity.
\newblock volution equations and their applications in physical and life sciences (Bad Herrenalb, 1998), 195--214,
Lecture Notes in Pure and Appl. Math., 215, Dekker, New York, (2001).

\bibitem{Xu2015}
Q. Xu.
\newblock $\HI$ functional calculus and maximal inequalities for semigroups of contractions on vector-valued $L_p$-spaces.
\newblock Int. Math. Res. Not. IMRN 2015, no. 14, 5715--5732.

\bibitem{Yag}
A. Yagi.
\newblock Co\"incidence entre des espaces d'interpolation et des domaines de puissances fractionnaires d'op\'erateurs.
\newblock C. R. Acad. Sci. Paris S\'er. I Math. 299 (1984), no. 6, 173--176.

\end{thebibliography}
\end{document}